\documentclass{amsart}

\usepackage[margin=1.2in]{geometry}
\usepackage[english]{babel}

\usepackage{mathtools}
\usepackage{amssymb}
\usepackage{graphicx}
\usepackage{xcolor}
\usepackage{comment}
\usepackage{booktabs}
\usepackage{tabularx}
\usepackage{enumitem}
\usepackage{tikz-cd}
\usepackage{bbm}
\usepackage{shuffle}
\usepackage{mathrsfs}
\usepackage{atbegshi}
\usepackage[autostyle]{csquotes}
\usepackage[toc,page]{appendix}
\usepackage{todonotes}

\usepackage[
    colorlinks,
    citecolor=red,
    urlcolor=blue,
    bookmarks=false,
    hypertexnames=true
]{hyperref}
\usepackage[capitalize]{cleveref}

\newcommand{\vertiii}[1]{%
    {\left\vert\kern-0.25ex
    \left\vert\kern-0.25ex
    \left\vert #1
    \right\vert\kern-0.25ex
    \right\vert\kern-0.25ex
    \right\vert}%
}

\theoremstyle{definition}
\newtheorem{example}{Example}

\theoremstyle{plain}
\newtheorem{theorem}{Theorem}[section]
\newtheorem{remark}[theorem]{Remark}
\newtheorem{lemma}[theorem]{Lemma}
\newtheorem{proposition}[theorem]{Proposition}
\newtheorem{Assumption}[theorem]{Assumption}
\newtheorem{corollary}[theorem]{Corollary}
\newtheorem{definition}[theorem]{Definition}

\newtheoremstyle{break}
    {\topsep}
    {\topsep}
    {\itshape}
    {}
    {\bfseries}
    {.}
    {\newline}
    {}

\newcommand{\R}{\mathbb{R}}

\newcommand{\N}{\mathbb{N}}

\newcommand{\dd}{\mathrm{d}}
\newcommand{\bmu}{\boldsymbol{\mu}}
\newcommand{\bgamma}{\boldsymbol{\gamma}}

\newcommand{\Sig}[1]{\mathrm{Sig}(#1)}
\newcommand{\Sigsym}[1]{\widehat{\mathrm{Sig}}(#1)}
\newcommand{\RP}[1]{\mathbf{#1}}

\newcommand{\gammashuffle}{{\,\bullet\,}}
\newcommand{\gammashufflesym}{{\,\widehat{\gammashuffle}\,}}
\newcommand{\symmetrizer}{{\,\widehat{\cdot}\,}}

\long\def\blue#1{{\color{blue}#1}}
\long\def\red#1{{\color{orange}#1}}
\long\def\orange#1{{\color{orange}#1}}
\long\def\wh#1{\widehat{#1}}

\title{Expected signatures via partial integration, coordinate change and symmetrization}\thanks{\emph{Acknowledgments:} Both authors would like to thank Peter Bank, Christian Bayer and Peter Friz for valuable comments and helpful  discussions.}

\author{Paul P. Hager}
\address{Paul P. Hager, Department of Statistics and Operations Research, University of Vienna,
Kolingasse 14--16, 1090 Vienna, Austria}
\author{Luca Pelizzari}
\address{Luca Pelizzari, Department of Statistics and Operations Research, University of Vienna,
Kolingasse 14--16, 1090 Vienna, Austria}
\date{}


\begin{document}

\begin{abstract}
We study signature transformations of heterogeneous paths $Y=(A,X)$ whose
components may differ in regularity and probabilistic structure. We introduce an
invertible change of coordinates $\Psi$ such that the transformed signature
$\Psi\circ\mathrm{Sig}$ eliminates mixed integration against the irregular
component $X$ and admits a representation in terms of signature coordinates
of $X$ and iterated integration against the regular component $A$. In addition, we exploit this representation to further represent partially symmetrized signatures. Our main
application concerns new expected signature formulas of processes with deterministic
augmentation. On the analytical side, these formulas are leveraged to study moment problems. On
the numerical side, they enable accurate computation of expected signatures, thereby overcoming typical
computational bottlenecks in applications. We illustrate these
advantages in a signature-based stochastic control problem driven by
fractional Brownian motion.
\end{abstract}

\maketitle


\section{Introduction}
Originating in Chen's works
\cite{chen1954iterated,chen1957integration} and later becoming central to
Lyons' development of rough path analysis
\cite{lyons1998differential}, iterated integrals
\[
    Y
    \quad\longmapsto\quad
    \langle i_1,\ldots,i_n,\Sig{Y}_{0,T}\rangle
    :=
    \int_{0<t_1<\cdots<t_n<T}
    \dd Y_{t_1}^{i_1}\cdots \dd Y_{t_n}^{i_n}
\]
are commonly viewed as \emph{monomials} on suitable $d$-dimensional (rough) path spaces:
their linear span is closed under pointwise multiplication, and they separate
points up to tree-like equivalence
\cite{HamLy10,boedihardjo2016signature}.
The collection of these iterated integrals forms an element in the extended tensor algebra
\[
    \Sig{Y}_{0,T}
    :=
    \sum_{w\in\mathcal{W}_{\mathcal{A}}}
    \langle w,\Sig{Y}_{0,T}\rangle e_w
    \in
    T((\mathbb{R}^d))
    :=
    \prod_{n\geq 0}
    (\mathbb{R}^d)^{\otimes n},
\]
called the \emph{signature} of $Y$.
Here
$\mathcal{A}:=\{1,\ldots,d\}$,
$\mathcal{W}_{\mathcal{A}}$ denotes the set of words over $\mathcal{A}$,
and $e_w$ denotes the tensor basis element associated with
$w\in\mathcal{W}_{\mathcal{A}}$, with $e_{\emptyset}:=1$.
We introduce this notation and the underlying tensor-algebraic structure
in more detail below.

The monomial viewpoint extends to the probabilistic setting, where the \emph{expected signature}
\[
    \mathbb{P}
    \quad\longmapsto\quad
    \bmu^{\mathbb{P}}
    :=
    \mathbb{E}_{Y\sim\mathbb{P}}
    \left[
        \Sig{Y}_{0,T}
    \right]
    =
    \sum_{w\in\mathcal{W}_{\mathcal{A}}}
    \mathbb{E}_{Y\sim\mathbb{P}}
    \left[
        \langle w,\Sig{Y}_{0,T}\rangle
    \right]e_w,
\]
whenever well defined, collects the \emph{moments} of the probability distribution
$\mathbb{P}$ on the underlying path space.
In analogy with the classical moment problem, but under a stronger than classic
growth conditions, it was shown in
\cite{chevyrev2016characteristic} that $\bmu^{\mathbb{P}}$ can characterize
$\mathbb{P}$. More precisely, provided $\bmu^{\mathbb{P}}$ has infinite
radius of convergence, that is,
\[
    \sum_{n\geq 0}
    \lambda^n
    \sum_{w\in\mathcal{W}_{\mathcal{A}}^n}
    \left|
        \langle w,\bmu^{\mathbb{P}}\rangle
    \right|
    <
    \infty
    \qquad
    \text{for every } \lambda>0,
\]
where $\mathcal{W}_{\mathcal{A}}^n$ denotes the set of words of length
$n$, the collection of these moments characterizes the distribution
$\Sig{\cdot}_{\#}\mathbb{P}$ and hence $\mathbb{P}$ modulo tree-like
equivalence.

Expected signatures thus feature in a broad range of applications, including
cubature and weak approximation on Wiener space
\cite{lyons2004cubature,gyurko2011efficient,litterer2012recombination,crisan2012solving,bayer2013cubature,shinozaki2017construction,hayakawa2022monte,ferrucci2026high}, the statistical identification, comparison, and dependence analysis of
stochastic-process laws
\cite{bonnier2020signature,salvi2021higherorder,
chevyrev2018signature,bonnier2023adapted,schell2023nonlinear,
andres2024validation,cass2024weighted,bonnier2025proper,
friz2025expectedkernel}, statistical learning, prediction, conditioning,
and generative modelling for sequential data
\cite{levin2013learning,lemercier2021distribution,
liao2024sigwasserstein,
chevyrev2026orthogonal}, the recovery of geometric information on manifolds
\cite{geng2024expected}, expected-signature-based parameter estimation and
signature-based stochastic modelling, calibration, and valuation
\cite{papavasiliou2011parameter,lyons2020numerical,
cuchiero2023signaturemodels,cuchiero2023signature,
abiJaber2025signaturevolatility}, and path-dependent optimization, trading,
and stochastic control
\cite{kalsi2020optimal,cartea2022double,bayer2021optimal,
futter2023signaturetrading,bayer2023primal,
abiJaber2025frictionalhedging,abiJaber2026control}.

This renders the computation of expected signatures a central, but generally
difficult, problem. Beginning with Fawcett's explicit formula for Brownian
motion \cite{fawcett2003problems}, PDE representations were developed for
diffusion processes \cite{ni2012expected}, alongside results for Brownian
motion stopped upon exiting a domain
\cite{lyons2015expected,boedihardjo2021finite,li2022expected,
boedihardjo2026representation}. Further foundational approaches include a signature 
L{\'e}vy--Khintchine-type formula 
\cite{friz2017general}, functional and recursive equations for expected
signatures and signature cumulants of general semimartingales
\cite{friz2022unified,friz2026expected}, linear moment equations based on
polynomial-process methods \cite{cuchiero2023signature}, and formulas for
fractional Brownian motion and more general Gaussian processes based on
covariance and Wiener-chaos methods
\cite{baudoin2007operators,boedihardjo2013expected,cass2024wiener}.
Convergence of numerical approximations based on time discretization and empirical estimation has been studied in
\cite{ni2015concentration,passeggeri2020signature,
lucchese2025learning}.

In multivariate stochastic models, the constituent processes are often
distinguished by their regularity, distributional class, dependence structure,
and purpose. Treating all components uniformly obscures
part of the structure and typically reduces generality in assumptions and numerical tractability. 
In rough path theory, several developments accommodate such heterogeneity.
The $(p,q)$-rough-path construction pairs two components of complementary
variation regularity through Young integration
\cite{lejay2006pq,gassiat2024gaussian}; see also
\cite[Section~9.4]{friz2010multidimensional}. The case of multiple
regularities, already formulated in \cite{lyons1998differential}, was
subsequently developed as $\Pi$-rough paths
\cite{gyurko2016differential} and, in a closely related H{\"o}lder
formulation, as anisotropic geometric rough paths
\cite{tapia2020geometry}.
Drivers of different probabilistic types have also been combined through
joint lifts of Brownian motion with a deterministic and, later, a possibly
correlated stochastic rough path
\cite{crisan2013robust,diehl2015levy,bonesini2024rough}.
Specifically motivated by rough volatility, the works
\cite{bayer2020regularity,fukasawa2024partial} account for the distinct roles
and regularities of log-price and volatility using, respectively, regularity
structures and partial rough paths.
Beyond a purely pathwise treatment, the hybrid framework of
\cite{friz2021rough} systematically distinguishes between semimartingale and
generic rough components, providing a general solution theory for rough
stochastic differential equations.

In signature modelling, however, such heterogeneity has largely remained
unaccounted for, even though it is routinely built into the underlying models
as a deliberate choice, most prominently through augmentation.\footnote{A notable precedent is the stochastic Taylor and cubature
literature, where inhomogeneous truncations assign different weights to time
and Brownian letters according to their distinct scalings
\cite{kloeden1992numerical,lyons2004cubature}. A recent exception is the
computational note \cite{nygaard2026pathsig}, which introduces anisotropic
signature truncations and projections onto prescribed sets of words.}
In this work, we propose the following approach. Suppose that the components
of the underlying process  separate into two groups $Y=(A,X)$,
reflected by a partition of the alphabet
\[
    \mathcal{A}
    =
    \mathcal{A}^{\prime}\mathbin{\dot\cup}\mathcal{A}^{\prime\prime}
    =
    \{\blue{a_1},\ldots,\blue{a_m}\}
    \mathbin{\dot\cup}
    \{\orange{b_1},\ldots,\orange{b_n}\}.
\]
Assuming for the sake of exposition that $Y$ is smooth, but viewing $A$ and
$X$ as placeholders where $X$ is the more irregular component, one may aim
to eliminate all integrations against $\dd X$ from the signature by
repeated partial integration.
Our first representation result facilitates this in a systematic way:
\begin{theorem}\label{thm:1_intro}
    There exists a graded invertible linear map $\Psi:
        T((\mathbb{R}^{d}))
        \to
        T((\mathbb{R}^{d}))$ such that for every smooth path
    $Y=(A,X)$ and any word $w\in\mathcal{W}_{\mathcal{A}}$, it holds that
    \[
        \left\langle
            w,
            \Psi\circ\Sig{Y}_{s,t}
        \right\rangle
        =
        \left\langle
            w_k,
            \Sig{X}_{s,t}
        \right\rangle
        \int_{s<t_1<\cdots<t_k<t}
        \prod_{j=1}^{k}
        \left\langle
            w_{j-1},
            \Sig{X}_{s,t_j}
        \right\rangle
        \dd A^{\blue{i_j}}_{t_j},
    \]
    where we uniquely decompose
    \begin{equation}\label{eq:decomp_intro}
        w
        =
        w_0\blue{i_1}w_1\cdots\blue{i_k}w_k,
        \qquad
        w_0,\ldots,w_k\in\mathcal{W}_{\mathcal{A}^{\prime\prime}},
        \quad
        \blue{i_1},\ldots,\blue{i_k}\in\mathcal{A}^{\prime}.
    \end{equation}
\end{theorem}
This representation has several immediate consequences in relation to the
aforementioned works on joint rough path lifts.
Indeed, if $X$ is enhanced to a geometric $p$-rough path and $A$ has complementary
$q$-variation, the right-hand side remain well defined through
Young integration and yields the canonical joint lift of $(A,X)$
\cite{lejay2006pq,friz2010multidimensional}. If $A$ is instead a
semimartingale, the integrals against $A$ may be interpreted stochastically,
recovering the corresponding joint stochastic lifts
\cite{diehl2015levy,bonesini2024rough} and fitting into the hybrid framework
of \cite{friz2021rough}. 

Moreover, the representation above directly exhibits the form obtained by
symmetrizing over the partial alphabet $\mathcal{A}^{\prime\prime}$ and thereby makes
transparent its relation to the partial rough path lift of
\cite{fukasawa2024partial}, in which $X$ is lifted only through multivariate
monomials.
With details postponed to a later section, denote by
$T_{\mathcal{A}^{\prime\prime}}((\mathbb{R}^{d}))$ the quotient algebra
obtained by imposing commutativity among the letters of
$\mathcal{A}^{\prime\prime}$, and by $\Sigsym{Y}$
the corresponding partially symmetrized signature. We then obtain the
following representation.
\begin{corollary}\label{cor:1_intro}
    There exists a graded invertible linear map
    $\widehat{\Psi}:
        T_{\mathcal{A}^{\prime\prime}}((\mathbb{R}^{d}))
        \to
        T_{\mathcal{A}^{\prime\prime}}((\mathbb{R}^{d}))$
    such that, for every smooth path $Y=(A,X)$ and any word
    $w\in\mathcal{W}_{\mathcal{A}}$ with decomposition as in \eqref{eq:decomp_intro}, it holds that
    \[
        \left\langle
            [w_0]\blue{i_1}[w_1]\cdots\blue{i_k}[w_k],
            \widehat{\Psi}\circ\Sigsym{Y}_{s,t}
        \right\rangle
        =
        \frac{X_{s,t}^{[w_k]}}{[w_k]!}
        \int_{s<t_1<\cdots<t_k<t}
        \prod_{j=1}^{k}
        \frac{X_{s,t_j}^{[w_{j-1}]}}{[w_{j-1}]!}
        \dd A^{\blue{i_j}}_{t_j},
    \]
    where, for $v\in\mathcal{W}_{\mathcal{A}^{\prime\prime}}$,
    \[
        X_{s,t}^{[v]}
        :=
        \prod_{\orange{i}\in\mathcal{A}^{\prime\prime}}
        \left(X_{s,t}^{\orange{i}}\right)^{\alpha_{\orange{i}}},
        \qquad
        [v]!
        :=
        \prod_{\orange{i}\in\mathcal{A}^{\prime\prime}}
        \alpha_{\orange{i}}!,
    \]
    with $\alpha_{\orange{i}}$ denoting the multiplicity of
    $\orange{i}$ in $v$.
\end{corollary}

Returning to our original motivation of computing expected signatures, the
representation in \Cref{thm:1_intro} suggests conditioning on $A$ as a
natural computational step. More concretely, suppose that $\mathbf{X}$ is a
random geometric $p$-rough path and that $A$ is deterministic with
complementary $q$-variation. The signature of $Y=(A,\mathbf{X})$ may then be
defined directly through the representation in \Cref{thm:1_intro}. For the
transformed expected signature $\Psi\circ\bmu^{\mathbb{P}}$, the expectation
can be taken directly inside the iterated integrals with respect to $A$.

\begin{corollary}\label{cor:expected_signature_intro}
    Let $\mathbb{P}$ be the law of $Y=(A,\mathbf{X})$, where $\mathbf{X}$ is
    a random geometric $p$-rough path and $A$ is deterministic with
    complementary $q$-variation. Suppose that the expected signature
    $\bmu^{\mathbb{P}}$ and the integrals below are well defined.
    Then, for every word
    $w=w_0\blue{i_1}w_1\cdots\blue{i_k}w_k$ with decomposition as in
    \eqref{eq:decomp_intro}, it holds that
    \begin{equation}\label{eq:exp_signature_Psi_intro}
        \left\langle
            w,
            \Psi\circ\bmu^{\mathbb{P}}
        \right\rangle
        =
        \int_{0<t_1<\cdots<t_k<T}
        \mathbb{E}_{Y\sim\mathbb{P}}
        \left[
            \prod_{j=1}^{k+1}
            \left\langle
                w_{j-1},
                \Sig{\mathbf{X}}_{0,t_j}
            \right\rangle
        \right]
        \dd A^{\blue{i_1}}_{t_1}
        \cdots
        \dd A^{\blue{i_k}}_{t_k},
        \qquad
        t_{k+1}:=T.
    \end{equation}
\end{corollary}

After partial symmetrization, the situation becomes even more concrete.
Suppose that $A$ is a deterministic path of bounded variation. Recall that,
by the representation in \Cref{cor:1_intro}, we may define $\Sigsym{Y}$ for
$Y=(A,X)$ whenever $X$ is a continuous stochastic process. Writing
$$
    \widehat{\bmu}^{\mathbb{P}}
    =
    \mathbb{E}_{Y\sim\mathbb{P}}\big[\Sigsym{Y}\big]
    \in T_{\mathcal{A}^{\prime\prime}}((\mathbb{R}^d)),
$$
we obtain a concrete expression for its $\widehat{\Psi}$-transform in terms
of multivariate increment correlators
\begin{equation}\label{eq:intro_correlators}
    \wh{\mathcal{C}}^{[w]}_{0,T}(t_1,\dots,t_k)
    =
    \mathbb{E}_{Y\sim\mathbb{P}}
    \left[
        \prod_{j=1}^{k+1}
        \frac{1}{[w_{j-1}]!}
        X_{0,t_j}^{[w_{j-1}]}
    \right],
    \qquad
    t_{k+1}=T,
    \quad
    w_0,\dots,w_k\in\mathcal{W}_{\mathcal{A}^{\prime\prime}}.
\end{equation}

\begin{corollary}\label{cor:expected_symmetrized_signature_intro}
    Let $\mathbb{P}$ be the law of $Y=(A,X)$, where $X$ is a continuous
    stochastic process and $A$ is deterministic with bounded variation.
    Suppose that the correlators in \eqref{eq:intro_correlators} are finite
    and continuous. Then, for every word
    $w=w_0\blue{i_1}w_1\cdots\blue{i_k}w_k$ with decomposition as in
    \eqref{eq:decomp_intro}, it holds that
    \begin{equation}\label{eq:exp_sig_symm_intro}
        \left\langle
            [w_0]\blue{i_1}[w_1]\cdots\blue{i_k}[w_k],
            \widehat{\Psi}\circ\widehat{\bmu}^{\mathbb{P}}
        \right\rangle
        =
        \int_{0<t_1<\cdots<t_k<T}
        \wh{\mathcal{C}}^{[w]}_{0,T}(t_1,\dots,t_k)
        \dd A^{\blue{i_1}}_{t_1}
        \cdots
        \dd A^{\blue{i_k}}_{t_k}.
    \end{equation}
\end{corollary}
The correlators in \eqref{eq:intro_correlators} are often available in
closed or semi-explicit form, for instance for Gaussian and polynomial
processes.
This gain in tractability after partial symmetrization does not have come at
the cost of losing the characterizing properties of signatures. If $A$
includes a sufficiently non-degenerate deterministic augmentation component, most notably a
time component, then partial symmetrization retains both the pathwise
injectivity of the signature and the law-determining property of the
expected signature. At the level of the underlying process $X$, the latter
requires only the particularly tractable assumption that its
one-dimensional marginals are moment-determinate.
\begin{theorem}\label{thm:characteristic_intro}
    Let $A$ be a deterministic path satisfying
    Assumption~\ref{ass:mambo}; in particular, this holds whenever $A$
    contains a time component. Fix
    $x_0\in\mathbb{R}^{|\mathcal{A}^{\prime\prime}|}$, and let
    $\mathbb{P}$ and $\mathbb{Q}$ be probability measures on
    $C([0,T];\mathbb{R}^d)$ such that
    $(\pi_{\mathcal{A}^{\prime}})_{\#}\mathbb{P}
    =
    (\pi_{\mathcal{A}^{\prime}})_{\#}\mathbb{Q}
    =
    \delta_A$, where $\pi_{\mathcal{A}^{\prime}}$ denotes the restriction
    to the $\mathcal{A}^{\prime}$-components, and which are concentrated on
    paths whose $\mathcal{A}^{\prime\prime}$-components start at $x_0$.
    Suppose that the partially symmetrized expected signatures
    $\widehat{\bmu}^{\mathbb{P}}$ and $\widehat{\bmu}^{\mathbb{Q}}$ are well
    defined and that, for every $t\in[0,T]$ and
    $\orange{i}\in\mathcal{A}^{\prime\prime}$, the one-dimensional
    marginals of $X_t^{\orange{i}}$ under $\mathbb{P}$ and $\mathbb{Q}$ are
    moment-determinate. Then
    \[
        \widehat{\bmu}^{\mathbb{P}}
        =
        \widehat{\bmu}^{\mathbb{Q}}
        \qquad\Longleftrightarrow\qquad
        \mathbb{P}
        =
        \mathbb{Q}.
    \]

    In particular, the linear functionals
    $X\mapsto\langle\ell,\Sigsym{A,X}_{0,T}\rangle$ form a universal class
    for the approximation of continuous functionals on compact subsets of
    $C([0,T];\mathbb{R}^{|\mathcal{A}^\prime|})$ restricted to paths starting at
    $x_0$, and the map $X\mapsto\Sigsym{A,X}_{0,T}$ is injective on this
    space.
\end{theorem}
In \Cref{sec:brownain_characteristicness}, we prove an analogue of this theorem in
which $A$ is no longer a deterministic augmentation path but an independent
standard Brownian motion. The case of a general, possibly correlated
semimartingale $A$ raises several interesting questions and opens
up further directions for future research.

Beyond their analytical consequences, the signature representations above also
offer numerical advantages. To illustrate these, consider the two-dimensional
augmented process $Y=(A,X)$, where $A_t=t$ is a time component and $X$ is an
irregular stochastic signal. Since $X$ is one-dimensional, the
expected-signature representations
\eqref{eq:exp_signature_Psi_intro}-\eqref{eq:exp_sig_symm_intro} reduce to the
same formula, which separate the probabilistic task of evaluating
correlators of $X$ from the deterministic simplex integration in
time. This computational factorization provides two major advantages:

First, it exploits the regularizing effect of expectation. Indeed, the
involved correlators are typically considerably more regular than the sample
paths themselves,\footnote{This is illustrated for Gaussian and polynomial
processes in Examples~\ref{ex:Gaussian_isserlis}--\ref{ex:polynomial}.}
and may therefore be integrated against $\dd t$ using higher-order
deterministic quadrature. In contrast, the conventional estimation of the
expected signature, based on piecewise-linear approximations of $Y$ and
averaging the resulting signatures, may suffer from a substantial
time-discretization bias when $X$ is irregular. We illustrate both numerically and theoretically in 
Appendix~\ref{app:weak_error}, that the bias if such an estimator is of order $|\pi|^{2H}$, where
$H$ denotes the Hurst parameter and $|\pi|$ the mesh size. The resulting slow
convergence becomes particularly restrictive in rough regimes
$H\ll 1/2$. Our transformed representation overcomes this shortcoming and, as shown in
Figure~\ref{fig:fbm_runtime_error}, yields highly accurate approximations at
substantially lower computational cost.

Second, the transformations naturally induce a mixed grading. Indeed, the
representation in Theorem~\ref{thm:1_intro} suggests truncating separately in
the numbers of $\mathcal{A}'$- and $\mathcal{A}''$-letters, rather than
retaining all words up to a common tensor level. These degrees determine,
respectively, the dimension of the deterministic simplex integral and the
order of the required correlator. Importantly, the coordinate transformations, their inverses, and the associated
dual products are all determined by explicit recursions and can therefore be
precomputed directly under any fixed mixed-degree truncation.\footnote{The implementation of mixed-degree truncations,
the coordinate transformations, and the associated dual products in efficient
contiguous data layouts is described in the companion working note
\cite{hager2026bidegree}. The corresponding functionality is being developed
for inclusion in the \texttt{tensordev} library.}
The resulting flexibility to retain high degrees in one alphabet while
keeping the other shallow yields relevant computational savings, as we
illustrate below.

In Section~\ref{sec:control}, we demonstrate both advantages in a
signature-based stochastic control application. In the spirit of
\cite{kalsi2020optimal}, we parametrize admissible controls by signature
functionals in the $\Psi$-coordinates
\[
    \alpha_t
    =
    \big\langle
        \ell,\Psi\circ\Sig{Y}_{0,t}
    \big\rangle,
\]
where $\ell$ is a linear combination of words truncated at mixed degree
$(N_A,N_X)$. By Theorem~\ref{thm:characteristic_intro}, these controls form a
rich approximation class. Moreover, since $\Psi$ is linear and invertible, the
signature-control approach of \cite{kalsi2020optimal} still applies, that is, linear-quadratic
stochastic control problems reduce to  deterministic convex\footnote{Convexity was established in  \cite[Lemma 3.10]{aqsha2026solving} for linear-quadratic problems.} optimization
problem
\begin{equation}\label{eq:convex_sig_opt}
    \inf_{\ell\in T^{(N_A,N_X)}(\mathbb{R}^d)}
    \big\langle F(\ell),\Psi\circ\bmu^{\mathbb{P}}\big\rangle,
    \qquad
    F:T^{(N_A,N_X)}(\mathbb{R}^d)
    \longrightarrow
    T^{(2N_A+3,2N_X)}(\mathbb{R}^d),
\end{equation}
where $F(\ell)$ is explicitly computable, see~\eqref{eq:shear_tracking_problem}. In standard signature coordinates, optimization problems of the type   \eqref{eq:convex_sig_opt} typically suffer from
two numerical bottlenecks: First, estimation errors in $\bmu^{\mathbb{P}}$ can lead to
ill-conditioned optimization problems, which was also recently
observed for signature-based optimal execution in \cite[Section 2.6]{morbelli2026signature}. Second,
since controls truncated at level $N$ generate cost terms up to level
$2N+3$, computing the required expected-signature coordinates quickly becomes
expensive even for moderate $N$. The transformed formulation addresses both issues: whenever the correlators
of $X$ are available, $\Psi\circ\bmu^{\mathbb{P}}$ can be evaluated with high
accuracy, while the mixed grading allows the truncation depths in the two
alphabets to be chosen independently. In our experiments with fractional Brownian motion, we observe that already small values of $N_X$, combined with
comparatively large values of $N_A$, yield highly accurate approximations of
the optimal value. This reduces the required tensor dimension by several orders of
magnitude, together with substantial savings in memory and runtime. We summarize the results over a range of Hurst parameters in
Figure~\ref{fig:tracking_heatmap_costs} and
Table~\ref{tab:tracking_costs}.

\section{Preliminaries}\label{sec:prelim}
Throughout the article we will frequently make use of standard notations used in the rough path and signature literature, which we shall briefly recall here. Our main reference here \cite{friz2010multidimensional}, to which we refer for all the details.

\subsection{Tensor and shuffle algebra}\label{sec:tensor_shuffle_algebra}
The \emph{extended tensor-algebra} is defined as
\[
T((\mathbb{R}^d))
:=
\prod_{n\geq 0}(\mathbb{R}^d)^{\otimes n},
\]
where $(\mathbb{R}^d)^{\otimes 0}:=\mathbb{R}$. We identify the basis element
$e_{i_1}\otimes\cdots\otimes e_{i_n}$ with the word
$w=\blue{i_1}\cdots\blue{i_n}$ over the alphabet
$\mathcal{A}=\{\blue{1},\dots,\blue{d}\}$, and write
$e_w:=e_{i_1}\otimes\cdots\otimes e_{i_n}$, with $e_\emptyset:=1$.
We denote the set of all words on $\mathcal{A}$ by
$\mathcal{W}_{\mathcal{A}}$, and simply by $\mathcal{W}$ when the alphabet is
clear from the context. For $w=\blue{i_1}\cdots\blue{i_n}$, let
$|w|:=n$, and write
$\mathcal{W}^{(n)}:=\{w\in\mathcal{W}:|w|=n\}$.
The span of $\mathcal{W}_{\mathcal{A}}$ coincides with the free associative
algebra $\mathbb{R}\langle\mathcal{A}\rangle$, equipped with concatenation of
words. It is naturally paired with the extended tensor algebra by the bilinear map
\begin{equation*}\label{def:span_words}
    \langle \cdot,\cdot\rangle:
    \mathbb{R}\langle\mathcal{A}\rangle
    \times
    T((\mathbb{R}^d))
    \longrightarrow
    \mathbb{R},
    \qquad
    \left\langle
        w,e_v
    \right\rangle
    = \mathbbm{1}_{\{w=v\}}.
\end{equation*}
The tensor algebra is equipped with the tensor concatenation product
$\otimes$, defined on basis elements by
$e_u\otimes e_v:=e_{uv}$ for $u,v\in\mathcal{W}_{\mathcal{A}}$, where $uv$ denotes the concatenation of words.
On $\mathbb{R}\langle \mathcal{A} \rangle$ we define the 
\emph{shuffle product} of
words recursively by
\begin{equation}\label{def:shuffle-product}
w \shuffle \emptyset
=
\emptyset \shuffle w
=
w,
\qquad
w\blue{i}\shuffle v\blue{j}
:=
(w\shuffle v\blue{j})\blue{i}
+
(w\blue{i}\shuffle v)\blue{j},
\end{equation}
for $w,v\in\mathcal{W}$ and
$\blue{i},\blue{j}\in\mathcal{A}$.
This operation characterizes an important subset of the extended tensor algebra
which contains the range of the signature map:
\begin{equation*}
G((\R^d))
=
\left\{
\mathbf{a}\in T((\R^d))\setminus\{\mathbf{0}\}
:
\langle w,\mathbf{a}\rangle
\langle v,\mathbf{a}\rangle
=
\langle w\shuffle v,\mathbf{a}\rangle,
\quad
\forall w,v\in\mathcal{W}
\right\}.
\end{equation*}
We denote by $T_1((\mathbb{R}^d))$ the elements
$\mathbf{a}\in T((\mathbb{R}^d))$ such that
$\langle\emptyset,\mathbf{a}\rangle=1$. Clearly,
\(
G((\mathbb{R}^d))
\subseteq
T_1((\mathbb{R}^d)).
\)

\subsection{Graded morphisms and induced pairings}\label{sec:graded_morphisms}

For $n\geq 0$, let $\pi^n$ denote the projection onto the $n$-th tensor
level, and write $\pi^{\leq N}:=\sum_{n=0}^{N}\pi^n$. For any sequence
$(\ell_w)_{w\in\mathcal{W}}\subset\mathbb{R}\langle\mathcal{A}\rangle$, we
associate the linear map
\begin{equation}\label{eq:linear_operator}
    \ell:
    T((\mathbb{R}^d))
    \longrightarrow
    T((\mathbb{R}^d)),
    \qquad
    \ell(\mathbf{a})
    :=
    \sum_{k=0}^{\infty}
    \sum_{w\in\mathcal{W}^{(k)}}
    e_w
    \left\langle
        \ell_w,\mathbf{a}
    \right\rangle.
\end{equation}
Equivalently, one may view
$\big(\langle\ell_w,e_v\rangle\big)_{w,v\in\mathcal{W}}$ as an infinite
matrix. Since each $\ell_w$ is a finite linear combination of words, this
matrix is row-finite, or \emph{finitary} in the sense that each row contains
only finitely many nonzero entries. We call
$(\ell_w)_{w\in\mathcal{W}}$ the row-vector representation of $\ell$.

The row-vector representation is called \emph{graded} if
$\ell_w\in\operatorname{span}(\mathcal{W}^{(|w|)})$ for every
$w\in\mathcal{W}$. In that case, the associated linear map preserves tensor
levels, that is $\ell\circ\pi^n=\pi^n\circ\ell$ for every $n\geq 0$, and
therefore also $\ell\circ\pi^{\leq N}=\pi^{\leq N}\circ\ell$ for every
$N\geq 0$.

For a graded linear map $\ell$ with row-vector representation
$(\ell_w)_{w\in\mathcal{W}}$, its \emph{graded transpose}
$$\ell^\ast:\mathbb{R}\langle\mathcal{A}\rangle\to
\mathbb{R}\langle\mathcal{A}\rangle$$ is defined by
\begin{equation}\label{eq:graded_transpose}
    \left\langle
        \ell^\ast(\alpha),
        \mathbf{a}
    \right\rangle
    :=
    \left\langle
        \alpha,
        \ell(\mathbf{a})
    \right\rangle,
    \qquad
    \alpha\in\mathbb{R}\langle\mathcal{A}\rangle,
    \quad
    \mathbf{a}\in T((\mathbb{R}^d)).
\end{equation}
In particular, $\ell^\ast(w)=\ell_w$ for every
$w\in\mathcal{W}_{\mathcal{A}}$. If $\ell$ is a graded automorphism, then
$(\ell^\ast)^{-1}$ is the graded transpose of $\ell^{-1}$.

Finally, we introduce a pairing induced by a graded linear map, which will
play an important role throughout this article.

\begin{definition}\label{def:induced_pairing}
    Let $\ell: T((\mathbb{R}^d))\rightarrow T((\mathbb{R}^d))$ be a graded linear map. We define the pairing \begin{equation}
    \label{eq:coordinates_pairing}
    \langle \alpha,\mathbf{a}\rangle_{\ell}:= \langle \alpha, \ell(\mathbf{a})\rangle = \langle \ell^\ast(\alpha),\mathbf{a}\rangle, \qquad \alpha \in \mathbb{R}\langle \mathcal{A}\rangle, \quad \mathbf{a}\in T((\mathbb{R}^d)). 
\end{equation} If $\ell$ is a graded automorphism, we call the family $(\langle w,\mathbf{a}\rangle_{\ell})_{w\in \mathcal{W}}$ the $\ell$-coordinates of $\mathbf{a}$.
\end{definition}

\subsection{Spliting the alphabet}
\newcommand{\decomp}[1]{\Delta_{#1}}
Let us now introduce a \emph{decomposition operator} on the space $\mathcal{W}_\mathcal{A}$ with respect to a given alphabet $\mathcal{A}$ and some sub-alphabet $\emptyset\neq \mathcal{A}' \subsetneq \mathcal{A}$, which shall become important for the main results of this work.
To this end, one first verifies that any $w \in \mathcal{W}_\mathcal{A}$ is uniquely represented as $$w=w_0 \blue{j_1}w_1\blue{j_2}w_2\cdots \blue{j_k}w_k,$$
for some $k\in\N_0$, $\blue{j_1\cdots j_k}\in \mathcal{W}^{(k)}_{\mathcal{A}^\prime}$ and $w_0,\dots,w_k \in \mathcal{W}_{\mathcal{A}\setminus{\mathcal{A}'}}$.
With this representation we define
\begin{equation}\label{def:Delta_operator}
\decomp{\mathcal{A}^\prime} :  \mathcal{W}_{\mathcal{A}} \to \bigcup_{k\in\N}(\mathcal{W}_{\mathcal{A} \setminus\mathcal{A}^\prime})^k \times \mathcal{W}^{(k)}_{\mathcal{A}^\prime}, \qquad w  \mapsto \left( (w_0,\,\dots,\,w_k),\,\blue{j_1}\cdots \blue{j_k}\right).
\end{equation}

In words, the operator $\decomp{\mathcal{A}^\prime}$ decomposes words in the large alphabet $\mathcal{A}$, into all sub-words in $\mathcal{A} \setminus{\mathcal{A}^\prime}$ separated by letters in $\mathcal{A}^\prime$. For example $\mathcal{A}= \{\blue{1},\blue{2},\red{3},\red{4}\}$ and $\mathcal{A}'= \{\blue{1},\blue{2}\}$\begin{equation*}
    \decomp{\mathcal{A}'}(\emptyset):= \{(\emptyset),\emptyset\}, \quad \decomp{\mathcal{A}'}(\red{3}\blue{2}\red{43}\blue{1})= \{(\red{3},\red{43},\emptyset),\blue{21}\}, \quad \decomp{\mathcal{A}'}(\blue{12}\red{334}\blue{1}\red{3})=\{(\emptyset,\emptyset,\red{334},\red{3}),\blue{121}\}
\end{equation*}

\noindent Let us also introduce the $n$-simplex over the interval $[a,b]$, given by \[
\Delta_{a,b}^{n}:= \{(t_1,\dots,t_n) \in [s,t]^n: a \leq t_1 \leq t_2 \leq \cdots \leq t_n \leq b \}.
\]
\subsection{Quotienting by the commutators of a subalphabet}
\label{sec:prelim_quo}

We introduce a quotient of the free associative algebra
$\mathbb{R}\langle\mathcal{A}\rangle$ which imposes commutativity between
letters in $\mathcal{A}^{\prime\prime}$. To this end, we define
\begin{equation}\label{def:partial_commutative_ideal}
    \mathcal{J}_{\mathcal{A}^{\prime\prime}}
    :=
    \operatorname{span}\left\{
        u\big(\orange{i}\orange{j}-\orange{j}\orange{i}\big)v
        \;:\;
        u,v\in\mathcal{W}_{\mathcal{A}},\quad
        \orange{i},\orange{j}\in\mathcal{A}^{\prime\prime}
    \right\}
    \subseteq \mathbb{R}\langle\mathcal{A}\rangle.
\end{equation}
Equivalently, $\mathcal{J}_{\mathcal{A}^{\prime\prime}}$ is the two-sided ideal
generated by the commutators $\{
        \orange{i}\orange{j}-\orange{j}\orange{i}
        \;\vert\;
        \orange{i},\orange{j}\in\mathcal{A}^{\prime\prime}\}$.
We consider the quotient map
\begin{equation*}
    [\cdot] = [\cdot]_{\mathcal{A}^{\prime\prime}} :\;
    \mathbb{R}\langle\mathcal{A}\rangle
    \;\to\;
    \mathbb{R}_{\mathcal{A}^{\prime\prime}}\langle\mathcal{A}\rangle
    :=
    \mathbb{R}\langle\mathcal{A}\rangle/
    \mathcal{J}_{\mathcal{A}^{\prime\prime}}.
\end{equation*}
Clearly, the quotient is an algebra with the partially symmetrized
concatenation product
\begin{equation*}
    [w][v]=[wv],
\end{equation*}
and $[\cdot]$ is canonically an algebra morphism.
In particular, for $u,v\in\mathcal{W}_{\mathcal{A}}$ and
$\orange{i},\orange{j}\in\mathcal{A}^{\prime\prime}$, it holds that $[u\orange{i}\orange{j}v]
    =
    [u\orange{j}\orange{i}v]$.
Thus, for $w\in\mathcal{W}_{\mathcal{A}}$ with
$\decomp{\mathcal{A}^{\prime}}(w)
    =
    \left(
        (w_0,\dots,w_k),
        \blue{j_1}\cdots\blue{j_k}
    \right)$,
the quotient class $[w]$ admits the representation
\begin{equation*}
    [w]
    =
    [w_0]\blue{j_1}[w_1]\cdots\blue{j_k}[w_k].
\end{equation*}
The quotient above naturally translates to the other side of the dual pairing $\mathbb{R}\langle\mathcal{A}\rangle
    \times
    T((\mathbb{R}^d))$.
Specifically, we define
\begin{equation}\label{def:partial_commutative_tensor_algebra}
    T_{\mathcal{A}^{\prime\prime}}((\mathbb{R}^d))
    :=
    \prod_{n\geq 0}
    \operatorname{span}\left\{
        e_{[w]}
        \;\Big\vert\;
        w\in \mathcal{W}_{\mathcal{A}}^{(n)}
    \right\},
\end{equation}
and equip it with the product induced by concatenation of quotient words,
\begin{equation*}
    e_{[w]}\widehat{\otimes}e_{[v]}
    :=
    e_{[wv]},
\end{equation*}
and the pairing
\begin{equation*}
    \langle [w],e_{[v]}\rangle
    :=
    \mathbbm{1}_{\{[w]=[v]\}},
\end{equation*}
for all $w,v\in\mathcal{W}_{\mathcal{A}}$ extended by linearity. We denote the associated quotient map by
\begin{equation*}
    \symmetrizer ~=~ \symmetrizer_{{\ }\mathcal{A}^{\prime\prime}} :
    T((\mathbb{R}^d))
    \to
    T_{\mathcal{A}^{\prime\prime}}((\mathbb{R}^d)),
    \qquad
    e_w\mapsto\widehat{e_w}:=e_{[w]}.
\end{equation*}
In particular, for $\mathbf{a}\in T((\mathbb{R}^d))$ and
$w\in\mathcal{W}_{\mathcal{A}}$, we have
\begin{equation}\label{eq:quotient_pairing}
    \left\langle
        [w],
        \widehat{\mathbf{a}}
    \right\rangle
    =
    \sum_{v\in\mathcal{W}_{\mathcal{A}}:\,[v]=[w]}
    \langle v,\mathbf{a}\rangle.
\end{equation}
The graded transpose of $\symmetrizer$ is the injective map
\[
    \symmetrizer^\ast:
    \mathbb{R}_{\mathcal{A}^{\prime\prime}}
    \langle\mathcal{A}\rangle
    \longrightarrow
    \mathbb{R}\langle\mathcal{A}\rangle,
    \qquad
    \widehat{[w]}^\ast := \symmetrizer^\ast([w])
    :=
    \sum_{\substack{
        v\in\mathcal{W}_{\mathcal{A}}\\
        [v]=[w]
    }}
    v.
\]
Thus, for
$\alpha\in
\mathbb{R}_{\mathcal{A}^{\prime\prime}}\langle\mathcal{A}\rangle$
and $\mathbf{a}\in T((\mathbb{R}^d))$,
\begin{equation}\label{eq:quotient_transpose_pairing}
    \left\langle
        \alpha,
        \widehat{\mathbf{a}}
    \right\rangle
    =
    \left\langle
        \widehat{\alpha}^\ast,
        \mathbf{a}
    \right\rangle.
\end{equation}

The conventions from \Cref{sec:graded_morphisms} apply analogously to the
quotient word space. Thus, if a graded map
$\widehat{\ell}:T_{\mathcal{A}^{\prime\prime}}((\mathbb{R}^d))\to
T_{\mathcal{A}^{\prime\prime}}((\mathbb{R}^d))$ has row-vector representation
$(\widehat{\ell}_{[w]})_{[w]\in[\mathcal{W}_{\mathcal{A}}]}$, then its graded
transpose satisfies
$\widehat{\ell}^{\ast}([w])=\widehat{\ell}_{[w]}$.

\subsection{Graded Hopf algebras and ideals}

Hopf-algebra terminology is not strictly needed for the main practical results
of this paper. It does, however, provide a concise language for some of the
algebraic structures appearing below. We therefore recall only the basic
notions needed in the sequel and refer to
\cite{preiss2021hopf,sweedler1969hopf,reutenauer1993free} for further details.

For a completed graded vector space
$H=\prod_{n\geq0}H_n$, write
\[
    H\boxtimes H
    :=
    \prod_{m,n\geq0}H_m\otimes H_n
\]
for the completed external tensor product. We reserve $\boxtimes$ for the
tensor product occurring in coproducts, whereas $\otimes$ denotes the respective concatenation product on $T((\mathbb{R}^d))$.

A graded Hopf algebra is a graded unital algebra
$(H,\diamond)$ equipped with a compatible coproduct
\[
    \Delta_{\star}:H\longrightarrow H\boxtimes H,
\]
a counit $\varepsilon:H\to\mathbb{R}$, and an antipode $S:H\to H$. Here, the
subscript $\star$ records the product on the graded dual corresponding to
$\Delta_{\star}$, characterized by
\[
    \left\langle
        \alpha\star\beta,
        \mathbf{a}
    \right\rangle
    =
    \left\langle
        \alpha\boxtimes\beta,
        \Delta_{\star}(\mathbf{a})
    \right\rangle.
\]
We write $\mathcal{G}(H)$ for the set of group-like elements of $H$, that is,
\[
    \mathcal{G}(H, \diamond, \Delta_{\star})
    :=
    \left\{
        \mathbf{g}\in H
        \;\middle|\;
        \varepsilon(\mathbf{g})=1,\;
        \Delta_{\star}(\mathbf{g})
        =
        \mathbf{g}\boxtimes\mathbf{g}
    \right\}.
\]
A two-sided ideal $I\subseteq H$ with respect to $\diamond$ is a Hopf ideal if
\[
    \Delta_{\star}(I)
    \subseteq
    I\boxtimes H
    +
    H\boxtimes I,
    \qquad
    \varepsilon(I)=0,
    \qquad
    S(I)\subseteq I.
\]
Then $H/I$ inherits a unique Hopf algebra structure for which the quotient
map $H\to H/I$ is a Hopf algebra morphism; see
\cite[Theorem~4.3.1]{sweedler1969hopf}. The commutator ideal underlying
partial symmetrization, as well as the ideals used for truncation, will be
considered in the main text.

For the tensor algebra
\(
    (H,\diamond)
    =
    \big(
        T((\mathbb{R}^d)),
        \otimes
    \big),
\)
the unshuffle coproduct, counit, and antipode are given by
\[
\begin{aligned}
    \Delta_{\shuffle}(e_{\blue{i}})
    &=
    e_{\blue{i}}\boxtimes e_{\emptyset}
    +
    e_{\emptyset}\boxtimes e_{\blue{i}},
    \qquad
    \blue{i}\in\mathcal{A},
    \\
    \varepsilon(e_w)
    &=
    \mathbbm{1}_{\{w=\emptyset\}},
    \qquad
    w\in\mathcal{W}_{\mathcal{A}},
    \\
    S(e_{\blue{i_1}\cdots\blue{i_n}})
    &=
    (-1)^n
    e_{\blue{i_n}\cdots\blue{i_1}},
    \qquad
    \blue{i_1}\cdots\blue{i_n}\in\mathcal{W}_{\mathcal{A}}.
\end{aligned}
\]
where $\Delta_{\shuffle}$ is extended multiplicatively with respect to
$\otimes$. These maps equip $H$ with the structure of a graded Hopf algebra.
Moreover, the shuffle product from \eqref{def:shuffle-product} is dual to
$\Delta_{\shuffle}$, in the sense that
\[
    \left\langle
        u\shuffle v,\mathbf{a}
    \right\rangle
    =
    \left\langle
        u\boxtimes v,
        \Delta_{\shuffle}(\mathbf{a})
    \right\rangle,
    \qquad
    u,v\in\mathcal{W}_{\mathcal{A}},
    \quad
    \mathbf{a}\in T((\mathbb{R}^d)).
\]
Consequently, the set $G((\mathbb{R}^d))$ introduced above is precisely the
set of group-like elements of this Hopf algebra, i.e.
$G((\mathbb{R}^d))=\mathcal{G}(T((\mathbb{R}^d)), \otimes, \Delta_{\shuffle})$.

\subsection{Path regularity and signatures}\label{sec:paths_signatures}

We conclude this section with some basic notation and definitions for paths and
their signature lift; we refer to \cite{friz2010multidimensional} for further
details. We are mainly interested in continuous paths
$X:[0,T]\rightarrow E$, where $(E,d)$ is a metric space, which are of bounded
$p$-variation.

\begin{definition}
Let $p\geq 1$ and let $(E,d)$ be a metric space. For a path
$X:[s,t]\to E$, $0\leq s<t$, its $p$-variation over $[s,t]$ is defined by
\begin{equation}\label{def:p-var}
    \Vert X \Vert_{p\text{-var};[s,t]}
    :=
        \sup_{\mathcal{P} \in \mathcal D([s,t])}
        \left(\sum_{[u,v] \in \mathcal{P}}
        d(X_{u},X_{v})^p
    \right)^{1/p},
\end{equation}
where the supremum is taken over all finite partitions
$\mathcal D([s,t])=\{s=t_0<t_1<\cdots<t_n=t\}$ of $[s,t]$.
We denote by $C^{p\text{-var}}([s,t],E)$ the space of all continuous paths
$X:[s,t]\to E$ such that
$\Vert X\Vert_{p\text{-var};[s,t]}<\infty$.
\end{definition}

In this paper we only consider Euclidean state spaces $E=\mathbb R^d$, for some
$d\in\mathbb N$, and we use the Euclidean distance throughout. For continuous
paths of finite variation, that is $X\in C^{1\text{-var}}([s,t],\mathbb R^d)$,
we define the signature of $X$ over $[s,t]$ as the formal tensor series
\begin{equation}\label{def:signature}
    \Sig{X}
    :=
    \left(
        1,
        \int_{s<u<t} \dd X_u,
        \dots,
        \int_{s<u_1<\cdots<u_n<t}
        \dd X_{u_1}\otimes\cdots\otimes \dd X_{u_n},
        \dots
    \right)
    \in G((\mathbb R^d)),
\end{equation}
where the iterated integrals are understood in the Riemann-Stieltjes sense;
see, e.g., \cite[Chapter~7.2]{friz2010multidimensional}. For any subinterval
$[u,v]\subseteq[s,t]$, we write $\Sig{X}_{u,v}
    :=
    \Sig{X|_{[u,v]}}$.

In Section~\ref{sec:sig_transform} we will consider signatures of paths
$Y=(A,X)$ whose components may have different regularity. In particular, for the
universality and characteristicness results below, we impose the following
richness condition on $A$. Recall that a path $A:[0,T]\to\mathbb R^m$ is
called Lipschitz continuous if there exists a constant $L<\infty$ such that $|A_t-A_s|\leq L|t-s|$ for $s,t\in[0,T]$.

\begin{Assumption}\label{ass:mambo}
Suppose $A \in C^{1-var}([0,T];\R^m)$ with $A=(A^\blue{a})_{\blue{a}\in \mathcal{A}'}$ and some alphabet $\mathcal{A}'=\{\blue{a_1},\dots,\blue{a_m}\}$, is such that: $\exists (w_n)_{n\in\mathbb N}\subset \mathcal{W}_{\mathcal{A}'}$ such that for any $t_0 \in (0,T]$, the family
\begin{equation}\label{eq:dense_familiy}
    \Big\{[0,t_0] \ni
        t\mapsto \partial_t
        \langle w_n,\operatorname{Sig}(A)_{t,t_0}\rangle
        :
        n\in\mathbb N
    \Big\} \subset \mathrm{Lip}([0,t_0],\mathbb{R})
\end{equation}
is total in $L^2([0,t_0],\mathbb R)$.
\end{Assumption}
To ensure injectivity for the transforms presented in the sequel, we will restrict to augmentations which coincide along a sequence of words, for which the assumption above holds. To make this rigorous, for a fixed path $A \in C^{1-var}([0,T],\mathbb{R}^m)$ such that the above assumption holds with respect to some alphabet $\mathcal{A}'$ and a sequence $\mathbf{w}=(w_n)_{n\in \mathbb{N}}$, we define the $\mathbf{w}$-fibre of $A$ as \begin{equation}
    \label{eq:fibre}
    [A]_{\mathbf{w}}:= \big \{  \tilde{A} \in C^{1-var}([0,T],\R^m): \langle w_n,\Sig{A}_{s,t}\rangle = \langle w_n,\Sig{\tilde{A}}_{s,t}\rangle, \, \forall n \in \mathbb{N},\,  \forall s\leq t\big \} 
\end{equation}
We illustrate the idea in the following example. 

\begin{example}
Assumption~\ref{ass:mambo} is in particular satisfied if $A$ contains a time
component, say $A_t^j=t$ for some $1\leq j\leq m$. Choosing
$w_n=j^{\otimes n}=j\cdots j$, one easily sees that
$\langle w_n,\Sig{A}_{t,t_0}\rangle=(t_0-t)^n/n!$, and hence
$\partial_t\langle w_n,\Sig{A}_{t,t_0}\rangle=-(t_0-t)^{n-1}/(n-1)!$, which is dense in $L^2([0,t_0],\mathbb R)$ by the Weierstrass
theorem. In this case, the fibre $[A]_\mathbf{w}$ consists of any $1$-variation path which has a time component. 
\end{example}
\begin{remark}
The conclusion of the last example more generally holds if $A$ has a scalar component
$A_t^j=:\phi(t)$ which is Lipschitz such that
$\phi'(t)>0$ (or $\phi'(t)<0)$ for a.e. $t$. Indeed, for $w_n=j^{\otimes n}$ we have \begin{equation*}
	    \partial_t \langle w_n,\Sig{A}_{t,t_0}\rangle = -\frac{ \phi'(t)(\phi(t_0)-\phi(t))^{n-1}}{(n-1)!}, \quad t \in [0,t_0], \quad n\in \mathbb{N}.
	\end{equation*} Since $\phi$ is strictly monotone, the claim
again follows from polynomial density after the change of variables induced by
$\phi$. Typical examples include $\phi(t)=e^{\lambda t}$, $\phi(t)=\log(1+t)$, and
$\phi(t)=t^\alpha$ for $\alpha\geq1$. In practice, a typical deterministic augmentation of interest is
$\phi(t)=\mathbb E[X_t^2]$ for a Gaussian process $X$, which fits this framework
whenever it satisfies the above monotonicity and regularity assumptions. For
semimartingales $X$, another natural quantity is the quadratic variation
process $A_t=[X]_t$, which typically satisfies the assumptions pathwise.\end{remark}

Finally, if $X$ has finite $p$-variation for some $p>1$, the definition of the
signature in \eqref{def:signature} is no longer immediate, since the
Riemann--Stieltjes meaning of $\dd X$ may fail. If $1<p<2$, the iterated integrals
in \eqref{def:signature} can still be defined by Young integration; see
\cite[Chapter~6]{friz2010multidimensional}. For $p>2$, a situation which is
typical for the sample-path regularity of many continuous-time stochastic
processes, even second-level terms such as $\int X^i\dd X^j$ are in general not
canonically defined as limits of Riemann sums from the path $X$ alone. A fundamental insight of Lyons' rough path theory
\cite{lyons1998differential} is that the missing information is precisely
encoded by the first $\lfloor p\rfloor$ iterated integrals. Thus, instead of the
path increments $X_{s,t}=X_t-X_s$ alone, one specifies an enhancement
\[
    \mathbf X_{s,t}
    =
    \big(
        1,X_{s,t},
        \mathbb X^{(2)}_{s,t},
        \dots,
        \mathbb X^{(\lfloor p\rfloor)}_{s,t}
    \big)
    \in G^{\lfloor p\rfloor}(\mathbb R^d),
\]
whose higher levels should be thought of as the iterated integrals of
$X$. Under algebraic and analytic assumptions mimicking those of truncated
signatures, Lyons' extension theorem \cite[Theorem~2.2.1]{lyons1998differential}
shows that such an enhanced path admits a unique extension to all higher levels,
that is a full signature $\Sig{\mathbf X}\in G((\mathbb R^d))$. These conditions
lead to the notion of a $p$-rough path, which we recall next.
\begin{definition}
    For $p\geq 1$, a weakly geometric $p$-rough path on $\mathbb{R}^d$ is a map \begin{equation*}
        \mathbf{X}=(1,\mathbb{X}^{(1)},\dots, \mathbb{X}^{(\lfloor p \rfloor)}): \Delta_{0,T}^2 \rightarrow G^{\lfloor p \rfloor}(\R^d),
    \end{equation*} with the following three properties \begin{enumerate}
        \item $\mathbf{X}$ is of finite $p$-variation, in the sense that $$\Vert \mathbf{X} \Vert_{p-var}:=\sup_{\mathcal{P}\in \mathcal{D}([0,T])}\max_{1\leq k \leq \lfloor p \rfloor}   \Big (\sum_{[u,v] \in \mathcal{P}} \Vert \mathbb{X}^{(k)}_{u,v} \Vert_{(\mathbb{R}^d)^{\otimes k}}^{p/k} \Big )^{k/p}<\infty.$$

        \item Chen's relation holds: $\mathbf{X}_{s,t} = \mathbf{X}_{s,u} \otimes \mathbf{X}_{u,t}$ for all $0 \leq s \leq u \leq t \leq T$.
        \item The shuffle-identity holds: $\langle w,\mathbf{X}_{s,t} \rangle\langle v,\mathbf{X}_{s,t} \rangle = \langle w\shuffle v,\mathbf{X}_{s,t} \rangle $ for all $(s,t)\in \Delta_{0,T}^{(2)}$ and $w,v \in \mathcal{W}$ with $|wv|\le \lfloor p \rfloor$.
    \end{enumerate}
    Finally, a weakly geometric rough path is called geometric, if it is the limit of a sequence $\pi_{\leq \lfloor p \rfloor}(\Sig{X^n})$, where $(X^n) \subset C^{1-var}$ with respect to the $p$ rough path norm $\Vert \cdot \Vert_{p-var}$. We denote by $\mathscr{C}^{p-var}_g([0,T|;\mathbb{R}^d)$ the space of weakly, resp. by $\mathscr{C}_{g}^{p-var,0}([0,T];\mathbb{R}^d)$ the space of geometric rough paths.
\end{definition}

\section{Signature transformations}\label{sec:sig_transform}

In this section, we introduce the change of coordinates for signatures that
lies at the center of this work.
Writing $Y=(A,X)$ for the partition of the components underlying
path, we provide an explicit construction of a graded automorphism
$\Psi:T((\mathbb{R}^d))\to T((\mathbb{R}^d))$ by recursively applying
integration by parts. With respect to the pairing induced by $\Psi$, the
signature coordinates are expressed entirely through the partial signature
of $X$ and iterated integrals against $\dd A$. The inverse $\Psi^{-1}$ is
likewise constructed explicitly. We then identify the product
$\gammashuffle$ that is dual to the  unshuffle coproduct under
this induced pairing, i.e., such that 
\begin{equation}\label{eq:char_gamma}
        \left\langle
        \alpha,
        \Sig{Y}
    \right\rangle_{\Psi}
    \left\langle
        \beta,
        \Sig{Y}
    \right\rangle_{\Psi} = \left\langle
        \alpha\gammashuffle\beta,
        \Sig{Y}
    \right\rangle_{\Psi}.
\end{equation}
The product admits a simple recursive description,
which in turn yields a direct recursion for the rows of $\Psi^{-1}$.

We then partially symmetrize over the alphabet
$\mathcal{A}^{\prime\prime}$. On the tensor side, this amounts to passing to
a genuine quotient of the tensor Hopf algebra. In the induced coordinates,
the partial signature terms of $X$ reduce to multivariate monomials in its
increments, while the corresponding dual product descends to the explicitly
defined product $\gammashufflesym$, which again admits a closed recursion.
We further establish an injectivity result for the partially symmetrized
signature.

Finally, we extend the coordinate representations to partially irregular
paths. In the unsymmetrized case, the construction requires a rough path
enhancement of $X$, whereas the partially symmetrized coordinates depend
only on the increments of $X$. This provides a natural framework for joint
rough and stochastic lifts and clarifies the relation to the partial rough
path spaces.

\subsection{Coordinate transformation by partial integration}\label{sec:pathwise_transform}
For any path of bounded variation $Y \in C^{1-var}([s,t],\mathbb{R}^{d})$ and fixed $(s,t) \in \Delta_{0,T}^{2}$, recall the signature $\Sig{Y}$ was defined in \eqref{def:signature}. Consider now an alphabet $\mathcal{A} = \{a_1, \dots, a_d\}$ and denote $Y^{a} = \langle a, Y\rangle$ for all $a\in\mathcal{A}$, where we recall the dual-pairing $\langle \cdot, \cdot \rangle$ introduced in Section~\ref{sec:tensor_shuffle_algebra}. In particular, we can pair any word $w=i_1\cdots i_n$ in the alphabet $\mathcal{A}$ with the signature \begin{equation}\label{eq:sig_pairing}
    \langle {i_1}\cdots {i_n},\Sig{Y}\rangle = \int_{{s<u_1<\cdots<u_n<t}}\dd {Y}^{{i_1}}_{u_1}\cdots \dd {Y}_{u_n}^{{i_n}}. 
\end{equation}
For a given partition of the alphabet\footnote{We always assume that for such a partition it holds $\mathcal{A}^\prime, \mathcal{A}^{\prime\prime} \neq \emptyset$.} $\mathcal{A} = \mathcal{A}^\prime \dot{\cup} \mathcal{A}^{\prime\prime}$
we divide the path $Y$ into two components \begin{equation}\label{eq:path_components}
    A = (A^{\blue{a}})_{\blue{a}\in\mathcal{A}^{\prime}} := (Y^{\blue{a}})_{\blue{a}\in\mathcal{A}^{\prime}},\qquad X = (X^{\orange{b}})_{\orange{b}\in\mathcal{A}^{\prime\prime}} := (Y^{\orange{b}})_{\orange{b}\in \mathcal{A}^{\prime\prime}}.
\end{equation}
\begin{example}\label{example:components}
    Consider a five-dimensional path $t\mapsto Y_t= (Y^1_t,\dots,Y^5_t)^\top$ and  $\mathcal{A}'=\{\blue{1},\blue{4}\} \subset \mathcal{A}=\{\blue{1},\red{2},\red{3},\blue{4},\red{5}\}$. The two components $(A,X)$ in our notation \eqref{eq:path_components} then read $$A_t =(A_t^1,A_t^4)^\top := (Y_t^1,Y_t^4)^\top, \qquad X_t=(X_t^2,X_t^3,X_t^5)^\top:= (Y_t^{2},Y_t^{3},Y_t^{5})^\top.$$
\end{example}
\begin{remark}\label{rmk:partial_sigs}
By restricting to words $w = \orange{i_1}\cdots \orange{i_n} \in \mathcal{W}_{\mathcal{A}^{\prime\prime}}$, we can extract the signature of $X$ from the signature of $Y$, setting
\begin{equation}\label{eq:partial_sig_X}
\langle w,\Sig{X}\rangle := \langle w,\Sig{Y}\rangle =\int_{{s<u_1<\cdots<u_n<t}}\dd{X}^{\orange{i_1}}_{u_1}\cdots \dd {X}_{u_n}^{\orange{i_n}}.
\end{equation}
Similarly we can extract the signature of the component $A$, restricting to words in $\mathcal{A}'$.    
\end{remark}
The choice of  $\mathcal{A}'\subset\mathcal{A}$ splits the
path $Y$ into the two components $(A,X)$ in
\eqref{eq:path_components}, which may play fundamentally different roles
in applications. The full signature $\Sig{Y}$ contains the partial
signatures $\Sig{A}$ and $\Sig{X}$, together with all mixed iterated
integrals against $\dd A$ and $\dd X$. The aim of this section is
to introduce a representation of $\Sig{Y}$ in different coordinates, in the sense of Definition~\ref{def:induced_pairing}, in which
the signature is represented only through  $\Sig{A}$, $\Sig{X}$, and
iterated integrals against the component $\dd A$. As we will see in the
subsequent sections, this coordinate representation is particularly
useful for deriving expected-signature formulas when $A$ is
deterministic and sufficiently regular. 

To achieve this, we will construct a graded linear map  $\Psi:T((\mathbb{R}^d))\rightarrow T((\mathbb{R}^d))$, such that the signature in $\Psi$-coordinates $(\langle w,\Sig{Y}\rangle_\Psi)_w$ (see Definition~\ref{def:induced_pairing}) is of the desired form. The
construction is based on a recursive application of integration by
parts, and we illustrate now the underlying idea at level two.

\begin{example}\label{ex:level_2_construction}
    For some alphabet $\mathcal{A}= \mathcal{A}'\cup \mathcal{A}''$, the second-level of $\Sig{Y}$ consists of \[
    \langle \blue{ij},\Sig{A}\rangle, \quad \langle \orange{ij},\Sig{X}\rangle, \quad \int_s^tX_{s,u}^{\orange{i}}\dd A_u^{\blue{j}}, \quad \int_s^tA_{s,u}^{\blue{i}}\dd X_u^{\orange{j}}, \qquad i,j \in \mathcal{A}.
    \]
    We see that the first three elements can be build from $\Sig{A},\Sig{X}$ or integration only against $\dd A$. For the last integral, an application of integration by parts shows $\int A \dd X=AX-\int X \dd A$, or in terms of words  $\blue{i}\orange{j}=\blue{i}\shuffle \orange{j}-\orange{j}\blue{i}$. Motivated by this, we define the sequence $(\Psi_w)_{w\in \mathcal{W}^{(2)}}$ by $\Psi_{w} =\blue{i}\shuffle \orange{j}$ if $w=\blue{i}\orange{j}$ and $\Psi_w=w$ else, so that (see Definition~\ref{def:induced_pairing}) $$\langle w, \Sig{Y}\rangle_{\Psi} =A_{s,t}^{\blue{i}}X_{s,t}^{\orange{j}} \quad \text{ if } w=\blue{i}\orange{j} \quad \text{ and } \quad \langle w, \Sig{Y}\rangle_{\Psi}= \langle w,\Sig{Y}\rangle \quad \text{ else.} $$ 
    Hence, all the level two $\Psi$-coordinates have the desired form. Moreover, the map $\Psi$ is clearly invertible, defining $\Psi^{-1}_{\blue{i}\orange{j}}=\blue{i}\orange{j}-\orange{j}\blue{i}$ and $\Psi^{-1}_w=w$ else, so that the standard coordinates can be recovered from $\langle w, \Sig{Y}\rangle=\langle \Psi^{-1}_w,\Sig{Y}\rangle_{\Psi}$.
\end{example}

In the remainder of this section, we generalize the preceding example
to all levels of the signature. More precisely, we construct a sequence
$(\Psi_w)_{w\in\mathcal W}$ such that the corresponding
$\Psi$-coordinates of $\Sig{Y}$ have the desired form. We then show that
the associated graded linear map (see \eqref{eq:linear_operator})
\begin{equation}\label{eq:linear_operator_2}
    \Psi(\mathbf a)
    =
    \sum_{w\in\mathcal W}
    e_w\langle\Psi_w,\mathbf a\rangle,
    \qquad
    \mathbf a\in T((\mathbb R^d)),
\end{equation}
is a graded automorphism. We 
construct its inverse explicitly  through an inverse sequence
$\bigl(\Psi^{-1}_w\bigr)_{w\in\mathcal W}$, so that the
standard signature coordinates can be recovered from the
$\Psi$-coordinates with
\[
    \langle w,\Sig{Y}\rangle
    =
    \langle\Psi^{-1}_w,\Sig{Y}\rangle_\Psi,
    \qquad
    w\in\mathcal W.
\] which we will exploit in the forthcoming sections. For our first result, recall the decomposition operator $\Delta_{\mathcal{A}'}$ defined in \eqref{def:Delta_operator}.
\begin{proposition}\label{prop:psi_inv}
    We recursively define a graded sequence $(\Psi_w)_{w \in \mathcal{W}_{\mathcal{A}}}\subset \mathbb{R}\langle \mathcal{A} \rangle $ by
    \begin{equation}\label{eq:Psi_inverse_explicit}
        \begin{dcases}
        \bigg.\Psi_{w} = w, &\text{for } w \in \mathcal{W}_{\mathcal{A}^{\prime\prime}}, \\   
     \Psi_w= \left (\Psi_{w_0\blue{j_1}\cdots \blue{j_{k-1}}w_{k-1}} \right )\blue{j_k} \shuffle w_k, & \text{for } w \in \mathcal{W}_{\mathcal{A}}\setminus\mathcal{W}_{\mathcal{A}^{\prime\prime}},
    \end{dcases}
    \end{equation}  
    with $((w_0, \dots, w_{k}), \blue{j_1 \cdots j_k}) = \decomp{\mathcal{A}^\prime}(w) $.
    Then, for any $Y=(A,X)\in C^{1-var}([s,t]; \R^d)$, we have \begin{equation}\label{def:gamma_signature}
        \langle w, \Sig{Y} \rangle_{\Psi} = 
        \langle w_k, \Sig{X}\rangle\int_{\Delta^{k}_{s,t}}\prod_{j=1}^{k}  
        \langle w_{j-1}, \Sig{X}_{s,t_j}\rangle \dd A^{\blue{i_{j}}}_{t_{j}},
    \end{equation} where we recall $\langle w,\cdot \rangle_\Psi:=\langle \Psi_w,\cdot \rangle$.  
    
\end{proposition}

\begin{proof}
We proceed by induction over the number of decomposing factors $k\in\N$ in $\Delta_{\mathcal{A}^\prime}(w) =((w_0, \dots, w_{k}), \blue{j_1 \cdots j_k})$, that is the number of letters $\blue{j_1},\dots,\blue{j_k} \in \mathcal{A}'$. For $k=0$, i.e.~$w\in \mathcal{W}_{\mathcal{A}^{\prime\prime}}$,
the claim follows directly from the fact that $\Psi_w=w$.
Now assume that the claim holds for some $k\in\N$. 
Then for $w\in\mathcal{W}_{\mathcal{A}}$ with $((w_0, \dots, w_{k+1}), \blue{j_1 \cdots j_{k+1}}) = \decomp{\mathcal{A}^\prime}(w)$ we have by the induction hypothesis that 
$$\langle \Psi_{w_0\blue{j_1}w_1\cdots \blue{j_k}w_k}, \Sig{Y}\rangle = \langle w_k, \Sig{X}\rangle\int_{\Delta^{k}_{s,t}}\prod_{j=1}^{k}  
        \langle w_{j-1}, \Sig{X}_{s,t_j}\rangle \dd A^{\blue{i_{j}}}_{t_{j}}.$$
Since $\Sig{Y}_{s,u} = \Sig{Y_{\cdot \wedge u}}$ for all $u\in[s,t]$ the above identity holds also on restricted time intervals.
Hence, by definition of $\Psi_w$ and the shuffle property of $\mathrm{Sig}$, we have
\begin{align*}
    \langle \Psi_w,\Sig{Y} \rangle 
    & := \Big.\langle \left (\Psi_{w_0\blue{j_1}w_1\cdots \blue{j_k}w_k}\right )\blue{j_{k+1}}\shuffle w_{k+1}, \Sig{Y}\rangle \\ & 
    = \langle w_{k+1}, \Sig{X} \rangle\int_s^t \langle \Psi_{w_0\blue{j_1}w_1\cdots \blue{j_k}w_k}, \Sig{Y}_{s,u}\rangle \dd A_u^{\blue{j_{k+1}}} \\ &
    = \langle w_{k+1}, \Sig{X}\rangle\int_s^t \langle w_k, \Sig{X}_{s,u}\rangle\int_{\Delta^{k}_{s,u}}\prod_{j=1}^{k}  
        \langle w_{j-1}, \Sig{X}_{s,t_j}\rangle \dd A^{\blue{i_{j}}}_{t_{j}} \dd A_u^{\blue{j_{k+1}}} \\ &
    = \langle w_{k+1}, \Sig{X}\rangle\int_{\Delta^{k+1}_{s,t}}\prod_{j=1}^{k+1}\langle w_{j-1}, \Sig{X}_{s,t_j}\rangle \dd A^{\blue{i_{j}}}_{t_j},
\end{align*} 
for all $Y\in  C^{1-var}([s,t],\R^d)$, which proves the claim.

\end{proof}

In the following example we provide some explicit computations of the operator $\Delta_{\mathcal{A}'}$  and the $\Psi$-coordinates.
\begin{example}\label{example:components_continued}
Remaining in the framework of Example \ref{example:components}, we find \begin{enumerate}
        \item For $w=\orange{i_1\cdots i_k}\in \mathcal{W}_{\mathcal{A}^{\prime\prime}}$, we have  $\Delta_{\mathcal{A}'}(w)=\{(\orange{i_1}\cdots \orange{i_k}),\emptyset\}$ and $\langle w,\Sig{Y}\rangle_{\Psi}= \langle w,\Sig{X}\rangle$.
        \item For $w=\blue{j_1}\cdots \blue{j_k} \in \mathcal{W}_{ \mathcal{A}'}$, we have  $\Delta_{\mathcal{A}'}(w)=\{(\emptyset, \cdots, \emptyset),\blue{j_1}\cdots \blue{j_k}\}$ and $$\langle w,\Sig{Y}\rangle_{\Psi} := \langle \emptyset,\Sig{X} \rangle \int_{\Delta_{s,t}^{k}}\prod_{j=1}^{k} \langle \emptyset,\Sig{X}_{s,t_j}\rangle \dd A_{t_j}^{\blue{i_j}} = \langle w, \Sig{A}\rangle.$$ 

        \item For $w=\blue{1}\orange{25}\blue{41}\orange{32}\in \mathcal{W}_{\mathcal A}$ we have $\Delta_{\mathcal{A}'}(w)=\{(\emptyset,\orange{25},\emptyset,\orange{32}),\blue{141}\}$ and \begin{align*}
            \langle w,\Sig{Y}\rangle_\Psi & =  \langle \orange{32},\Sig{X}\rangle \int_{\Delta_{s,t}^3}\langle \emptyset,\Sig{X}_{s,u_1}\rangle \langle \orange{25},\Sig{X}_{s,u_2}\rangle\langle \emptyset,\Sig{X}_{s,u_3}\rangle \dd A_{u_1}^{\blue{1}}\dd A_{u_2}^{\blue{4}}\dd A_{u_3}^{\blue{1}} \\ & =  \langle \orange{32},\Sig{X}\rangle \int_{\Delta_{s,t}^3}\langle \orange{25},\Sig{X}_{s,u_2}\rangle \dd A_{u_1}^{\blue{1}}\dd A_{u_2}^{\blue{4}}\dd A_{u_3}^{\blue{1}}.
        \end{align*}
    \end{enumerate}   
\end{example}
\begin{remark}\label{rem:shear_like} The linear map $\Psi$ defined in \eqref{eq:linear_operator_2} through the
sequence $(\Psi_w)$ of Proposition~\ref{prop:psi_inv} may be viewed as a
graded shear-like transformation, in analogy with elementary shears in affine
geometry. Given a vector space $V$, a shear map (see, e.g., \cite[Section 2.7]{gallier2011geometric}) $\Phi:V\rightarrow V$ translates all vectors parallel to a given subspace $W\subset V$. More precisely, writing $V=W'\oplus W$, a shear map is such that $\Phi(w',w)=(w',w)+(0,L(w'))$, for some linear $L:W'\rightarrow W$. 

Restricted to tensor level two, our transform $\Psi$ reads \begin{equation*}\label{eq:shears_example}
\begin{aligned}
&\Psi\big(
    (\blue{ij},{\orange{ij}}),
    ({\orange{j}\blue{i}},{\blue{i}\orange{j}})
\big)
 =
\big(
    (\blue{ij},\orange{ij}),
    (\orange{j}\blue{i},\blue{i}\shuffle\orange{j})
\big)
 =
\big(
    (\blue{ij},\orange{ij}),
    (\orange{j}\blue{i},\blue{i}\orange{j})
\big)
+
\big(
    (0,0),
    (0,\orange{j}\blue{i})
\big),
\end{aligned}
\end{equation*} which is by definition a shear map with respect to the subspace of mixed words. Now by definition, for higher tensor-levels the map $\Psi$ continues to leave $\mathcal{W}_{\mathcal{A}'}$ and $\mathcal{W}_{\mathcal{A}''}$ invariant, and the transformation is recursively applied together with shuffling $\mathcal{A}''$-words, which is a linear operation. With an inductive argument one may show that for any $n\in \mathbb{N}$, the map $\Psi$ factors into a finite composition of shear maps.
\end{remark}
At this point it is far from clear that $\Psi$ is invertible. In the next result we construct its inverse, which compared to $\Psi$ is only implicitly defined. 
\begin{proposition}\label{prop:Psi}
    Consider the following system of equations for a sequences $(\Psi^{-1}_w)_{w\in \mathcal{W}_\mathcal{A}} \subset \R\langle \mathcal A\rangle$:
    \begin{equation}\label{eq:implicit_psi}
    \begin{dcases}
        \bigg.\Psi^{-1}_{w} = w, &\text{for } w \in \mathcal{W}_{\mathcal{A}^{\prime\prime}}, \\   
        \begin{split}
            \Psi^{-1}_w =& \left (\Psi^{-1}_{w_0\cdots \blue{j_{k-1}}w_{k-1}}\right )\blue{j_{k}}\orange{i_1 \cdots i_m}\\
        &- \sum_{l=1}^{m} \sum_{v\in \mathcal{W}^{(|w|)}_\mathcal{A}} \big\langle \big(w_0\cdots \blue{j_{k-1}}w_{k-1} \shuffle \orange{i_1}\cdots \orange{i_{l}}\big)\blue{j_k}\orange{i_{l+1}}\cdots \orange{i_m}, \,e_v \big\rangle \Psi^{-1}_{v}
        \end{split}, & \text{for } w \in \mathcal{W}_{\mathcal{A}}\setminus\mathcal{W}_{\mathcal{A}^{\prime\prime}},
    \end{dcases}
    \end{equation}  with $((w_0, \dots, w_{k-1}, \orange{i_1 \cdots i_m}), \blue{j_1 \cdots j_k}) = \decomp{\mathcal{A}^\prime}(w)$.
    It holds that \eqref{eq:implicit_psi} is uniquely solved by a graded sequence $(\Psi^{-1}_w)_{w\in\mathcal{W}_{\mathcal{A}}}$.
    Moreover, for any $Y\in C^{1-var}([s,t]; \R^d)$ we have   \begin{equation*}
        \langle w,\Sig{Y}\rangle = \langle \Psi_w^{-1},\Sig{Y}\rangle_\Psi.
    \end{equation*}
\end{proposition}

\begin{remark}
Although it may not be immediately visible, one observes that in coherence with Example~\ref{example:components_continued} it holds $\Psi^{-1}_v = v$ for all $v \in \mathcal{W}_{\mathcal{A}^\prime}$.
\end{remark}
\begin{proof}
    We will use a nested induction over the number of decomposing factors $k\in\mathbb{N}$ and the length of the suffix $m\in\N$ to prove that for all $w\in \mathcal{W}_\mathcal{A}$ the term $\Psi^{-1}_w$ is uniquely determined by \eqref{eq:implicit_psi}, $\Psi^{-1}_{w} \in \mathrm{span}(\mathcal{W}^{(|w|)})$ and that it holds $\langle \Psi^{-1}_w, \Sig{Y}\rangle_\Psi = \langle w, \Sig{Y}\rangle$ for all $Y\in C^{1-var}([s,t]; \R^d)$.
    For $k=0$, i.e., 
    for any $w \in \mathcal{W}_{\mathcal{A}^{\prime\prime}}$ the claim follows immediately since again  $\Psi^{-1}_w = w$.
    Now assume that for some $k \in \N$ the induction claim holds for all words $v \in \mathcal{W}_{\mathcal{A}}$ with $((v, \dots, v_{k}), \blue{j_1 \cdots j_k}) = \decomp{\mathcal{A}^\prime}(w)$.
    We next want to prove that the claim holds for all $w\in \mathcal{W}_{\mathcal{A}}\setminus\mathcal{W}_{\mathcal{A}^{\prime\prime}}$ with $((w_0, \dots, w_{k}, \orange{i_1 \cdots i_m}), \blue{j_1 \cdots j_{k+1}}) = \decomp{\mathcal{A}^\prime}(w)$.
    To this end we start a second induction over $m \in\mathbb{N}$. Let $Y\in C^{1-var}([s,t]; \R^d)$ be arbitrary.
    For $m=0$, the equation \eqref{eq:implicit_psi} simplifies to
    \begin{equation}\label{eq:induct_proof_psi_inv}\Psi^{-1}_{w}= \left (\Psi^{-1}_{w_0\cdots \blue{j_{k}}w_{k}}\right)\blue{j_{k+1}},\end{equation} which uniquely determines $\Psi^{-1}_{w}\in \mathrm{span}(\mathcal{W}^{(|w|)})$. Furthermore, we have
    \begin{align*}
        \langle \Psi^{-1}_w,\Sig{Y}\rangle_\Psi &= \int_s^t \langle \Psi^{-1}_{w_0\cdots \blue{j_{k}}w_{k}},\Sig{Y}_{s,u}\rangle dA^{\blue{j_{k+1}}}_u \\
        &=\int_s^t\langle w_0\cdots \blue{j_{k}}w_{k},\Sig{Y}_{s,u}\rangle dA_u^{\blue{j_{k+1}}}\\
        &=\langle w,\Sig{Y}\rangle,
    \end{align*}
    where the first equality uses Proposition~\ref{prop:psi_inv} and \eqref{eq:induct_proof_psi_inv}, the second equality uses the induction hypothesis, and the last equality the definition of the signature.
    For $m\ge 1$, we notice that the right-hand side of \begin{equation*}
    \Psi^{-1}_w = \left(\Psi^{-1}_{w_0\cdots \blue{j_{k}}w_{k}}\right)\blue{j_{k+1}}\orange{i_1 \cdots i_m} - \sum_{l=1}^{m} \sum_{v\in \mathcal{W}^{(|w|)}_\mathcal{A}} \big\langle \big(w_0\cdots \blue{j_{k}}w_{k} \shuffle \orange{i_1}\cdots \orange{i_{l}}\big)\blue{j_{k+1}}\orange{i_{l+1}}\cdots \orange{i_m}, \,e_v \big\rangle \Psi^{-1}_{v}
    \end{equation*}
   
    only requires the terms $\Psi^{-1}_{w_0\cdots \blue{j_{k}}w_{k}}$ and $\Psi^{-1}_{v}$ for $v\in \mathcal{W}_{\mathcal{A}}$ with $((v_0, \dots, v_{k+1}), \blue{j_1 \cdots j_{k+1}}) = \decomp{\mathcal{A}^\prime}(v)$ and $|v_{k+1}|< m$, which are already determined by the induction hypothesis. 
    One readily checks that also $\Psi^{-1}_{w} \in \mathrm{span}(\mathcal{W}^{(|w|)})$.
    Further, on the one hand
    \begin{align*}
       \big\langle \left (\Psi^{-1}_{w_0\cdots \blue{j_{k}}w_{k}}\right)\blue{j_{k+1}}w_{k+1}, \Sig{Y} \big\rangle_\Psi 
        &= \big\langle w_{k+1},\Sig{X}\big\rangle \big\langle\left(\Psi^{-1}_{w_0\cdots \blue{j_{k}}w_{k}}\right)\blue{j_{k+1}}, \Sig{Y}\big\rangle_\Psi \\ 
        &= \big\langle w_{k+1},\Sig{X}\big\rangle \big\langle w_0\cdots \blue{j_{k}}w_{k}\blue{j_{k+1}}, \Sig{Y}\big\rangle \\
        &= \big\langle (w_0\cdots \blue{j_{k}}w_{k}\blue{j_{k+1}}) \shuffle w_{k+1}, \Sig{Y} \big\rangle,
    \end{align*}
    where again the first equality uses Proposition~\ref{prop:psi_inv}, the second equality the induction hypothesis, and the last equality the shuffle property of the signature.
    On the other hand, it follows from linearity, the induction hypothesis and the shuffle identity for the signature that
    \begin{align*}
        \bigg\langle \sum_{l=1}^{m}\sum_{v\in \mathcal{W}_\mathcal{A}^{(|w|)}} &\big\langle \big(w_0\cdots \blue{j_{k}}w_{k} \shuffle \orange{i_1}\cdots \orange{i_{l}}\big)\blue{j_{k+1}}\orange{i_{l+1}}\cdots \orange{i_m}, \,e_v \big\rangle \Psi^{-1}_{v}, \;\Sig{Y}\bigg\rangle_\Psi  \\ 
        &= \sum_{l=1}^{m+1}\sum_{v\in \mathcal{W}_\mathcal{A}^{(|w|)}}\big\langle \big(w_0\cdots \blue{j_{k}}w_{k} \shuffle \orange{i_1}\cdots \orange{i_{l}}\big)\blue{j_{k+1}}\orange{i_{l+1}}\cdots \orange{i_m}, \,e_v \big\rangle  \left \langle \Psi^{-1}_{v}, \Sig{Y} \right \rangle_\Psi \\ 
        &= \sum_{l=1}^{m+1}\sum_{v\in \mathcal{W}_\mathcal{A}^{(|w|)}}\big\langle \big(w_0\cdots \blue{j_{k}}w_{k} \shuffle \orange{i_1}\cdots \orange{i_{l}}\big)\blue{j_{k+1}}\orange{i_{l+1}}\cdots \orange{i_m}, \,e_v \big\rangle  \left \langle v, \Sig{Y} \right \rangle \\ 
        &= \left\langle \sum_{l=1}^{m+1} \big(w_0\cdots \blue{j_{k}}w_{k} \shuffle \orange{i_1}\cdots \orange{i_{l}}\big)\blue{j_{k+1}}\orange{i_{l+1}}\cdots \orange{i_m}, \;\Sig{Y} \right \rangle.
    \end{align*} 

    Finally, the claim follows by noting that form the definition of the shuffle-product \eqref{def:shuffle-product} we have
    $$w = (w_0\cdots \blue{j_{k}}w_{k}\blue{j_{k+1}}) \shuffle w_k - \sum_{l=1}^{m} \big(w_0\cdots \blue{j_{k}}w_{k} \shuffle \orange{i_1}\cdots \orange{i_{l}}\big)\blue{j_{k+1}}\orange{i_{l+1}}\cdots \orange{i_m}.$$
\end{proof}

Combining Propositions \ref{prop:psi_inv} and \ref{prop:Psi}, the main theoretical result of this section can be summarized as follows.

\begin{theorem}\label{prop:main_result}
    The mappings $\Psi^{-1},\Psi: T((\R^d)) \rightarrow T((\R^d))$ are automorphisms such that $$\Psi^{-1} \circ \Psi = \Psi\circ \Psi^{-1}= \mathrm{Id}_{T((\R^d))}.$$ 
\end{theorem}
\begin{proof}
By Proposition~\ref{prop:Psi} and Definition~\ref{def:induced_pairing}, for any $\mathbf{a}=\Sig{Y} \in T((\mathbb{R}^d))$, we have \[
\langle w,\Psi^{-1}(\Psi(\mathbf{a}))\rangle \rangle:=\langle w,\Psi(\Psi^{-1}(\mathbf{a}))\rangle \rangle := \langle \Psi_w^{-1},\mathbf{a}\rangle_\Psi = \langle w,\mathbf{a}\rangle,
\] so that the claim holds restricted to signature elements. Moreover, by Propositions~\ref{prop:psi_inv}-\ref{prop:Psi}, the sequences $(\Psi^{-1}_w)_{w\in\mathcal{W}_\mathcal{A}}$ and $(\Psi_w)_{w\in\mathcal{W}_\mathcal{A}}$ are graded, so that $\Psi^{-1} \circ \Psi=\mathrm{Id}$ restricted to the set $$G^N := \{\pi_{\le N}\Sig{Y}: Y\in C^{1-var}([s,t],\R^d)\} \subset T^{\leq N}(\R^d).$$ 
On the other hand, it was shown in \cite[Lemma 5]{diehl2019invariants} that $$\mathrm{span}(G^N) = T^{\leq N}(\R^d).$$ 
Therefore $\Psi^{-1} \circ \Psi= \mathrm{Id}$ on $T^{\leq N}(\R^d)$, and by linearity this implies $\Psi\circ \Psi^{-1} = \mathrm{Id}$ on the same space. Since this holds for all truncation levels, this already concludes the proof. 
\end{proof}

We conclude this section by recording the Hopf-algebraic identities in the
$\Psi$-coordinates introduced above. In particular, we identify the product
dual to the unshuffle coproduct with respect to the induced pairing from
\Cref{def:induced_pairing}, and thereby obtain the resulting character identity
for the $\Psi$-coordinates of the signature. We then derive an intrinsic
recursive description of this product and use it to obtain a corresponding
recursion for the rows of $\Psi^{-1}$.

By the graded-transpose convention from \Cref{sec:graded_morphisms}, the maps $\Psi^\ast,
    (\Psi^{-1})^\ast
    \colon
    \mathbb{R}\langle\mathcal{A}\rangle
    \longrightarrow
    \mathbb{R}\langle\mathcal{A}\rangle$
satisfy
\[
    \Psi^\ast(w)=\Psi_w,
    \qquad
    (\Psi^{-1})^\ast(w)=\Psi^{-1}_w,
    \qquad
    (\Psi^{-1})^\ast=(\Psi^\ast)^{-1}
\]
for every $w\in\mathcal{W}_{\mathcal{A}}$. We therefore define the dual product with respect to the $\Psi$-pairing by 
\begin{equation}\label{eq:gamma_dual_product_explicit}
    \alpha\gammashuffle\beta
    :=
    (\Psi^{-1})^\ast
    \left(
        \Psi^\ast(\alpha)
        \shuffle
        \Psi^\ast(\beta)
    \right),
    \qquad
    \alpha,\beta\in\mathbb{R}\langle\mathcal{A}\rangle.
\end{equation}
Indeed, for every $\mathbf{a}\in T((\mathbb{R}^d))$, we then have
\[
    \langle \alpha\gammashuffle\beta,\mathbf{a}\rangle_{\Psi}
    =
    \left\langle \Psi^\ast(\alpha\gammashuffle\beta),\mathbf{a}\right\rangle
    =
    \left\langle \Psi^\ast(\alpha)\shuffle\Psi^\ast(\beta),\mathbf{a}\right\rangle
    =
    \left\langle \Psi^\ast(\alpha)\otimes\Psi^\ast(\beta),\Delta_{\shuffle}\mathbf{a}\right\rangle
    =
    \left\langle \alpha\otimes\beta,\Delta_{\shuffle}\mathbf{a}\right\rangle_{\Psi\otimes\Psi}.
\]
In particular this implied the character equation \eqref{eq:char_gamma} with respect to the $\Psi$-pairing.

The product $\gammashuffle$ admits the following intrinsic
recursive description. It separates the shuffle of the terminal
$\mathcal{A}^{\prime\prime}$-blocks from the branching at
$\mathcal{A}^{\prime}$-letters.

\begin{proposition}\label{prop:gamma_shuffle_recursion}
For $w,v\in\mathcal{W}_{\mathcal{A}}$,
$r,q\in\mathcal{W}_{\mathcal{A}^{\prime\prime}}$, and
$\blue{i},\blue{j}\in\mathcal{A}^{\prime}$,
\begin{equation}\label{eq:gamma_shuffle_recursion}
\left\{\begin{aligned}
r\gammashuffle q
&=
r\shuffle q,
\\
(w\blue{i}r)\gammashuffle q
&=
w\blue{i}(r\shuffle q),
\\
r\gammashuffle(v\blue{j}q)
&=
v\blue{j}(r\shuffle q),
\\
(w\blue{i}r)\gammashuffle(v\blue{j}q)
&=
\bigl((w\blue{i})\gammashuffle v\bigr)
\blue{j}(r\shuffle q)
+
\bigl(w\gammashuffle(v\blue{j})\bigr)
\blue{i}(r\shuffle q).
\end{aligned}\right.
\end{equation}
These identities uniquely determine $\gammashuffle$.
\end{proposition}
\begin{proof}
By \eqref{def:gamma_signature},
\[
    \left\langle
        w\blue{i}r,
        \Sig{Y}_{s,t}
    \right\rangle_\Psi
    =
    \left\langle r,\Sig{X}_{s,t}\right\rangle
    \int_s^t
    \left\langle w,\Sig{Y}_{s,u}\right\rangle_\Psi
    \,\dd A_u^{\blue{i}}.
\]
The first three identities follow directly from this relation and the shuffle
identity for $\Sig{X}$. We prove the last identity by induction on the total
number of $\mathcal{A}^{\prime}$-letters in the two factors.

Applying integration by parts to the two integrals and using the shuffle
identity for $\Sig{X}$ gives
\begin{align*}
\left\langle
    w\blue{i}r,
    \Sig{Y}_{s,t}
\right\rangle_\Psi
\left\langle
    v\blue{j}q,
    \Sig{Y}_{s,t}
\right\rangle_\Psi
&=
\left\langle
    r\shuffle q,
    \Sig{X}_{s,t}
\right\rangle
\int_s^t
\left\langle
    w\blue{i},
    \Sig{Y}_{s,u}
\right\rangle_\Psi
\left\langle
    v,
    \Sig{Y}_{s,u}
\right\rangle_\Psi
\,\dd A_u^{\blue{j}}
\\
&\quad+
\left\langle
    r\shuffle q,
    \Sig{X}_{s,t}
\right\rangle
\int_s^t
\left\langle
    w,
    \Sig{Y}_{s,u}
\right\rangle_\Psi
\left\langle
    v\blue{j},
    \Sig{Y}_{s,u}
\right\rangle_\Psi
\,\dd A_u^{\blue{i}}.
\end{align*}
By the induction hypothesis and \eqref{def:gamma_signature}, the right-hand
side equals
\[
    \left\langle
        \bigl(
            (w\blue{i}\gammashuffle v)\blue{j}
            +
            (w\gammashuffle v\blue{j})\blue{i}
        \bigr)(r\shuffle q),
        \Sig{Y}_{s,t}
    \right\rangle_\Psi.
\]
On the other hand, the character identity \eqref{eq:char_gamma} gives
\[
    \left\langle
        w\blue{i}r\gammashuffle v\blue{j}q,
        \Sig{Y}_{s,t}
    \right\rangle_\Psi
    =
    \left\langle
        w\blue{i}r,
        \Sig{Y}_{s,t}
    \right\rangle_\Psi
    \left\langle
        v\blue{j}q,
        \Sig{Y}_{s,t}
    \right\rangle_\Psi.
\]
Since $\Psi$ is an automorphism and truncated signatures span every truncated
tensor algebra, the truncations of $\Psi\circ\Sig{\cdot}$ have the same spanning
property. Hence the preceding identities yield
\eqref{eq:gamma_shuffle_recursion}.

The first three identities determine the cases in which one factor contains no
$\mathcal{A}^{\prime}$-letter, while the last identity decreases the total
number of $\mathcal{A}^{\prime}$-letters. Hence the recursion is unique.
\end{proof}

The same recursion also gives an explicit construction of the rows of
$\Psi^{-1}$. In contrast to the triangular recursion in \Cref{prop:Psi}, the
following formula reduces directly to rows with fewer
$\mathcal{A}^{\prime}$-letters.

\begin{corollary}\label{cor:Psi_gamma_recursion}
The rows $(\Psi^{-1}_w)_{w\in\mathcal{W}_{\mathcal{A}}}$ from
\Cref{prop:Psi} satisfy
\begin{equation}\label{eq:Psi_gamma_recursion}
\Psi^{-1}_r
=
r,
\qquad
\Psi^{-1}_{w\blue{i}\orange{i_1}\cdots\orange{i_m}}
=
\sum_{k=0}^{m}
(-1)^k
\bigl(
    \Psi^{-1}_w
    \gammashuffle
    \orange{i_k}\cdots\orange{i_1}
\bigr)
\blue{i}
\orange{i_{k+1}}\cdots\orange{i_m},
\end{equation}
for $w\in\mathcal{W}_{\mathcal{A}}$,
$\blue{i}\in\mathcal{A}^{\prime}$, and
$\orange{i_1},\ldots,\orange{i_m}\in\mathcal{A}^{\prime\prime}$, where
empty strings of letters are omitted.
\end{corollary}

\begin{proof}
By \eqref{eq:gamma_dual_product_explicit},
$\Psi^\ast$ intertwines $\gammashuffle$ and $\shuffle$ and fixes
every word in $\mathcal{W}_{\mathcal{A}^{\prime\prime}}$. Moreover,
\eqref{eq:Psi_inverse_explicit} gives
\[
\Psi^\ast(\alpha\blue{i}r)
=
\bigl(
    \Psi^\ast(\alpha)\blue{i}
\bigr)
\shuffle r.
\]
Applying $\Psi^\ast$ to the right-hand side of
\eqref{eq:Psi_gamma_recursion} yields
\[
\sum_{k=0}^{m}
(-1)^k
\bigl(
    (w\shuffle
    \orange{i_k}\cdots\orange{i_1})
    \blue{i}
\bigr)
\shuffle
\orange{i_{k+1}}\cdots\orange{i_m}
=
w\blue{i}\orange{i_1}\cdots\orange{i_m},
\]
where the last identity follows by repeated application of the shuffle
recursion. Since
\[
    \Psi^\ast
    \bigl(
        \Psi^{-1}_{w\blue{i}\orange{i_1}\cdots\orange{i_m}}
    \bigr)
    =
    w\blue{i}\orange{i_1}\cdots\orange{i_m},
\]
injectivity proves \eqref{eq:Psi_gamma_recursion}.
\end{proof}

\begin{remark}\label{rem:ODE}
An alternative characterization of the signature \eqref{def:signature}
is through the tensor-algebra-valued linear differential equation
$\dd S_t=S_t\otimes\dd Y_t$ with $S_0=1$. For
$\mathbf{a},\mathbf{b}\in T((\mathbb{R}^d))$, we define the concatenation product in
$\Psi$-coordinates by transporting
\begin{equation}\label{eq:gamma_tensor_product}
    \mathbf{a}\odot\mathbf{b}
    :=
    \Psi
    \left(
        \Psi^{-1}(\mathbf{a})
        \otimes
        \Psi^{-1}(\mathbf{b})
    \right).
\end{equation}
By construction, $\Psi$ is an algebra isomorphism from
$\bigl(T((\mathbb{R}^d)),\otimes\bigr)$ to
$\bigl(T((\mathbb{R}^d)),\odot\bigr)$. Setting
$\Gamma_t=\Psi(\Sig{Y}_{0,t})$, we therefore have
\[
\Psi\Bigl(1+\int_0^t S_u\otimes\dd Y_u\Bigr)
=
1+\int_0^t\Psi(S_u\otimes\dd Y_u)
=
1+\int_0^t
\Psi\bigl(\Psi^{-1}(\Gamma_u)\otimes\Psi^{-1}(\dd Y_u)\bigr)
=
1+\int_0^t\Gamma_u\odot\dd Y_u,
\]
where in the last identity we used \eqref{eq:gamma_tensor_product} and
$\Psi^{-1}|_{T^1(\mathbb{R}^d)}=\mathrm{Id}$. Consequently, the
$\Psi$-coordinates are equivalently characterized as the unique solution of
\[
\Gamma_0=1,
\qquad
\dd\Gamma_t=\Gamma_t\odot\dd Y_t,
\qquad
0<t\leq T.
\]

The coordinate transformation also transports the corresponding Lie
structure. Writing
\[
    G_{\Psi}
    :=
    \Psi
    \left(
        G((\mathbb{R}^d))
    \right),
    \qquad
    \mathfrak{g}_{\Psi}
    :=
    \Psi
    \left(
        \operatorname{Lie}((\mathbb{R}^d))
    \right),
\]
the product $\odot$ makes $G_{\Psi}$ a group, and $\Psi^{-1}$ restricts to
a Lie algebra isomorphism
\[
    \Psi^{-1}:
    \left(
        \mathfrak{g}_{\Psi},
        [\cdot,\cdot]_{\odot}
    \right)
    \longrightarrow
    \left(
        \operatorname{Lie}((\mathbb{R}^d)),
        [\cdot,\cdot]_{\otimes}
    \right),
\]
where $\operatorname{Lie}((\mathbb{R}^d))$ denotes the space of Lie
series and the commutator brackets are given by
\(
    [\mathbf{a},\mathbf{b}]_{\star}
    :=
    \mathbf{a}\star\mathbf{b}
    -
    \mathbf{b}\star\mathbf{a}
\).
In particular,
\[
    [\mathbf{a},\mathbf{b}]_{\odot}
    =
    \Psi
    \left(
        [\Psi^{-1}(\mathbf{a}),\Psi^{-1}(\mathbf{b})]_{\otimes}
    \right).
\]
Consequently, $\dd\Gamma_t=\Gamma_t\odot\dd Y_t$ is a left-invariant
equation on $G_{\Psi}$.
\end{remark}

\subsection{Partial symmetrization}
\label{sec:partial_sym}
Denote, as in \eqref{eq:path_components}, by $(A,X)$ the split of
$Y\in C^{1-var}([s,t],\mathbb{R}^{d})$ into its
$\mathcal{A}^{\prime}$ and $\mathcal{A}^{\prime\prime}$ coordinates,
respectively. 
In the case $|\mathcal{A}^{\prime\prime}| = 1$, the $\Psi$-coordinates
are particularly simple, as the $\Sig{X}$ terms in \eqref{def:gamma_signature}
simply turn into normalized moments, and we obtain
\begin{equation}\label{eq:gamma_signature_sym_univariate}
        \langle w, \Sig{Y} \rangle_{\Psi} = 
        \frac{1}{|w_k|!}(X_{s,t})^{|w_k|}\int_{\Delta^{k}_{s,t}}\prod_{j=1}^{k}  
        \frac{1}{|w_{j-1}|!}(X_{s,t_j})^{|w_{j-1}|} \dd A^{\blue{i_{j}}}_{t_{j}}.
\end{equation}
for $w\in\mathcal{W}_{\mathcal{A}}$ with $
    \decomp{\mathcal{A}^{\prime}}(w)
    =
    \left(
        (w_0,\ldots,w_k),
        \blue{i_1}\cdots\blue{i_k}
    \right)
    $.
This is particularly convenient when $A$ is deterministic and we aim to compute expectations (cf. \Cref{sec:expected_signature}).
For the general case $|\mathcal{A}^{\prime\prime}|>1$, it is therefore tempting to symmetrize in \eqref{def:gamma_signature} over all coordinates in $\mathcal{A}^{\prime\prime}$, thus obtaining multivariate monomials instead of $\Sig{X}$-terms.

\Cref{thm:partial_symmetrized_gamma} below shows that the algebraic structure resulting from this partial symmetrization can be expressed very explicitly when the signature is written in $\Psi$-coordinates. Indeed, a general Hopf-algebraic argument implies that the partially symmetrized signature satisfies a reduced ``shuffle''-type relation. However, when expressed in standard tensor coordinates, this relation becomes algebraically cumbersome. In $\Psi$-coordinates, the resulting dual product is a simple shuffle over $\mathcal{A}^{\prime\prime}$-blocks.
Furthermore, when the augmenting path $A$ has at least one strictly monotone component, we retain injectivity as a map of $X$, and universal approximation by linear functionals follows directly.

To this end, recall from \Cref{sec:prelim_quo} the quotient maps
\begin{align*}
    [\cdot]
    =
    [\cdot]_{\mathcal{A}^{\prime\prime}}&:
    \mathbb{R}\langle\mathcal{A}\rangle
    \rightarrow
    \mathbb{R}_{\mathcal{A}^{\prime\prime}}\langle\mathcal{A}\rangle,
    \\
    \symmetrizer
    =
    \symmetrizer_{\ \mathcal{A}^{\prime\prime}}&:
    T((\mathbb{R}^d))
    \rightarrow
    T_{\mathcal{A}^{\prime\prime}}((\mathbb{R}^d)),
\end{align*}
obtained by quotienting by the commutators of letters in
$\mathcal{A}^{\prime\prime}$ and taking the corresponding quotient on the tensor
side. For $w=\orange{i_1}\cdots\orange{i_m}\in
\mathcal{W}_{\mathcal{A}^{\prime\prime}}$, we introduce the notation
\begin{equation*}
    X_{s,t}^{[w]}
    :=
    X_{s,t}^{\orange{i_1}}\cdots X_{s,t}^{\orange{i_m}}
    =
    \prod_{\orange{i}\in\mathcal{A}^{\prime\prime}}
    \left(X_{s,t}^{\orange{i}}\right)^{\alpha_{\orange{i}}},
    \qquad
    [w]!
    :=
    \prod_{\orange{i}\in\mathcal{A}^{\prime\prime}}
    \alpha_{\orange{i}}!,
\end{equation*}
where
$\alpha_{\orange{i}}:=\#\{r\in\{1,\ldots,m\}:\orange{i_r}=\orange{i}\}$.
Both definitions depend only on the quotient class $[w]$. For
$v,w\in\mathcal{W}_{\mathcal{A}^{\prime\prime}}$, we further use the
additive notation $[v]+[w]:=[vw]$, which corresponds to addition of the
associated multiplicity multi-indices. In particular,
$$X_{s,t}^{[w]}X_{s,t}^{[v]}=X_{s,t}^{[w]+[v]}.$$ Further, we write
\begin{equation*}
    \Sigsym{\cdot}
    =
    \widehat{\Sig{\cdot}}.
\end{equation*}
For $w\in\mathcal{W}_{\mathcal{A}^{\prime\prime}}$, the shuffle identity yields
\begin{equation}\label{eq:quotient_signature_moment}
    \big\langle [w],\Sigsym{X}_{s,t}\big\rangle
    =
    \frac{1}{[w]!}X_{s,t}^{[w]}.
\end{equation}
While \eqref{eq:quotient_signature_moment} gives a simple expression for the partially symmetrized signature of $X$, it is not immediate how to obtain an analogous representation for the partially symmetrized full signature of $Y$ in terms of its components $(A,X)$. In contrast, the structure of the $\Psi$-coordinates directly yields a simple expression after applying the quotient.
The following theorem expresses the partially symmetrized
$\Psi$-coordinates and their algebraic properties explicitly.

\begin{theorem}
\label{thm:partial_symmetrized_gamma}
Let $Y=(A,X)\in C^{1-var}([s,t],\mathbb{R}^d)$ as above. For any
$w\in\mathcal{W}_{\mathcal{A}}$ with
\[
    \decomp{\mathcal{A}^{\prime}}(w)
    =
    \left(
        (w_0,\ldots,w_k),
        \blue{i_1}\cdots\blue{i_k}
    \right),
    \qquad
    \alpha_j:=[w_j],
    \quad
    j=0,\ldots,k,
\]
it holds that
\begin{equation}\label{def:gamma_signature_sym}
    \big\langle [w],\Sigsym{Y}\big\rangle_\Psi
    =
    \frac{1}{\alpha_k!}X_{s,t}^{\alpha_k}
    \int_{\Delta^{k}_{s,t}}
    \prod_{j=1}^{k}
    \frac{1}{\alpha_{j-1}!}X_{s,t_j}^{\alpha_{j-1}}
    \dd A^{\blue{i_j}}_{t_j}.
\end{equation}
Moreover, for all $w,v\in\mathcal{W}_{\mathcal{A}}$,
\begin{equation}\label{eq:partial_symmetrized_gamma_shuffle}
    \big\langle [w],\Sigsym{Y}\big\rangle_\Psi
    \big\langle [v],\Sigsym{Y}\big\rangle_\Psi
    =
    \big\langle
        [w]\gammashufflesym[v],
        \Sigsym{Y}
    \big\rangle_\Psi,
\end{equation}
where $\gammashufflesym$ is the bilinear product recursively determined, for
$w,v\in\mathcal{W}_{\mathcal{A}}$, terminal block classes
$\alpha=[r]$ and $\beta=[q]$ with
$r,q\in\mathcal{W}_{\mathcal{A}^{\prime\prime}}$, and
$\blue{i},\blue{j}\in\mathcal{A}^{\prime}$, by
\begin{equation}\label{eq:partial_symmetrized_gamma_recursion}
\left\{
\begin{aligned}
\alpha\gammashufflesym\beta
&:=
\frac{(\alpha+\beta)!}{\alpha!\beta!}
(\alpha+\beta)\Big.,
\\
\bigl([w]\blue{i}\alpha\bigr)\gammashufflesym\beta
&:=
[w]\blue{i}
\bigl(
    \alpha\gammashufflesym\beta
\bigr)\Big.,
\\
\alpha\gammashufflesym\bigl([v]\blue{j}\beta\bigr)
&:=
[v]\blue{j}
\bigl(
    \alpha\gammashufflesym\beta
\bigr)\Big.,
\\
\bigl([w]\blue{i}\alpha\bigr)
\gammashufflesym
\bigl([v]\blue{j}\beta\bigr)
&:=
\bigl(
    [w\blue{i}]\gammashufflesym[v]
\bigr)
\blue{j}
\bigl(
    \alpha\gammashufflesym\beta
\bigr)
+
\bigl(
    [w]\gammashufflesym[v\blue{j}]
\bigr)
\blue{i}
\bigl(
    \alpha\gammashufflesym\beta
\bigr)\Big..
\end{aligned}
\right.
\end{equation}
Suppose now $\widetilde{A} \in C^{1-var}([s,t],\R^m)$ satisfies Assumption~\ref{ass:mambo} with respect to some fixed sequence $\mathbf{w}=(w_n)_{n\in \mathbb{N}}$, and recall its fibre $[\widetilde{A}]_\mathbf{w}$ introduced in \eqref{eq:fibre}. Then the map
\begin{equation*}
    \Psi \circ \wh{\mathrm{Sig}}:
    \left\{
        Y=(A,X)\in C^{1-var}([s,t],\mathbb{R}^d)
        \;\vert\; Y_s = Y_0,\, A \in [\widetilde{A}]_\mathbf{w}
    \right\}
    \,\to\,
    T_{\mathcal{A}^{\prime\prime}}((\mathbb{R}^d))
\end{equation*}
is injective for some fixed initial condition $X_0$.
\end{theorem}
\begin{proof}
Recall from \Cref{sec:prelim_quo} that the graded adjoint
$\symmetrizer^\ast:
    \mathbb{R}_{\mathcal{A}^{\prime\prime}}
    \langle\mathcal{A}\rangle
    \longrightarrow
    \mathbb{R}\langle\mathcal{A}\rangle$
is given on quotient words by
\[
    \widehat{[w]}^\ast
    :=
    \symmetrizer^\ast([w])
    =
    \sum_{\substack{p\in\mathcal{W}_{\mathcal{A}}\\ [p]=[w]}}
    p,
\]
where the words $p$ satisfying $[p]=[w]$ are obtained by independently
permuting the letters in each of the
$\mathcal{A}^{\prime\prime}$-blocks $w_0,\ldots,w_k$.
By \eqref{eq:quotient_pairing} and the definition of $\wh{\mathrm{Sig}}$,
\begin{equation*}
    \left\langle [w],\Sigsym{Y}\right\rangle_\Psi
    =
    \left\langle
        \widehat{[w]}^\ast,
        \Sig{Y}
    \right\rangle_\Psi.
\end{equation*}
Applying \eqref{eq:quotient_signature_moment} independently to these
blocks gives \eqref{def:gamma_signature_sym}.
We next show that the recursion
\eqref{eq:partial_symmetrized_gamma_recursion} is the partial
symmetrization of the recursion for $\gammashuffle$.
For terminal block
classes $\alpha$ and $\beta$, one has
\begin{equation}\label{eq:partial_symmetrized_orbit_shuffle}
    \widehat{\alpha}^{\ast}
    \shuffle
    \widehat{\beta}^{\ast}
    =
    \frac{(\alpha+\beta)!}{\alpha!\beta!}
    \widehat{\alpha+\beta}^{\ast}.
\end{equation}
Indeed, each representative of $\alpha+\beta$ occurs with multiplicity
$(\alpha+\beta)!/(\alpha!\beta!)$ on the left-hand side. Moreover,
\begin{equation}\label{eq:symmetrizer_block_concatenation}
    \widehat{[w]\blue{i}\alpha}^{\ast}
    =
    \widehat{[w]}^{\ast}
    \blue{i}
    \widehat{\alpha}^{\ast}.
\end{equation}

Using \eqref{eq:partial_symmetrized_orbit_shuffle},
\eqref{eq:symmetrizer_block_concatenation}, and
\Cref{prop:gamma_shuffle_recursion}, induction on the total number of
letters from $\mathcal{A}^{\prime}$ in the two factors gives
\begin{equation}\label{eq:partial_symmetrized_gamma_descent}
    \symmetrizer^\ast
    \bigl(
        [w]\gammashufflesym[v]
    \bigr)
    =
    \widehat{[w]}^{\ast}
    \gammashuffle
    \widehat{[v]}^{\ast}.
\end{equation}
Indeed, the pure-block case is
\eqref{eq:partial_symmetrized_orbit_shuffle}. The one-sided cases follow
from
\begin{align*}
    \widehat{[w]\blue{i}\alpha}^{\ast}
    \gammashuffle
    \widehat{\beta}^{\ast}
    &=
    \widehat{[w]}^{\ast}
    \blue{i}
    \bigl(
        \widehat{\alpha}^{\ast}
        \shuffle
        \widehat{\beta}^{\ast}
    \bigr),
    \\
    \widehat{\alpha}^{\ast}
    \gammashuffle
    \widehat{[v]\blue{j}\beta}^{\ast}
    &=
    \widehat{[v]}^{\ast}
    \blue{j}
    \bigl(
        \widehat{\alpha}^{\ast}
        \shuffle
        \widehat{\beta}^{\ast}
    \bigr),
\end{align*}
while, if both factors contain a letter from
$\mathcal{A}^{\prime}$,
\begin{align*}
    \widehat{[w]\blue{i}\alpha}^{\ast}
    \gammashuffle
    \widehat{[v]\blue{j}\beta}^{\ast}
    &=
    \bigl(
        \widehat{[w\blue{i}]}^{\ast}
        \gammashuffle
        \widehat{[v]}^{\ast}
    \bigr)
    \blue{j}
    \bigl(
        \widehat{\alpha}^{\ast}
        \shuffle
        \widehat{\beta}^{\ast}
    \bigr) +
    \bigl(
        \widehat{[w]}^{\ast}
        \gammashuffle
        \widehat{[v\blue{j}]}^{\ast}
    \bigr)
    \blue{i}
    \bigl(
        \widehat{\alpha}^{\ast}
        \shuffle
        \widehat{\beta}^{\ast}
    \bigr).
\end{align*}
The induction hypothesis and
\eqref{eq:partial_symmetrized_orbit_shuffle} identify these expressions
with the images under $\symmetrizer^\ast$ of the corresponding
right-hand sides in
\eqref{eq:partial_symmetrized_gamma_recursion}. This proves
\eqref{eq:partial_symmetrized_gamma_descent}.
Using the character identity for $\gammashuffle$ in \eqref{eq:char_gamma} we finally show \eqref{eq:partial_symmetrized_gamma_shuffle}as follows
\begin{align*}
    \big\langle [w],\Sigsym{Y}\big\rangle_\Psi
    \big\langle [v],\Sigsym{Y}\big\rangle_\Psi
    &=
    \big\langle
        \widehat{[w]}^{\ast},
        \Sig{Y}
    \big\rangle_\Psi
    \big\langle
        \widehat{[v]}^{\ast},
        \Sig{Y}
    \big\rangle_\Psi
    \\
    &=
    \big\langle
        \widehat{[w]}^{\ast}
        \gammashuffle
        \widehat{[v]}^{\ast},
        \Sig{Y}
    \big\rangle_\Psi
    \\
    &=
    \big\langle
        \symmetrizer^\ast
        \bigl(
            [w]\gammashufflesym[v]
        \bigr),
        \Sig{Y}
    \big\rangle_\Psi
    \\
    &=
    \big\langle
        [w]\gammashufflesym[v],
        \Sigsym{Y}
    \big\rangle_\Psi.
\end{align*}

The proof of the
injectivity statement is analogous to the moment argument for time-extended
signatures in \cite[Lemma~2.6]{cuchiero2023signaturemodels}.

For the injectivity assertion, we first the following: For any integrable $f$ and smooth $g$ we have \begin{equation}\label{eq:cauchy_formula}
    \int_{\Delta^k_{s,t}}f(t_1)\dd g_{t_1}^{i_1}\cdots \dd g_{t_k}^{i_k} = -\int_s^t f(u)\partial_u \langle i_1\cdots i_k, \Sig{g}_{u,t}\rangle \dd u.
    \end{equation}
    Indeed, applying Fubini we can write \begin{align*}
    \int_{\Delta^k_{s,t}}f(t_1)\dd g_{t_1}^{i_1}\cdots \dd g_{t_k}^{i_k} & = \int_{[s,t]^k}1_{\Delta^{k}_{s,t}}(t_1,\dots,t_k)f(t_1)\prod_{l=1}^k\dot{g}^{ik}_{t_l}\dd t_l \\ & = \int_s^tf(t_1)\left \{\int_{t_1}^t\cdots \int_{t_1}^{t_2} \dd g^{i_2}_{t_2}\cdots \dd g_{t_k}^{i_k} \right \} \dot{g}_{t_1}^{i_1}\dd t_1 \\ & = \int_s^t f(u) \langle i_2\cdots i_k,\Sig{g}_{u,t} \rangle \dot{g}_{u}^{i_1}\dd u.
    \end{align*} It follows directly by definition that $\partial_u \langle i_1\cdots i_k,\Sig{g}_{u,t} \rangle = - \langle i_2\cdots i_k,\Sig{g}_{u,t} \rangle \dot{g}_{u}^{i_1}$, which shows formula  \eqref{eq:cauchy_formula}. Now let 
$Y^1=(A^1,X^1)$ and $Y^2=(A^2,X^2)$ with 
$Y^1_s = Y^2_s$ and $A^1,A^2 \in [\widetilde{A}]_\mathbf{w}$. Suppose that
$
    \Sigsym{Y^1}
    =
    \Sigsym{Y^2}.
$
For every $\blue{i}\in\mathcal{A}^{\prime}$ and $w_n =w_n^1\cdots w_n^{j_n}$ from the sequence $\mathbf{w}=(w_n)_{n\in \mathbb{N}}$, we have
\begin{equation*}
    0 =\left\langle
        [\blue{i}\blue{w_n}],
        \Sigsym{Y^1}
    \right\rangle_\Psi
    -
    \left\langle
        [\blue{i}\blue{w_n}],
        \Sigsym{Y^2}
    \right\rangle_\Psi
    =
    \int_s^t
    A_{r_1}^{1,\blue{i}}\dd A^{1,\blue{w^1_n}}_{s,r_1}\cdots \dd A^{1,\blue{w^{j_n}_n}}_{r_{j_n}}-\int_s^t
    A_{r_1}^{2,\blue{i}}\dd A^{2,\blue{w^1_n}}_{s,r_1}\cdots \dd A^{2,\blue{w^{j_n}_n}}_{r_{j_n}}
\end{equation*}
Now applying \eqref{eq:cauchy_formula}, and using that $\Sig{A^1}|_{\mathbf{w}}=\Sig{A^2}|_{\mathbf{w}}$ since $A^1,A^2 \in [\widetilde{A}]_\mathbf{w}$ by assumption, we have \begin{equation}\label{eq:mj_conlcu}
    0 = \int_s^t (A_{s,u}^{1,\blue{i}}-A_{s,u}^{2,\blue{i}}) \partial_u \langle w_n,\Sig{\widetilde{A}}_{u,t}\rangle \dd u, \qquad \forall \blue{i} \in \mathcal{A}', \quad \forall n\in\mathbb{N}.
\end{equation}

Therefore, by Assumption~\ref{ass:mambo}, the function $f(u)=A^{1,\blue{i}}_{s,u}-A^{2,\blue{i}}_{s,u}$ lies in the orthogonal complement of a total familiy of functions in $L^2([s,t],\mathbb{R})$, which implies $f=0$ almost everywhere. Since $A_s^{1}=A_s^2$, this implies $A^1=A^2$ almost everywhere. Exactly the same argument but $[\orange{i}\blue{w_n}]$ with $\orange{i}\in \mathcal{A}''$ leads to \eqref{eq:mj_conlcu} with respect to the integrand $X_{s,u}^{1,\orange{i}}-X_{s,u}^{2,\orange{i}}$.
Thus $Y^1=Y^2$ which finishes the proof. 
\end{proof}

We now relate the quotient pairing and the product
$\gammashufflesym$ introduced above to the standard Hopf-algebraic
structure of the partially symmetrized tensor algebra. To this end, we
first show that the partial symmetrization kernel is a Hopf ideal and is
preserved by the $\Psi$ and its inverse.

On the dual side, partial symmetrization is encoded by the ideal
$\mathcal{J}_{\mathcal{A}^{\prime\prime}}$ from
\eqref{def:partial_commutative_ideal}. On the tensor side, the corresponding
kernel is
\[
    \mathcal{I}_{\mathcal{A}^{\prime\prime}}
    :=
    \ker(\symmetrizer)
    =
    \prod_{n\geq 0}
    \operatorname{span}
    \left\{
        e_{u\orange{i}\orange{j}v}
        -
        e_{u\orange{j}\orange{i}v}
        \;:\;
        u,v\in\mathcal{W}_{\mathcal{A}},
        \ |u|+|v|=n-2,\ 
        \orange{i},\orange{j}\in\mathcal{A}^{\prime\prime}
    \right\},
\]
with the convention that the span is zero for $n<2$. The following lemma
records the two quotient-compatibility properties needed below: this kernel is
a Hopf ideal for the standard signature Hopf algebra, and it is preserved by
the $\Psi$-maps.

\begin{lemma}\label{lem:partial_symmetrization_ideal}
The ideal
$\mathcal{I}_{\mathcal{A}^{\prime\prime}}=\ker(\symmetrizer)$ is a Hopf ideal
for
$
(
T((\mathbb{R}^d)),
\otimes,
\Delta_{\shuffle}
)
$.
Moreover,
\[
    \Psi^{-1}\big(
        \mathcal{I}_{\mathcal{A}^{\prime\prime}}
    \big)
    =
    \mathcal{I}_{\mathcal{A}^{\prime\prime}},
    \qquad
    \Psi\big(
        \mathcal{I}_{\mathcal{A}^{\prime\prime}}
    \big)
    =
    \mathcal{I}_{\mathcal{A}^{\prime\prime}}.
\]
\end{lemma}
\begin{proof}
By definition, $\mathcal{I}_{\mathcal{A}^{\prime\prime}}$ is a two-sided
ideal with respect to $\otimes$. We show that it is also compatible with the
coalgebra structure, counit, and antipode.
For every $\orange{i},\orange{j}\in\mathcal{A}^{\prime\prime}$, one has
\begin{equation*}
    \Delta_{\shuffle}\left(
        e_{\orange{i}}e_{\orange{j}}
        -
        e_{\orange{j}}e_{\orange{i}}
    \right)
    =
    \left(
        e_{\orange{i}}e_{\orange{j}}
        -
        e_{\orange{j}}e_{\orange{i}}
    \right)\boxtimes 1
    +
    1\boxtimes
    \left(
        e_{\orange{i}}e_{\orange{j}}
        -
        e_{\orange{j}}e_{\orange{i}}
    \right).
\end{equation*}
Since $\Delta_{\shuffle}$ is multiplicative with respect to $\otimes$, it
follows that
$$\Delta_{\shuffle}(\mathcal{I}_{\mathcal{A}^{\prime\prime}})
\subseteq
\mathcal{I}_{\mathcal{A}^{\prime\prime}}\boxtimes T((\mathbb{R}^d))
+
T((\mathbb{R}^d))\boxtimes\mathcal{I}_{\mathcal{A}^{\prime\prime}}.$$
Moreover,
$\varepsilon(e_{\orange{i}}e_{\orange{j}}-e_{\orange{j}}e_{\orange{i}})=0$,
and since $\varepsilon$ is multiplicative, $\varepsilon$ vanishes on all
elements of $\mathcal{I}_{\mathcal{A}^{\prime\prime}}$. Finally,
\[
    S\left(
        e_{\orange{i}}e_{\orange{j}}
        -
        e_{\orange{j}}e_{\orange{i}}
    \right)
    =
    e_{\orange{j}}e_{\orange{i}}
    -
    e_{\orange{i}}e_{\orange{j}}
    =
    -
    \left(
        e_{\orange{i}}e_{\orange{j}}
        -
        e_{\orange{j}}e_{\orange{i}}
    \right).
\]
Since $S$ is an anti-algebra morphism, this implies
$S(\mathcal{I}_{\mathcal{A}^{\prime\prime}})
\subseteq
\mathcal{I}_{\mathcal{A}^{\prime\prime}}$.
Thus $\mathcal{I}_{\mathcal{A}^{\prime\prime}}$ is a Hopf ideal.

It remains to prove invariance under the the $\Psi$ and $\Psi^{-1}$. By the definition of
$\symmetrizer^\ast$ in \Cref{sec:prelim_quo},
\[
    \mathcal{I}_{\mathcal{A}^{\prime\prime}}
    =
    \bigl(\operatorname{im}\symmetrizer^\ast\bigr)^\perp.
\]
It is therefore enough to show that $\Psi^\ast$ preserves
$\operatorname{im}\symmetrizer^\ast$.
We prove this by induction on the number of letters from
$\mathcal{A}^{\prime}$. If $w\in\mathcal{W}_{\mathcal{A}^{\prime\prime}}$,
then $\Psi_{u}=u$ for every representative $u$ of $[w]$, and therefore
\[
    \Psi^\ast\bigl(\widehat{[w]}^\ast\bigr)
    =
    \widehat{[w]}^\ast.
\]
For the induction step, let
$\decomp{\mathcal{A}^{\prime}}(w)
=
((w_0,\ldots,w_k),\blue{i_1}\cdots\blue{i_k})$
with $k\geq 1$. Summing the recursion for $\Psi$ over all representatives
of $[w]$ gives
\begin{equation*}
    \Psi^\ast\bigl(\widehat{[w]}^\ast\bigr)
    =
    \left(
        \Psi^\ast
        \bigl(
            \symmetrizer^\ast
            \bigl(
                [w_0]\blue{i_1}[w_1]\cdots
                \blue{i_{k-1}}[w_{k-1}]
            \bigr)
        \bigr)
        \blue{i_k}
    \right)
    \shuffle
    \widehat{[w_k]}^\ast.
\end{equation*}
By the induction hypothesis, the first factor lies in
$\operatorname{im}\symmetrizer^\ast$. Appending the letter
$\blue{i_k}\in\mathcal{A}^{\prime}$ preserves this image. The ordinary
shuffle recursion shows directly that
$\operatorname{im}\symmetrizer^\ast$ is closed under $\shuffle$. Hence
$\Psi^\ast(\operatorname{im}\symmetrizer^\ast)
\subseteq
\operatorname{im}\symmetrizer^\ast$.

Taking annihilators gives
$\Psi(\mathcal{I}_{\mathcal{A}^{\prime\prime}})
\subseteq
\mathcal{I}_{\mathcal{A}^{\prime\prime}}$.
Since $\Psi$ is a graded automorphism, this inclusion is an equality on
each homogeneous level. Thus
$\Psi(\mathcal{I}_{\mathcal{A}^{\prime\prime}})
=
\mathcal{I}_{\mathcal{A}^{\prime\prime}}$, and applying $\Psi^{-1}$ gives
$\Psi^{-1}(\mathcal{I}_{\mathcal{A}^{\prime\prime}})
=
\mathcal{I}_{\mathcal{A}^{\prime\prime}}$.
\end{proof}

The Hopf-ideal property in
\Cref{lem:partial_symmetrization_ideal} allows us to form the standard
partially symmetrized signature Hopf algebra by quotient. The invariance of
$\mathcal{I}_{\mathcal{A}^{\prime\prime}}$ under $\Psi$ and $\Psi^{-1}$
also gives well-defined descended maps on this quotient. We now record these
formal consequences and identify the quotient pairing introduced above with
the pairing induced by the descended $\Psi$-map.

\begin{proposition}\label{prop:partial_symmetrized_hopf}
There exists a graded Hopf algebra structure
$
    \left(
        T_{\mathcal{A}^{\prime\prime}}((\mathbb{R}^d)),
        \widehat{\otimes},
        \Delta_{\widehat{\shuffle}}
    \right)
$
for which
\begin{equation*}
    \symmetrizer:
    \left(
        T((\mathbb{R}^d)),
        \otimes,
        \Delta_{\shuffle}
    \right)
    \longrightarrow
    \left(
        T_{\mathcal{A}^{\prime\prime}}((\mathbb{R}^d)),
        \widehat{\otimes},
        \Delta_{\widehat{\shuffle}}
    \right)
\end{equation*}
is a Hopf algebra morphism. In particular,
\begin{equation}\label{eq:partial_symmetrized_signature_coproduct}
    \Delta_{\widehat{\shuffle}}\circ\symmetrizer
    =
    \left(
        \symmetrizer\boxtimes\symmetrizer
    \right)
    \circ\Delta_{\shuffle}.
\end{equation}
Denoting by
\begin{equation*}
    \widehat{\shuffle}:
    \mathbb{R}_{\mathcal{A}^{\prime\prime}}\langle\mathcal{A}\rangle
    \times
    \mathbb{R}_{\mathcal{A}^{\prime\prime}}\langle\mathcal{A}\rangle
    \longrightarrow
    \mathbb{R}_{\mathcal{A}^{\prime\prime}}\langle\mathcal{A}\rangle
\end{equation*}
the corresponding graded dual product, $\Sigsym{Y}$ is group-like and,
equivalently, for all $u,v\in\mathcal{W}_{\mathcal{A}}$,
\begin{equation}\label{eq:partial_symmetrized_signature_shuffle}
    \big\langle [u],\Sigsym{Y}\big\rangle
    \big\langle [v],\Sigsym{Y}\big\rangle
    =
    \big\langle
        [u]\widehat{\shuffle}[v],
        \Sigsym{Y}
    \big\rangle.
\end{equation}
Moreover, $\Psi$ and $\Psi^{-1}$ induce inverse graded automorphisms
$\widehat{\Psi}$ and $\widehat{\Psi}^{-1}$ on
$T_{\mathcal{A}^{\prime\prime}}((\mathbb{R}^d))$, characterized by
\begin{equation}\label{eq:quotient_transform_commutative_diagram}
    \symmetrizer\circ\Psi
    =
    \widehat{\Psi}\circ\symmetrizer,
    \qquad
    \symmetrizer\circ\Psi^{-1}
    =
    \widehat{\Psi}^{-1}\circ\symmetrizer.
\end{equation}
The quotient pairing from \eqref{eq:quotient_pairing} is the pairing induced
by $\widehat{\Psi}$; that is,
\begin{equation}\label{eq:quotient_induced_pairing}
    \big\langle
        \alpha,
        \widehat{\mathbf{a}}
    \big\rangle_\Psi
    =
    \big\langle
        \widehat{\Psi}^\ast(\alpha),
        \widehat{\mathbf{a}}
    \big\rangle
\end{equation}
for all
$\alpha\in
\mathbb{R}_{\mathcal{A}^{\prime\prime}}\langle\mathcal{A}\rangle$
and
$\widehat{\mathbf{a}}\in
T_{\mathcal{A}^{\prime\prime}}((\mathbb{R}^d))$.
\end{proposition}

\begin{proof}
By \Cref{lem:partial_symmetrization_ideal},
$\mathcal{I}_{\mathcal{A}^{\prime\prime}}=\ker(\symmetrizer)$ is a Hopf
ideal for
$
(
T((\mathbb{R}^d)),
\otimes,
\Delta_{\shuffle}
)
$.
The quotient construction for Hopf ideals in
\cite[Theorem~4.3.1]{sweedler1969hopf} therefore yields the stated graded Hopf
algebra structure on
$T_{\mathcal{A}^{\prime\prime}}((\mathbb{R}^d))$, and makes
$\symmetrizer$ a Hopf algebra morphism. This gives
\eqref{eq:partial_symmetrized_signature_coproduct}.

Since $\Sig{Y}$ is group-like with respect to $\Delta_{\shuffle}$, we obtain
\[
    \Delta_{\widehat{\shuffle}}\bigl(\Sigsym{Y}\bigr)
    =
    \Delta_{\widehat{\shuffle}}
    \bigl(
        \symmetrizer(\Sig{Y})
    \bigr)
    =
    \left(
        \symmetrizer\boxtimes\symmetrizer
    \right)
    \Delta_{\shuffle}\bigl(\Sig{Y}\bigr)
    =
    \Sigsym{Y}\boxtimes\Sigsym{Y}.
\]
By graded duality, this is equivalent to
\eqref{eq:partial_symmetrized_signature_shuffle}.
It remains to record the descent of the $\Psi$-maps. Again by
\Cref{lem:partial_symmetrization_ideal}, both $\Psi$ and $\Psi^{-1}$ preserve
$\mathcal{I}_{\mathcal{A}^{\prime\prime}}$. Hence the assignments
\[
    \widehat{\Psi}
    \bigl(
        \widehat{\mathbf{a}}
    \bigr)
    :=
    \widehat{\Psi(\mathbf{a})},
    \qquad
    \widehat{\Psi}^{-1}
    \bigl(
        \widehat{\mathbf{a}}
    \bigr)
    :=
    \widehat{\Psi^{-1}(\mathbf{a})}
\]
are well defined. They are inverse graded automorphisms because
$\Psi$ and $\Psi^{-1}$ are inverse graded automorphisms. Their defining
property is exactly
\eqref{eq:quotient_transform_commutative_diagram}. Taking graded adjoints
therein gives
\begin{equation}\label{eq:psi_hat_ast_form}
    \Psi^\ast\circ\symmetrizer^\ast
=
\symmetrizer^\ast\circ\widehat{\Psi}^\ast, \qquad (\Psi^{-1})^\ast\circ\symmetrizer^\ast
=
\symmetrizer^\ast\circ(\widehat{\Psi}^{-1})^\ast.
\end{equation}
Finally, let
$\widehat{\mathbf{a}}=\symmetrizer(\mathbf{a})$. Using
\eqref{eq:quotient_pairing}, \eqref{eq:psi_hat_ast_form} and the definition of
the induced pairing, \eqref{eq:quotient_induced_pairing} follows by
\[
    \left\langle \alpha,\widehat{\mathbf{a}}\right\rangle_\Psi
    =
    \left\langle \widehat{\alpha}^\ast,\mathbf{a}\right\rangle_\Psi
    =
    \left\langle \Psi^\ast\bigl(\widehat{\alpha}^\ast\bigr),\mathbf{a}\right\rangle
    =
    \left\langle \symmetrizer^\ast\bigl(\widehat{\Psi}^\ast(\alpha)\bigr),\mathbf{a}\right\rangle
    =
    \left\langle \widehat{\Psi}^\ast(\alpha),\widehat{\mathbf{a}}\right\rangle.
\]
\end{proof}

We finally identify the product $\gammashufflesym$ from
\Cref{thm:partial_symmetrized_gamma} with the product dual to the standard
quotient coproduct with respect to the induced pairing.

\begin{proposition}\label{prop:quotient_gamma_dual_product}
For all
$\alpha,\beta\in
\mathbb{R}_{\mathcal{A}^{\prime\prime}}\langle\mathcal{A}\rangle$,
\begin{equation}\label{eq:quotient_gamma_dual_product_explicit}
    \alpha\gammashufflesym\beta
    =
    \bigl(
        \widehat{\Psi}^{-1}
    \bigr)^\ast
    \left(
        \widehat{\Psi}^\ast(\alpha)
        \widehat{\shuffle}
        \widehat{\Psi}^\ast(\beta)
    \right).
\end{equation}
Consequently, for all
$\widehat{\mathbf{a}}\in
T_{\mathcal{A}^{\prime\prime}}((\mathbb{R}^d))$
\begin{equation}\label{eq:quotient_gamma_coproduct_duality}
    \left\langle
        \alpha\gammashufflesym\beta,
        \widehat{\mathbf{a}}
    \right\rangle_\Psi
    =
    \left\langle
        \alpha\boxtimes\beta,
        \Delta_{\widehat{\shuffle}}
        \bigl(
            \widehat{\mathbf{a}}
        \bigr)
    \right\rangle_{\Psi\boxtimes\Psi}.
\end{equation}
\end{proposition}

\begin{proof}
Using \eqref{eq:psi_hat_ast_form} and the graded dual of
\eqref{eq:partial_symmetrized_signature_coproduct}, we obtain
\begin{align*}
    \symmetrizer^\ast
    \left(
        \bigl(
            \widehat{\Psi}^{-1}
        \bigr)^\ast
        \left(
            \widehat{\Psi}^\ast(\alpha)
            \widehat{\shuffle}
            \widehat{\Psi}^\ast(\beta)
        \right)
    \right)
    &=
    (\Psi^{-1})^\ast
    \left(
        \symmetrizer^\ast
        \left(
            \widehat{\Psi}^\ast(\alpha)
        \right)
        \shuffle
        \symmetrizer^\ast
        \left(
            \widehat{\Psi}^\ast(\beta)
        \right)
    \right)
    \\
    &=
    (\Psi^{-1})^\ast
    \left(
        \Psi^\ast
        \bigl(
            \widehat{\alpha}^\ast
        \bigr)
        \shuffle
        \Psi^\ast
        \bigl(
            \widehat{\beta}^\ast
        \bigr)
    \right)
    \\
    &=
    \widehat{\alpha}^\ast
    \gammashuffle
    \widehat{\beta}^\ast
    \\
    &=
    \symmetrizer^\ast
    \bigl(
        \alpha\gammashufflesym\beta
    \bigr).
\end{align*}
where the last equality follows from
\eqref{eq:partial_symmetrized_gamma_descent}. Since
$\symmetrizer^\ast$ is injective, this proves
\eqref{eq:quotient_gamma_dual_product_explicit}.
Using \eqref{eq:quotient_induced_pairing},
\eqref{eq:quotient_gamma_dual_product_explicit}, and the graded duality
between $\widehat{\shuffle}$ and
$\Delta_{\widehat{\shuffle}}$, we obtain \eqref{eq:quotient_gamma_coproduct_duality} from
\begin{align*}
    \left\langle
        \alpha\gammashufflesym\beta,
        \widehat{\mathbf{a}}
    \right\rangle_\Psi
    &=
    \left\langle
        \widehat{\Psi}^\ast
        \bigl(
            \alpha\gammashufflesym\beta
        \bigr),
        \widehat{\mathbf{a}}
    \right\rangle
    \\
    &=
    \left\langle
        \widehat{\Psi}^\ast(\alpha)
        \widehat{\shuffle}
        \widehat{\Psi}^\ast(\beta),
        \widehat{\mathbf{a}}
    \right\rangle
    \\
    &=
    \left\langle
        \widehat{\Psi}^\ast(\alpha)
        \boxtimes
        \widehat{\Psi}^\ast(\beta),
        \Delta_{\widehat{\shuffle}}
        \bigl(
            \widehat{\mathbf{a}}
        \bigr)
    \right\rangle
    \\
    &=
    \left\langle
        \alpha\boxtimes\beta,
        \Delta_{\widehat{\shuffle}}
        \bigl(
            \widehat{\mathbf{a}}
        \bigr)
    \right\rangle_{\Psi\boxtimes\Psi}.
\end{align*}
\end{proof}

\subsection{Change of coordinates for partial rough paths and related works}\label{sec:sigtransform_roughpath}
So far we have considered pairs $Y=(A,X)$ of finite-variation paths, and represented both the canonical lift $\Sig{Y}$ and the partially symmetrized version $\wh{\mathrm{Sig}}(Y)$ in $\Psi$-coordinates, see \eqref{def:gamma_signature} and \eqref{def:gamma_signature_sym}. In this section, we extend these considerations to \emph{partially irregular paths}: the component $A$ remains regular enough to be integrated against, while we allow $X\in C^{p\text{-var}}([0,T];\mathbb R^d)$; see \eqref{def:p-var}. The advantage of the representations \eqref{def:gamma_signature} and \eqref{def:gamma_signature_sym} is that they make explicit which information is required from the irregular component. For the non-symmetrized object \eqref{def:gamma_signature}, we require the partial signature of $X$. To deal with the first case, we use rough path signature lifts introduced in Section~\ref{sec:paths_signatures}. In particular, we will assume that we are given a geometric $p$-rough path $\mathbf X\in \mathscr C_g^{p-var,0}([0,T];\mathbb R^d)$, such that $\pi_1(\mathbf X)=X$, with unique signature lift denoted by $\Sig{\RP{X}}$. For the symmetrized object \eqref{def:gamma_signature_sym}, the situation is simpler. Here the partial signature of $X$ is only used through its symmetrization, which is given by symmetric tensor powers of the increments of $X$. Hence these coordinates are already determined by the continuous path $X$ itself and do not require a rough path enhancement. In both cases, the remaining operations are integrations against $\dd A$, which are well-defined under the regularity assumptions imposed on $A$.

Under sufficient regularity of $A$, one can canonically lift the pair $(A,\RP{X})$ to a joint rough path $\RP{Y}\in \mathscr{C}_g^{p\text{-var}}([0,T];\R^{m+d})$, and thus canonically define the signature $\Sig{A,\RP{X}}$. This construction is well known and can for instance be found in \cite[Chapter 9.4]{friz2010multidimensional}. We now show how the same construction appears naturally through the lens of the $\Psi$-coordinates.

\begin{corollary}\label{cor:rough_gamma}
    Let $p>2$ and let $q\geq 1$ be such that $1/p+1/q>1$. For any pair $(A,\RP{X})\in C^{q\text{-var}}([0,T];\R^m)\times \mathscr{C}^{p\text{-var},0}_g([0,T];\R^d)$, we define now the $T((\mathbb{R}^d))$-valued path $\{\mathbf{S}(Y)_{s,t}:(s,t)\in \Delta_{0,T}^2\}$ in $\Psi$-coordinates  \begin{equation}\label{def:gamma_sig_rough}
        \langle \emptyset,\mathbf{S}(Y)_{s,t}\rangle_\Psi :=1, \quad \langle w, \mathbf{S}(Y)_{s,t}\rangle_\Psi := \langle w_k, \Sig{\RP{X}}_{s,t}\rangle\int_{\Delta^{k}_{s,t}}\prod_{j=1}^{k}  \langle w_{j-1}, \Sig{\RP{X}}_{s,t_j}\rangle \dd A^{\blue{i_{j}}}_{t_{j}},
    \end{equation}
    for all $w\in \mathcal{W}_{\mathcal{A}}$ with $\Delta_{\mathcal A'}(w)=\{(w_0,\dots,w_k),\blue{i_1}\cdots\blue{i_k}\}$, where the integrals are understood in the Young sense. Then $\pi_{\leq \lfloor p\rfloor}(\mathbf{S}(Y))= \mathbf{Y} \in \mathscr{C}_g^{p-var}([0,T];\R^{m+d})$ with $\pi_1(\mathbf{Y})=(A,X)$, and the full signature lift is determined by  $\langle w,\Sig{\RP{Y}}_{s,t}\rangle =\langle \Psi_w^{-1},\mathbf{S}(Y)_{s,t}\rangle_\Psi$.
\end{corollary}
\begin{proof}
First note that  $\mathbf{S}(Y)$ is well-defined. Indeed, by Lyons' extension theorem, the full signature $\Sig{\RP{X}}$ is canonically determined by the rough path $\RP{X}$, and for every word $v \in \mathcal{W}_{\mathcal{A}''}$, the path $r\mapsto \langle v,\Sig{\RP{X}}_{s,r}\rangle$ has finite $p$-variation. Since $A$ has finite $q$-variation and $1/p+1/q>1$, the integrals appearing in the definition of $\mathbf{S}(Y)$ are well-defined as Young integrals. Now using the fact that there exists a joint $p$-rough path $\mathbf{Y}$ above $(A,\mathbf{X})$, given by \cite[Theorem 9.28]{friz2010multidimensional}, we can find a sequence of finite-variation paths $(A^n,X^n) \rightarrow (A,X)$, such that $\Sig{A^n,X^n}$ converges in the $p$-rough path topology. On the other-hand, by definition \eqref{def:gamma_sig_rough}, continuity of Young integration and the fact that $\pi_{\leq \lfloor p \rfloor }(\Sig{X^n}) \rightarrow \mathbf{X}$, we have coordinate-wise convergence of $\mathbf{S}(Y^n)$ to $\mathbf{S}(Y)$. Combining these observations with Proposition~\ref{prop:Psi}, we have $$\mathbf{Y}= \lim_{n\rightarrow \infty}\pi_{\leq \lfloor p\rfloor}(\Sig{A^n,X^n})= \lim_{n\rightarrow \infty}\pi_{\leq \lfloor p\rfloor}(\mathbf{S}(Y^n))´)=\pi_{\leq \lfloor p\rfloor}(\mathbf{S}(Y)),$$ and by Lyons' extension we can do the same for any $N>\lfloor p \rfloor$, so that in particular $\mathbf{S}(Y)=\Sig{\mathbf{Y}}$.
\end{proof}

In applications, the irregular component $(X_t)_{t\in[0,T]}$ can be interpreted as the data or underlying signal, while $A$ is typically a deterministic and regular feature, such as time $A_t=t$, which ensures expressiveness of the signature, resp. the symmetrized version thereof. Another interesting class of examples arises when $A$ is a (semi-)martingale. In this case, the complementary Young regularity required in Corollary~\ref{cor:rough_gamma} can never be fulfilled, since both $p,q>2$, and the cross-terms $\int \mathbf{X} \dd A$ are no longer canonically defined. Nevertheless, by using stochastic calculus for the $\dd A$-integrals, it is still possible to construct a joint stochastic rough path $\mathbf{Y}$ above $(A,\mathbf{X})$, and thus by Lyons' extension, the signature $\mathrm{Sig}(\mathbf{Y})$. This is, for instance, of interest when studying mixed rough and stochastic differential equations (or RSDEs) \cite{friz2021rough}, i.e.
$$
Z_0=\xi \in \mathbb{R}^m, \qquad  
\dd Z_t=b_t(Z_t)\dd \mathbf{X}_t+\sigma_t(Z_t)\dd A_t, 
\qquad 0<t \leq T,
$$
which have found many applications. Constructing the joint-lift has been carried out for Brownian motion $A$ already in \cite{diehl2015levy,diehl2017stochastic}, see also \cite[Chapter 12.2]{friz2020course}, and has since been developed in many further works. For a general treatment, we refer to the recent work \cite{friz2020rough}. In the following remark, we illustrate how such lifts can be constructed naturally using our machinery. A fully rigorous treatment is outside the scope of this article.

\begin{remark}\label{rem:partial_stochastic_lift}
Suppose $(A_t)_{t\in [0,T]}$ is a semimartingale adapted to some complete filtered probability space. Given a deterministic rough path $\mathbf{X} \in \mathscr{C}_g^{p\text{-var}}$, one may still define a stochastic version of \eqref{def:gamma_sig_rough}, where the integrals against $\dd A$ are understood as stochastic integrals. For example, if $2<p\leq 3$ and $\mathbf{X}=(X,\mathbb{X})$, the first level is still $\mathbf{S}(Y)^{(1)}_{s,t}=(A_{s,t},X_{s,t})$, and on the second level define
\begin{equation}\label{eq:expl_rough_gamma}
  \mathbf{S}(Y)^{(2)}_{s,t}:=\begin{pmatrix}
    \int_{s<u<v<t} \circ \dd A_u^{\blue{i}}\circ \dd A_v^{\blue{j}} 
    & 
    \langle \orange{i},\mathbf{X}_{s,t}\rangle A^{\blue{j}}_{s,t} 
    \\
    \int_{s<u<t} \langle\orange{i},\mathbf{X}_{s,u}\rangle \circ \dd A_u^{\blue{j}} 
    & 
    \langle \orange{ij},\mathbf{X}_{s,t}\rangle
\end{pmatrix}
\in \mathbb{R}^{m+d} \otimes \mathbb{R}^{m+d},
\end{equation}
with the obvious block interpretation and where $\circ \dd A$ denotes Stratonovich integration. It follows directly from the definition of $\Psi^{-1}$ that $\mathbf{Y}_{s,t}:=\Psi^{-1}(1,\mathbf{S}(Y)^{(1)}_{s,t},\mathbf{S}(Y)^{(2)}_{s,t})$ is group-like, using the geometricity of the partial rough path $\mathbf{X}$ and of the Stratonovich lift of $A$, and that Chen's relation holds. The remaining point is to verify the required sample-path regularity, which is typically done by means of a suitable Kolmogorov criterion for rough paths; see, for example, \cite[Theorem 3.1]{friz2020course}. Similarly, for rough paths $\mathbf{X}$ of lower regularity, say $p>3$, one may define the corresponding higher-order stochastic $\Psi$-coordinates and then transform them via $\Psi^{-1}$ to obtain a stochastic joint rough path lift.
\end{remark}
Another interesting example, which arises in the context of rough volatility modeling \cite{bayer2023rough}, is given by joint stochastic dynamics $Y=(S,V)$, where one may think of $S$ as the price process and $V$ as the volatility process of an asset. Here $S$ is typically a semimartingale, while $V$ may be neither a semimartingale nor a Markov process, for instance when $V$ is driven by a fractional Brownian motion correlated with $S$. While one may still lift the $V$-component to a stochastic geometric rough path $\mathbf{V}$, the joint-lift construction discussed above becomes more delicate. For example, let $A$ be a standard Brownian motion and let $V_t=\int_0^t(t-s)^{H-1/2}\dd W_s$ with $H<1/2$ and $\langle W,A\rangle_t=\rho t$ for some $\rho\neq0$. Then the Stratonovich integral corresponding to the word $w=\orange{i}\blue{j}$ in \eqref{eq:expl_rough_gamma} is no longer well-defined, since the corresponding quadratic covariation correction explodes. A joint lift construction in this setting has recently been presented in \cite{bonesini2024rough}, and we can also illustrate their construction through our machinery.

\begin{remark}\label{rem:joint_lifts_rvol}
    The important observation for the $2<p\leq 3$ example in \eqref{eq:expl_rough_gamma}, and similarly for $p>3$, is that the choice of integration for the components $\orange{i}\blue{j}$ is not relevant for the algebraic group-like property; the latter is taken care of by the transform $\Psi^{-1} $. Indeed, recall that by definition of $\Psi^{-1}$ in \eqref{eq:implicit_psi} we have
    $$
    \Psi^{-1}_{\orange{i}\blue{j}} = \orange{i}\blue{j}, \quad \Psi^{-1}_{\blue{j}\orange{i}}= \blue{j}\orange{i}-\orange{i}\blue{j}, \qquad   \langle \orange{i}\shuffle \blue{j},\Psi^{-1}(\mathbf{x})\rangle = \langle \Psi^{-1}_{\orange{i}\blue{j}}+\Psi^{-1}_{\blue{j}\orange{i}},\mathbf{x}\rangle =\langle \blue{j}\orange{i},\mathbf{x}\rangle.
    $$
    Now for $\mathbf{S}(Y)^{\leq 2}=(1,Y,\mathbb{Y}^{(2)})$ in \eqref{eq:expl_rough_gamma}, we have $\langle \blue{j}\orange{i}, \mathbf{S}(Y)^{\leq 2}\rangle = \langle \blue{j},\mathbf{S}(Y)^{\leq 2}\rangle \langle \orange{i},\mathbf{S}(Y)^{\leq 2}\rangle$. On the other hand, for the diagonal terms $\orange{i}\orange{j}$ and $\blue{i}\blue{j}$ the shuffle identity holds by geometricity of the individual lifts. In summary, $\Psi^{-1} \circ \mathbf{S}$ is group-like independently of how the component $\orange{i}\blue{j}$ is defined. Thus, in view of the ill-posed Stratonovich integral, we may define this component via Itô integration. The resulting transform is still a geometric rough path, provided the Chen relations and the required sample-path regularity continue to hold. The same idea applies to more irregular paths, that is, to $p>3$, where the construction has to be carried out slightly more carefully. In this case, in \eqref{def:gamma_sig_rough} one uses $\circ \dd A^{\blue{i_j}}$ whenever $w_{j-1}=\varnothing$ and $\dd A^{\blue{i_j}}$ otherwise. Using again the definition of $\Psi^{-1}$ in \eqref{eq:implicit_psi}, one can show that the resulting object necessarily lies in the group; see also \cite[Theorem 2.6]{bonesini2024rough}.
\end{remark}

Finally, we end this section with a connection of the symmetrized $\Psi$-transform to so-called partial rough path spaces introduced in \cite{fukasawa2024partial}, also in the context of rough volatility models.
\begin{remark}\label{rem:partial_rough_paths}
In Remarks~\ref{rem:partial_stochastic_lift} and~\ref{rem:joint_lifts_rvol},  we illustrated how, given a rough path $\mathbf{X}$ together with a possibly correlated semimartingale $A$, one can construct a full joint rough path lift for $(A,X)$. In the context of rough volatility, $\mathbf{X}$ represents a rough path lift of the volatility process $\sqrt{V}$, whereas $A$ represents the correlated driver of the (log-)price dynamics. Constructing a rough path lift of the volatility process may, however, be delicate: the process can have very low regularity, and a canonical enhancement may not be available.

To circumvent this difficulty, the authors of~\cite{fukasawa2024partial} introduce a partial rough path space that does not require a lift above $V$, but nevertheless contains sufficient information to construct and expand the relevant equations. This construction is closely related to the symmetrized $\Psi$-coordinates, which we illustrate now. Written in notation closer to ours, a partial rough path in the sense of~\cite[Definition~2.1]{fukasawa2024partial} is a triplet $\mathbf{Y}^{\text{par}}=\bigl(X,(A^{(i)})_i,(\mathbb A^{(j,k)})_{j,k}\bigr)$
where $X$ is of $\beta$-Hölder regularity (the volatility process, $\beta$ typically small), while $A^{(0)}=A$ is $\alpha$-Hölder with $\alpha \in (1/3,1/2)$ (the price dynamics, typically Brownian regularity). The remaining coordinates are supposed to postulate the values 
\[
A^{(i)}_{s,t}
=
\frac{1}{i!}
\int_s^t X_{s,r}^{i}\dd A_r,
\qquad
\mathbb A_{s,t}^{(j,k)}
=
\frac{1}{j!k!}
\int_s^t\int_s^u
X_{s,r}^{j}X_{s,u}^{k}
\dd A_r\otimes\dd A_u, \qquad i,j,k \in \mathbb{N}^d
\]
where the required multi-index length depends on the path-regularities, i.e.  $|i|\beta+\alpha \leq 1$ and $|j+k|\beta + 2\alpha \leq 1$.
As in the previous discussions in this section, if $A$ is a
semimartingale, then all the components of $\mathbf{Y}^{\mathrm{par}}$
are well defined by means of stochastic integration. In view of the symmetrized $\Psi$-coordinates in~\eqref{def:gamma_signature_sym}, we can observe that
\[
\mathbf{Y}^{par}
=
\left(
\left(
    \left\langle
        \orange{k},
        \Sigsym{A,X}
    \right\rangle_\Psi
\right)_{\orange{k}\in\mathcal{A}''},
\left(
    \left\langle
        [\orange{w}]\blue{i},
        \Sigsym{A,X}
    \right\rangle_\Psi
\right)_{
    \substack{
        \orange{w}\in\mathcal{W}_{\mathcal{A}''},\\
        \blue{i}\in\mathcal{A}'
    }
},
\left(
    \left\langle
        [\orange{w}]\blue{i}
        [\orange{v}]\blue{j},
        \Sigsym{A,X}
    \right\rangle_\Psi
\right)_{
    \substack{
        \orange{w},\orange{v}\in
        \mathcal{W}_{\mathcal{A}''},\\
        \blue{i},\blue{j}\in\mathcal{A}'
    }
}
\right),
\]
subject to the same regularity-dependent length restrictions on the words
$\orange{w}$ and $\orange{v}$ as seen for the multi-indices in \cite[Definition 2.1]{fukasawa2024partial}. Thus, $\mathbf{Y}^{par}$ may be viewed as a coordinate
projection of the truncated $\wh{\Psi}$-coordinates,
retaining precisely the components required to construct the relevant
rough path lifts in rough volatility. Conversely, adding the
remaining coordinates of the truncated $\wh{\Psi}$-coordinates yields
an algebraic completion of the partial rough paths introduced
in~\cite{fukasawa2024partial}: the enlarged collection is closed under
the product, see the preceding
sections.

\end{remark}


\section{Expected signature transforms and moment-problems}\label{sec:expected_signature}
In this section we leverage the signature representations in $\Psi$- and $\wh{\Psi}$-coordinates from in Section \ref{sec:sig_transform}, to derive formulas for the expected signatures  above stochastic processes. More precisely, suppose $(\mathbf{X}_t)_{t\in [0,T]}$ is a random $p$-rough path, adapted to some complete filtered probability space $(\Omega,\mathcal{F},\mathbb{F}=(\mathcal{F}_t)_{t\in [0,T]},\mathbb{P})$. Given a deterministic path $A \in C^{q-var}([0,T],\mathbb{R}^d)$ with $\frac1q+\frac1p>1$, the aim of the section is to characterize (whenever well-defined) \begin{equation}\label{eq:expected_signature} \bmu_{s,t}:= \mathbb{E}[\Sig{A,\RP{X}}_{s,t}] \quad \text{and} \quad \widehat{\bmu}_{s,t}:=\mathbb{E}[\widehat{\mathrm{Sig}}(A,X)_{s,t}], \qquad  (s,t) \in \Delta_{0,T}^{2}.
\end{equation} 
The typical example relevant to have in mind is a one-dimensional augmentation $A:[0,T]\rightarrow \mathbb{R}$, such as $A_t=t$, and semi-martingale or Gaussian drivers $X$ with canonical rough path lifts $\mathbf{X}$. The following result is an immediate consequence of linearity of the expectation and the transformation $\Psi$ from Section~\ref{sec:sig_transform}.

\begin{proposition}\label{prop:exp_sig_prop}
    Whenever well-defined, we have  \begin{equation}\label{eq:expected_gamma}
        \langle w,\bmu_{s,t}\rangle_\Psi =\int_{\Delta^k_{s,t}} \mathcal{C}^w_{s,t}(t_1,\dots,t_k) \dd A^{\blue{i_1}}_{t_1}\cdots \dd A^{\blue{i_k}}_{t_k} \quad \text{and} \quad \langle [w],\wh \bmu_{s,t}\rangle_\Psi =\int_{\Delta^k_{s,t}} \wh{\mathcal{C}}^{[w]}_{s,t}(t_1,\dots,t_k) \dd A^{\blue{i_1}}_{t_1}\cdots \dd A^{\blue{i_k}}_{t_k}
    \end{equation}
    where $((w_0, \dots, w_{k}), \blue{j_1 \cdots j_k}) = \decomp{\mathcal{A}^\prime}(w)$, and the correlators $\mathcal{C},\wh{\mathcal{C}}$ are defined by \begin{equation}\label{eq:correlators_prop} \mathcal{C}^w_{s,t}(\mathbf{t})=\mathbb{E}\Big[  \prod_{j=1}^{k+1} \langle w_{j-1},\Sig{\RP{X}}_{s,t_j}\rangle \Big ], \qquad \wh{ \mathcal{C}}^{[w]}_{s,t}(\mathbf{t})=\mathbb{E}\Big [\prod_{j=1}^{k+1}\frac{1}{[w_{j-1}]!}X_{s,t_j}^{[w_{j-1}]} \Big ] \qquad t_{k+1}:=t,
    \end{equation} where we used the notation $\mathbf{t}=(t_1,\dots,t_k)$.
\end{proposition}
The main advantage of this representation for the expected-signature transform
is that the problems of taking expectations and integrating against the
deterministic components \(A\) are disentangled. In the sequel we focus on
the symmetrized correlators \(\wh{\mathcal C}\) (resp. $|\mathcal{A}''|=1$). As we demonstrate in the
examples below, these deterministic correlators are often available in closed
form and exhibit substantially higher regularity than the underlying noise
\(X\). One can take advantage of this smoothing effect by using efficient
numerical integration techniques to compute the deterministic integrals in
\eqref{eq:expected_gamma}; see also
Section~\ref{sec:numerics} below. 

\begin{example}\label{ex:Gaussian_isserlis}
    Suppose $X$ is a  centered Gaussian processes with covariance function given by $R(s,t)=\mathbb{E}[X_s\otimes X_t]$, and $X_0=0$. In this case, an application of the Isserlis (or Wick) formula to the
    symmetrized correlators in \eqref{eq:correlators_prop} gives \[
    \wh{\mathcal{C}}^w_{0,T}(\mathbf{t})=\begin{cases}
        \prod_{r=1}^{{k+1}}\frac{1}{|w_{r-1}|!}\sum_{\sigma \in \mathcal{P}_2^w}\prod_{(a,b)=(w_{i}^{n},w_{j}^{m})\in \sigma}R^{a,b}(t_n,t_m), & |w_0|+\cdots + |w_k| \in 2\mathbb{N} \\ 0, & \text{else}
    \end{cases},
    \] where we recall the convention $t_{k+1}=t$, and  $\mathcal{P}_2^w$ denotes the set of all distinct ways to partition $\{w_j^l:1\leq j \leq k, \, 1 \leq l \leq |w_j| \}$ into pairs $(w_i^m,w_j^m)$. Thus, in this case computing the symmetrized expected
    signature reduces to the evaluation of deterministic integrals of the form $$\int_{\Delta_{0,T}^{k}}\prod_{(a,b) \in \sigma }R^{a,b}(t_{a(i)},t_{a(j)}) \dd A^{\blue{i_1}}_{t_1}\cdots \dd A^{\blue{i_k}}_{t_k}, \qquad t_{k+1}:=t.$$
\end{example}

\begin{example}\label{ex:polynomial} Another class of interest is given by polynomial diffusions
    \cite{cuchiero2012polynomial,filipovic2016polynomial}, that is,
    solutions to SDEs of the form $$\dd X_t=b(X_t)\dd t + \sigma(X_t) \dd B_t, \qquad b \in \mathrm{Pol}_1, \quad \sigma\sigma^\top \mathrm{Pol}_2.$$ as well as various generalizations thereof
    \cite{filipovic2020polynomial,cuchiero2019probability,
    cuchiero2021infinite,cuchiero2023signature,abi2024polynomial,
    cuchiero2025polynomial}. Such processes are characterized by
    semi-explicit moment transforms. For the diffusion case considered here,
    \cite{benth2021correlators} presents closed-form formulas for correlators
    of the type $$\wh{\mathcal{C}}^w_{0,T}(\mathbf{t})
        =
        e_k^\top A(X_0)e^{\tilde{G}^{(m)}T}
        \prod_{l=1}^{k}
        e^{\tilde{G}^{(k-l)}(t_l-t_{l-1})}
        \{I_l\otimes e_{k-l}\}$$ where $A$ and $\tilde{G}^{l}$ can be constructed explicitly from the
    generator of the diffusion; see \cite[Theorem 4.5]{benth2021correlators}
    for details. In particular, in view of our expected-signature transform,
    this leaves us with the evaluation of deterministic integrals of the form $$\int_{\Delta_{0,T}^k}\prod_{l=1}^{k+1}e^{\tilde{G}^{(k-l)}(t_{l}-t_{l-1})}\dd A^{\blue{i_1}}_{t_1}\cdots \dd A^{\blue{i_k}}_{t_k}.$$

\end{example}
The two examples above illustrate that, after applying the expected-signature
transform, the remaining computational task is the evaluation of deterministic
simplex integrals of the form $\int_{\Delta_{0,T}^k}
    \Xi(T,t_1,\ldots,t_k)
    \dd A^{\blue{i_1}}_{t_1}\cdots
    \dd A^{\blue{i_k}}_{t_k}$ for some given integrand $\Xi$. In many relevant cases, these correlators are
substantially more regular than the sample paths of $X$ itself, which opens
the door to accurate numerical integration techniques, such as quadrature.

\subsection{Moment determinacy for smooth components}
The expected signature is of particular importance in applications, since in many situations it characterizes the law of the underlying stochastic process. When this \emph{moment-determinacy} property holds, the expected signature provides a natural and computationally tractable way of comparing probability measures on path space. However, establishing that the expected signature is indeed law-determining is a highly non-trivial problem. One route is through the duality between universality and characteristicness; see, for instance, \cite[Theorem~7]{chevyrev2018signature}. In concrete applications, universality is often established by Stone-Weierstrass type arguments, which naturally yield results on compact subsets of path space. Global law-determinacy of the expected signature was studied in \cite{chevyrev2016characteristic}, where sufficient conditions are given in terms of an infinite radius of convergence of the expected signature. This condition is satisfied, for example, by several classes of Gaussian rough paths \cite[Example~6.7]{chevyrev2016characteristic}.

The goal of this section is to show that, for sufficiently rich deterministic
augmentations $A$ (in the sense of Assumption~\ref{ass:mambo}), the expected signature $\boldsymbol\mu$, and likewise its
partially symmetrized version $\widehat{\boldsymbol\mu}$, characterize the law
of the underlying path $X$. The argument relies on  classical moment-determinance of the one-dimensional marginals of $X$. More precisely, writing $\mu_{X_t^i}:=\operatorname{Law}(X_t^i)$,  we assume that, for every $t\in[0,T]$ and $i\in\{1,\dots,d\}$, the measure
$\mu_{X_t^i}$ is moment-determinate, in the sense that
\begin{equation}\label{eq:moment-deter-real}
	\left[
	\nu\in\mathcal P(\mathbb R),\quad
	\int_{\mathbb R} x^m\,\nu(\dd x)
	=
	\int_{\mathbb R} x^m\,\mu_{X_t^i}(\dd x)
	\ \text{ for all } m\geq 0
	\right]
	\quad\Longrightarrow\quad
	\nu=\mu_{X_t^i}.
\end{equation}
Here $\mathcal P(\mathbb R)$ denotes the set of probability measures on
$\mathbb R$. By Petersen's theorem \cite[Theorem~3]{petersen1982relation},
assumption \eqref{eq:moment-deter-real} implies that, for every
$0\leq t_1\leq\cdots\leq t_p\leq T$, $p\geq1$, the finite-dimensional law
of $(X_{t_1},\dots,X_{t_p})$ is uniquely determined by its multivariate moments. Consequently, the collection
of correlators $\widehat{\mathcal C}$ defined in
\eqref{eq:correlators_prop}, when known for all time-tuples and words, determines all finite-dimensional distributions of $X$.

As a first result, we show that the symmetrized expected signature representation
from Proposition~\ref{prop:exp_sig_prop}, together with the previous observations,
yields the desired characteristicness. To this end, fix a deterministic
augmentation $A$. For a probability measure
$\mathbb P\in\mathcal P(C_o([0,T];\mathbb R^d))$, where $C_o$ denotes the space of continuous paths starting at some fixed $X_0=x_o$, we denote by
$X$ the canonical process on $C_o([0,T];\mathbb R^d)$ and write
$\widehat{\boldsymbol\mu}^{\mathbb P}:=\widehat{\boldsymbol\mu}^{\mathbb P}(A)$ for the corresponding symmetrized
expected signature, whenever it is
well-defined. More precisely, we define
\begin{equation}\label{eq:path_measures_gamma}
	\mathcal P_A
	:=
	\left\{
	\mathbb P\in\mathcal P(C_o([0,T];\mathbb R^d)) :
	\left|
	\langle w, \wh{\bmu}^\mathbb{P}_{0,T}(A) \rangle 
	\right|<\infty
	\text{ for all words } w
	\right\}.
\end{equation}
We can now state the first main result of this section.
\begin{proposition}\label{prop:moment_determinisme}
    Consider a deterministic path $A=(A_t)_{t\in [0,T]}$ satisfying Assumption~\ref{ass:mambo} and let $\mathbb{P}\in \mathcal{P}_A$, such that the one-dimensional marginals of the canonical process satisfy \eqref{eq:moment-deter-real}. Then \begin{equation}\label{eq:moment_determ}
        \wh{\bmu}^\mathbb{P}_{0,T}(A) =\wh{\bmu}^\mathbb{Q}_{0,T}(A) \quad \Leftrightarrow \quad \mathbb{P}=\mathbb{Q}, \qquad \forall \,  \mathbb{Q} \in \mathcal{P}_A.
    \end{equation}
\end{proposition}
\begin{proof}
    The direction $\Leftarrow$ is straightfoward. In order to prove $\Rightarrow$, from the discussion previous to the statement, we know that it suffices to show that the correlators $\wh{\mathcal{C}}^\mathbb{P}$ and $\wh{\mathcal{C}}^\mathbb{Q}$ coincide, so that we can concude using \eqref{eq:moment-deter-real} and  \cite[Theorem 3]{petersen1982relation}. By Proposition~\ref{prop:exp_sig_prop} we also know $\langle [w],\wh{\bmu}_{0,T}^\mathbb{Q}\rangle_w=\langle [w],\wh{\bmu}_{0,T}^\mathbb{P}\rangle_\Psi$ for all words $w\in \mathcal{W}$. Next we claim that for any word of the form $w=  w_1 \blue{j^1_1\cdots j^1_{k_1}}\cdots w_{n} \blue{j^n_1\cdots j^n_{k_n}}$ with $w_l\in \mathcal{W}_{\mathcal{A}''}$ with $w_l \neq \varnothing$ and $j^l_{j_l} \in \mathcal{A}'$, we have \begin{equation}\label{eq:mj_identity}
    	\langle [w],\wh{\bmu}_{0,T}\rangle_\Psi = (-1)^n\int_{\Delta_{0,T}^n} \wh{\mathcal{C}}^{[w]}_{0,T}(t_1,\dots,t_n)\prod_{l=1}^n \partial_{t_l}\langle \blue{j^l_{1}\cdots j^l_{k_l}}, \Sig{A}_{t_l,t_{l+1}}\rangle \dd t_{1}\cdots \dd t_{n}, \qquad t_{n+1}:=T.\end{equation}
     Similar to the injectivity proof in Theorem~\ref{thm:partial_symmetrized_gamma}, we apply \eqref{eq:cauchy_formula} but repeatedly to the $\Psi$-coordinate representation in  \eqref{eq:expected_gamma}, which shows \eqref{eq:mj_identity}. By assumption, we in particular find \begin{equation}\label{eq:orthogonality}
    0 = (-1)^n\int_{\Delta_{0,T}^n} \Big \{ \wh{\mathcal{C}}_\mathbb{P}^{[w]}(t_1,\dots,t_n)-\wh{\mathcal{C}}_\mathbb{Q}^{[w]}(t_1,\dots,t_n) \Big \}\prod_{l=1}^n \partial_{t_l}\langle \blue{v_l}, \Sig{A}_{t_l,t_{l+1}}\rangle \dd t_{1}\cdots \dd t_{n},
    \end{equation} for all $w_1,\dots,w_n \in \mathcal{W}_{\mathcal{A}''}$, $\blue{v_1},\dots,\blue{v_n} \in \mathcal{W}_{\mathcal{A}'}$ and $n\in \mathbb{N}$. It remains to show that \eqref{eq:orthogonality} implies $\wh{\mathcal{C}}^{[w]}_\mathbb{P}=\wh{\mathcal{C}}^{[w]}_\mathbb{Q}$ almost everywhere for any $w_1,\dots,w_n$. We will show more generally by induction over $n\in \mathbb{N}$, it holds that   \begin{equation}\label{eq:induction_orthogonality}
   0 = \int_{\Delta_{0,T}^n} F(t_1,\dots,t_n)\prod_{l=1}^n \partial_{t_l}\langle \blue{v_l}, \Sig{A}_{t_l,t_{l+1}}\rangle \dd t_{1}\cdots \dd t_{n}, \, \forall \blue{v_l} \in \mathcal{W}_{\mathcal{A}'} \, \Longrightarrow \,  F=0 \text{ a.e.}
    \end{equation} For $n=1$, we have $$0=\int_0^TF(t)\partial_t \langle \blue{v}, \Sig{A}_{t,T}\rangle \dd t, \qquad \forall \blue{v} \in \mathcal{W}_{\mathcal{A}'}.$$  But then function $F$ in particular lies in the orthogonal complement of the set \eqref{eq:dense_familiy}, and is thus equal to zero almost everywhere by Assumption~\ref{ass:mambo}. Now assume that \eqref{eq:induction_orthogonality} holds for $n-1$, then \begin{align*}
     0 & = \int_{\Delta_{0,T}^n} F(t_1,\dots,t_n)\prod_{l=1}^n \partial_{t_l}\langle \blue{v_l}, \Sig{A}_{t_l,t_{l+1}}\rangle \dd t_{1}\cdots \dd t_{n} \\ & = \int_{\Delta_{0,T}^{n-1}} \Big \{ \int_0^{t_2}F(t_1,\dots,t_n)\partial_{t_1} \langle \blue{v_1},\Sig{A}_{t_1,t_2}\rangle \dd t_1 \Big \}\prod_{l=2}^n \partial_{t_l}\langle \blue{v_l}, \Sig{A}_{t_l,t_{l+1}}\rangle \dd t_{2}\cdots \dd t_{n}.
    \end{align*} By the induction hypothesis we have $$\int_0^{t_2}F(t_1,\dots,t_n)\partial_{t_1} \langle \blue{v_1},\Sig{A}_{t_1,t_2}\rangle \dd t_1 = 0 \text{ a.e. and for all } \blue{v_1} \in \mathcal{W}_{\mathcal{A}'},$$ which again by the $n=1$ case shows $F(\cdot,t_2,\dots,t_n)=0$ almost everywhere. This proves the claim \eqref{eq:induction_orthogonality}, and choosing $F=\wh{\mathcal{C}}^{[w]}_\mathbb{P}-\wh{\mathcal{C}}^{[w]}_\mathbb{Q}$ finishes the proof.
\end{proof}

From the moment-determinacy of the symmetrized expected signature we can easily
deduce the corresponding statement for the full expected signature. Indeed, let
$\pi_{\mathrm{Sym}}(\cdot)=\symmetrizer$ denote the partial symmetrization map
from Section~\ref{sec:prelim_quo}. Since $\widehat{\boldsymbol\mu}
=
\pi_{\mathrm{Sym}}\big(\boldsymbol\mu\big)$,
we have, whenever the involved quantities are well-defined,
\[
\langle w, \boldsymbol\mu^{\mathbb P}\rangle
=
\langle w, \boldsymbol\mu^{\mathbb Q}\rangle
 \quad \forall w\in \mathcal{W} \qquad \Longrightarrow \langle [w], \wh{\bmu}^\mathbb{P}_{0,T}\rangle_\Psi=\langle [w], \wh{\bmu}^\mathbb{Q}_{0,T}\rangle_\Psi \quad \forall w \in \mathcal{W}.
\]
Thus Proposition~\ref{prop:moment_determinisme} applies. To make this formal for rough path laws, define
\begin{equation}\label{eq:laws_RP}
	\mathscr{P}_A^p
	:=
	\left\{
	\mathbb{P}\in\mathcal{P}
	\big(\mathscr{C}^{p\text{-var}}_g([0,T];\mathbb{R}^d)\big):
	|\langle w,\boldsymbol\mu_{0,T}^{\mathbb P}(A)\rangle|<\infty
	\text{ for all words } w
	\right\}.
\end{equation}
We denote by $\pi_1{}_\#\mathbb P$ the push-forward of
$\mathbb P\in\mathscr P_A^p$ under the first-level projection $\pi_1$.
\begin{corollary}\label{cor:full_exp_sig_first_level}
	Suppose $A$ is as in Proposition~\ref{prop:moment_determinisme}. Let
	$\mathbb P\in\mathscr P_A^p$ be such that the one-dimensional marginals of
	$\pi_1{}_\#\mathbb P$ satisfy the moment-determinacy condition
	\eqref{eq:moment-deter-real}. Then, for every $\mathbb Q\in\mathscr P_A^p$,
	\[
	\boldsymbol\mu_{0,T}^{\mathbb P}(A)
	=
	\boldsymbol\mu_{0,T}^{\mathbb Q}(A)
	\quad\Longrightarrow\quad
	\pi_1{}_\#\mathbb P
	=
	\pi_1{}_\#\mathbb Q .
	\]
\end{corollary}
\begin{remark}
	If the rough path enhancement is fixed as a measurable function of the first
	level, then the preceding conclusion can be strengthened to equality of rough
	path laws. More precisely, suppose there exists a Borel measurable lift $X \mapsto \mathbf{X}=\Lambda(X)$
	such that $\pi_1\circ\Lambda=\operatorname{id}$. Then, under the assumptions of Corollary~\ref{cor:full_exp_sig_first_level}, restricted to laws of random rough paths with such a measurable lift, we have
	\[
	\boldsymbol\mu_{0,T}^{\mathbb P}(A)
	=
	\boldsymbol\mu_{0,T}^{\mathbb Q}(A)
	\quad\Longleftrightarrow\quad
	\mathbb P=\mathbb Q .
	\]
	This applies to canonical rough path lifts which are constructed as measurable
	functions of the first-level path, such as the canonical Stratonovich lift of
	continuous semimartingales and canonical Gaussian rough path lifts.
\end{remark}

\subsection{Moment determinacy for Brownian components}\label{sec:brownain_characteristicness}
The techniques developed in the previous section for proving characteristicness
under sufficiently rich deterministic augmentations can also be adapted to
stochastic augmentations. In this section we consider the case where
$A=(A_t)_{t\in[0,T]}$ is a standard Brownian motion, independent of the process
$X$. In this setting, the role of the deterministic time component is replaced
by the non-trivial expected iterated integrals of Brownian motion. The following
formula for suitable words in the $\Psi$-coordinates allows us to deduce law-determinacy by the same correlator-recovery
argument.
\begin{lemma}\label{lem:expected_sig_bnrownian}
	Let $A$ be a standard Brownian motion, and $\mathbb{P}\in \mathcal{P}_A$ independent of $A$. For any word of the form $w=w_1 \blue{ii}w_2\blue{ii}\cdots w_n\blue{ii}$ where $w_1,\dots,w_n \in \mathcal{W}_{\mathcal{A}''}$ and $\blue{i} \in \mathcal{A}'$, it holds that \begin{equation}\label{eq:brownian_gamma_formula}
		\langle [w], \wh{\bmu}^\mathbb{P}_{s,t} \rangle_\Psi =2^{-n} \int_{\Delta_{s,t}^n} \wh{\mathcal{C}}^{\mathbb{P},[w]}_{s,t}(t_1,\dots,t_n)\dd t_1 \cdots \dd t_n.
	\end{equation}
\end{lemma}
\begin{proof}
	Let $X$ bet the canonical process with respect to $\mathbb{P}$, and recall from  \eqref{def:gamma_signature_sym} with respect to Stratonovich integration \begin{equation}\langle [w], \wh{\mathrm{Sig}}(A,X) \rangle_\Psi = \frac{1}{[w_k]!}X_{s,t}^{[w_k]}
    \int_{\Delta^{k}_{s,t}}
    \prod_{j=1}^{k}
    \frac{1}{[w_{j-1}]!}X_{s,t_j}^{[w_{j-1}]}
    \circ dA^{\blue{i_j}}_{t_j}.
\end{equation}
for $w\in\mathcal{W}_{\mathcal{A}}$ with $
    \decomp{\mathcal{A}^{\prime}}(w)
    =
    \left(
        (w_0,\ldots,w_k),
        \blue{i_1}\cdots\blue{i_k}
    \right)
    $. Now for the word $w=w_1 \blue{ii}w_2\blue{ii}\cdots w_n\blue{ii}$, an application of the tower-property shows \begin{align*}
    \langle [w], \wh{\bmu}_{0,T}\rangle  & = \mathbb{E}\Big [ \mathbb{E}\Big [\int_{\Delta_{0,T}^{2n}}\prod_{j=1}^{n}
    \frac{1}{[w_{j-1}]!}X_{s,t_{2j-1}}^{[w_{j-1}]}
    \circ \dd A^{\blue{i}}_{t_{2j-1}}\circ \dd A^{\blue{i}}_{t_{2j}} \big | \mathcal{F}^A_T \Big ]\Big ]  \\ & = \mathbb{E}\Big [ \int_{\Delta_{0,T}^{2n}}\mathbb{E}\Big [\prod_{j=1}^{2n}
    \frac{1}{[w_{j-1}]!}X_{s,t_{2j-1}}^{[w_{j-1}]}\big | \mathcal{F}^A_T \Big ]
    \circ \dd A^{\blue{i}}_{t_1} \cdots \circ \dd A^{\blue{i}}_{t_{2n}} \Big ]   \\ & = \mathbb{E}\Big [ \int_{\Delta_{0,T}^{2n}}\wh{\mathcal{C}}_{0,T}^{\mathbb{P},[w]}(t_1,t_3,\dots,t_{2n-1}) \circ \dd A_{t_1}^{\blue{i}} \cdots \circ \dd A_{t_{2n}}^{\blue{i}}\Big ],
    \end{align*} where we used the assumed independence between $X$ and $A$. By definition of the Stratonovich integral, we have \[
\int_{\Delta_{0,T}^{2n}}\wh{\mathcal{C}}_{0,T}^{\mathbb{P},[w]}(t_1,\dots,t_{2n-1}) \circ \dd A_{t_1}^{\blue{i}} \cdots \circ \dd A_{t_{2n}}^{\blue{i}} = \text{MG} + \frac{1}{2}\int_{\Delta_{0,T}^{2n-1}}\wh{\mathcal{C}}_{0,T}^{\mathbb{P},[w]}(t_1,\dots,t_{2n-1}) \circ \dd A_{t_1}^{\blue{i}} \cdots \circ \dd A_{t_{2n-2}}^{\blue{i}} \dd [A^{\blue{i}}]_{t_{2n-1}}.
    \] Taking expectation and using $[A^{\blue{i}}]_t=t$, repeated $n$ times leads to  \[
    \langle w, \wh{\bmu}_{0,T}\rangle_\Psi  =\frac{1}{2^n} \int_{\Delta_{0,T}^n}\wh{\mathcal{C}}_{0,T}^{\mathbb{P},[w]}(t_1,\dots,t_n)\dd t_1\cdots \dd t_n.
    \]
\end{proof}

The lemma shows in particular that, by restricting to words in which the 
Brownian augmentation appears only through adjacent pairs $\blue{ii}$, the $\Psi$-coordinates of $\wh{\bmu}$ contain the same 
integrated correlators as in the deterministic time-augmentation case 
$A_t=t$, up to the deterministic factors $2^{-n}$. Hence the following statement follows directly from 
Proposition~\ref{prop:moment_determinisme}, respectively from 
Corollary~\ref{cor:full_exp_sig_first_level}.
\begin{proposition}\label{prop:brownian_aug_characteristic}
	Let $A$ and $\mathbb{P}$ be as in Lemma~\ref{lem:expected_sig_bnrownian} such that the one-dimensional marginals of $\mathbb{P}$ satisfy \eqref{eq:moment-deter-real}. Whenever well-defined, we have
	\[
	\widehat{\boldsymbol\mu}_{0,T}^{\mathbb P}(A)
	=
	\widehat{\boldsymbol\mu}_{0,T}^{\mathbb Q}(A)
	\quad\Longleftrightarrow\quad
	\mathbb P=\mathbb Q,
	\]
	and
	\[
	\boldsymbol\mu_{0,T}^{\mathbb P}(A)
	=
	\boldsymbol\mu_{0,T}^{\mathbb Q}(A)
	\quad\Longleftrightarrow\quad
	\mathbb P=\mathbb Q.
	\]
\end{proposition} \begin{remark}\label{rem:AMX}
We expect that the above argument can be extended to more general stochastic
augmentations. For instance, one may replace the Brownian augmentation by
suitable semimartingales, provided their iterated integrals generate sufficiently
rich deterministic testing kernels after taking expectations. The case of
non-independent augmentations is more delicate, since the factorization of the
$X$-correlators from the augmentation no longer holds directly. We leave a
systematic treatment of such extensions for future work.
\end{remark}

\section{Numerics}\label{sec:numerics}
At the beginning of Section~\ref{sec:expected_signature}, we derived formulas for
expected signatures which separate the probabilistic and deterministic parts of
the computation: expectations enter only through correlators, while the
remaining integration is performed against the augmentation. In the symmetrized setting, and in particular if 
$|\mathcal A''|=1$, the computation of the expected signature can be
factorized into:
\begin{itemize}
    \item[1.)] Compute $\Psi$-coordinates  $\langle w,\widehat{\boldsymbol\bmu}_{0,T}\rangle_\Psi
        =
        \int_{\Delta}
        \widehat{\mathcal C}^{[w]}_{0,T}(\mathbf t)\,\dd A_\mathbf t$, e.g. with higher-order quadrature methods.

    \item[2.)] Transform back into conventional expected-signature coordinates
    $\langle w, \widehat{\boldsymbol\mu}\rangle
        =
        \langle \Psi^{-1}_w,\wh{\bmu}_{0,T}\rangle_\Psi$.
\end{itemize}
We stress that the second step is only necessary when the application requires the expected signature in the standard coordinates. In many situations, see in particular in Section~\ref{sec:control}, one may work directly with the $\Psi$-coordinates, since the transform is linear and invertible. Nevertheless, the second step only requires a one-time implementation of the coordinate transformation described in Section~\ref{sec:pathwise_transform}. The corresponding code, together with all the other implementations relevant for this work, is available at
\url{https://github.com/lucapelizzari/expected_signature_transforms}.

Apart from the computational advantages of $\Psi$-coordinates illustrated in the last sections, another interesting aspect is that they naturally suggest a mixed grading with respect to subalphabets. Let $\mathcal{A}=\mathcal{A}'\cup\mathcal{A}''$ be a decomposition into two disjoint subalphabets. Recalling from Section~\ref{sec:tensor_shuffle_algebra} the length of a word $|i_1\cdots i_n|:=n$, we can similarly denote by $|\cdot|_{\mathcal{A}'}$ and $|\cdot|_{\mathcal{A}''}$ the number of letters from the corresponding subalphabets. In particular, we define the mixed-degree words
\begin{equation}\label{eq:mixed_degree_words}
    \mathcal{W}_{\mathcal{A}}^{(n_A,n_X)}
    :=
    \left\{
        w\in\mathcal{W}_{\mathcal{A}}:
        |w|_{\mathcal{A}'}=n_A,\,
        |w|_{\mathcal{A}''}=n_X
    \right\},
    \qquad
    (n_A,n_X)\in\mathbb{N}_0^2.
\end{equation}
In particular, the spaces $T((\mathbb{R}^d))$ and $\mathbb{R}\langle A\rangle$ introduced in Section~\ref{sec:tensor_shuffle_algebra} can be truncated according to this mixed grading in the natural way, namely by bounding separately the word lengths associated with the two subalphabets. Depending on the application, it may be sufficient to have a higher truncation degree in only one of the subalphabets while keeping the other one low; see, for example, Section~\ref{sec:control} below. The resulting reduction in dimension, compared with increasing both truncation levels simultaneously, makes it possible to reach substantially higher degrees.

\subsection{Expected signature of augmented fractional Brownian motions}\label{sec:numerics_fbm}
In all the numerical examples below, we will work with time-augmented fractional Brownian motion $Y_t=(t,X_t^H)$ for $t\in [0,1]$. In particular, we have $|\mathcal A'|=|\mathcal A''|=1$ and we set $A_t=t$. Here
$X^H=(X_t^H)_{t\in[0,1]}$ denotes a one-dimensional centered continuous Gaussian
process with covariance
\[
    R_H(s,t)
    :=
    \mathbb E[X_s^H X_t^H]
    =
    \frac12\left(s^{2H}+t^{2H}-|t-s|^{2H}\right),
    \qquad s,t\geq0,\quad H\in(0,1).
\]
Since both alphabets only have one letter, we simply write $\mathcal{A}'=\{\blue{\bullet}\}$ and $\mathcal{A}''=\{\orange{\bullet}\}$, so that $Y^{\blue{\bullet}}_t=t$ and $Y_t^{\orange{\bullet}}=X_t^H$. For words coming only from one alphabet, e.g. $w\in \mathcal{W}_{\mathcal{A}''}=:\mathcal{W}_{\orange{\bullet}}$, we write $w=\orange{\bullet}\cdots \orange{\bullet}=\orange{\bullet}^{k}$ if $|w|=k$. Moreover, for the mixed-degrees in \eqref{eq:mixed_degree_words} we simply write $|\cdot |_{\mathcal{A}'}=|\cdot |_{\blue{\bullet}}$ and $|\cdot |_{\mathcal{A}''}=|\cdot |_{\orange{\bullet}}$.

Since $X$ is a Gaussian process, Example~\ref{ex:Gaussian_isserlis} provides closed-form expressions for its correlators $\mathcal{C}$. To compute the expected signature $\bgamma$ in $\Psi$-coordinates in \eqref{eq:expected_gamma}, we exploit that $A_t=t$ and show in the following remark that, under a mixed truncation of degree $(n_A,n_X)$, every integral appearing in \eqref{eq:expected_gamma} can be reduced to an integral over a simplex of dimension at most $n_X$.

\begin{remark}
    Any word $w \in \mathcal{W}_{\mathcal{A}}^{(n_A,n_X)}$ can be written as $w={w_0} \blue{\bullet}{w_1} \blue{\bullet}\cdots \blue{\bullet}{w_{n_A}}$, where $w_j= \orange{\bullet}^{k_j}$ with $k_j\geq 0$ and $k_0+\cdots +k_{n_A} = n_X$. Let $J(w)= \{j \in \{1,\dots,n_A\}: w_{j-1}\neq \varnothing \}= \{j_1,\dots,j_p \}$, and define $a_l:=j_{l+1}-j_l-1$, $j_{l+1}:=p$. Clearly $p \leq n_X$, and from the general formula \eqref{eq:mj_identity}, we know \begin{align}
\left\langle w,\bmu_{0,T}\right\rangle_\Psi
&=
\int_{\Delta_{0,T}^{n_A}}
\mathcal C_{0,T}^w(t_1,\ldots,t_{n_A})
\,\dd t_1\cdots\dd t_{n_A}
\notag\\
&=
\frac{1}{\prod_{j=0}^{n_A}|w_j|!}
\int_{\Delta_{0,T}^{n_A}}
\mathbb E\left[
    X_T^{|w_{n_A}|}
    \prod_{r=1}^{p}
    X_{t_{j_r}}^{|w_{j_r-1}|}
\right]
\,\dd t_1\cdots\dd t_{n_A}
\notag\\
&=
\frac{1}{\prod_{j=0}^{n_A}|w_j|!}
\int_{\Delta_{0,T}^{p}}
\mathbb E\left[
    X_T^{|w_{n_A}|}
    \prod_{r=1}^{p}
    X_{s_r}^{|w_{j_r-1}|}
\right]
\prod_{r=0}^{p}
\frac{(s_{r+1}-s_r)^{a_r}}{a_r!}
\,\dd s_1\cdots\dd s_p,\label{eq:reduced_integrals}
\end{align} where we use the conventions $s_0=0$ and $s_{p+1}=T$. Thus, the dimension of the reduced integration domain is bounded by
$\min\{n_A,n_X\}$. Consequently, whenever either $n_A$ or $n_X$ is
small, since the expectations in \eqref{eq:reduced_integrals} are available
explicitly, the resulting low-dimensional integrals can be computed
efficiently using, for instance, Gauss--Legendre quadrature. For
higher-dimensional integrals, one may instead employ quasi-Monte Carlo
integration.
\end{remark}

\begin{example}\label{ex:n_A=2} From formula \eqref{eq:reduced_integrals}, and the fact that $X^H$ is centered and Gaussian, we have $\langle w, \bmu_{0,T}\rangle_\Psi=0$ whenever $n_X\in 2\mathbb{N}+1$. The first non-trivial process degree is therefore $n_X=2$. By \eqref{eq:reduced_integrals} and the definition of $R_H$, all
    resulting integrals are finite linear combinations of terms which,
    up to explicit multiplicative constants, are of the form  \[
 \Lambda_T(a,b)
        :=
        T^{2Hb_1}
        \int_0^T
        s^{a_0+2Hb_2}
        (T-s)^{a_1+2Hb_3}
        \dd s, \qquad \begin{cases}
    b_1,b_2,b_3 \in \{0,1\}  & b_1+b_2+b_3=1 \\
    a_0,a_1 \in \mathbb{N}_0 & a_0+a_1=n_A-1
\end{cases}.
\] After the change of variables $s=Tu$, we obtain the closed-form
    expression
    \[
        \Lambda_T(a,b)
        =
        T^{n_A+2H}
        \mathrm B\bigl(
            a_0+2Hb_2+1,
            a_1+2Hb_3+1
        \bigr),
    \]
    where $\mathrm B$ denotes the beta function $B(x,y)=\int_0^1u^{x-1}(1-u)^{y-1}\dd u$. Consequently, all expected shear coordinates of mixed degree
    $(n_A,2)$ admit closed-form expressions. A systematic investigation of closed-form or recursive representations for general mixed degrees would be interesting, but lies beyond the scope of this paper.

\end{example}

\subsection{Weak approximation error}\label{sec:numerical_integators}
In this section, we illustrate the numerical errors arising when expected signatures are computed using the conventional Monte Carlo approach, and highlight the advantages of representing them through the $\Psi$-coordinates. The conventional method consists of sampling
piecewise-linear approximations of $X^H$, computing the signature of each sample
path, for instance using standard libraries such as \emph{iisignature}, and then
averaging the resulting signature coordinates to approximate $\boldsymbol\mu$. This  procedure contains two sources of error: the Monte Carlo
sampling error and the weak discretization bias induced by the piecewise-linear
approximation. More precisely, if $\boldsymbol\mu^{\Delta t}$ denotes the
expected signature of the time-augmented piecewise-linear interpolation of
$X^H$ on a grid of mesh size $\Delta t$, then we observe numerically, and prove
in Appendix~\ref{app:weak_error}, that for every fixed truncation level $K$, we have $ \Vert
        \pi_{\leq K}\big(\boldsymbol\mu^{\Delta t}\big)
        -
        \pi_{\leq K}\big(\boldsymbol\mu\big)
    \Vert
    =
    O((\Delta t)^{2H})$.
For small Hurst parameters $H$, this rate is extremely slow, and consequently, the
number of discretization points required to reach even moderate accuracy quickly
becomes computationally infeasible.\footnote{For example, on a uniform grid with $\Delta t=1/N$ and $H=0.1$, the requirement $\Delta t^{2H}\leq 10^{-3}$ implies $N\geq 10^{15}$.}

In Figure~\ref{fig:fbm_runtime_error} we compare the conventional Monte Carlo
estimator of $\langle w,\bmu^{\Delta_t}\rangle $ described above with the correlator-based computation via $\langle \Psi^{-1}_w,\boldsymbol\mu \rangle_\Psi$
for the time-augmented fractional Brownian motion, truncated at level $K=4$. For
the Monte Carlo estimator we use $M=2^{18}$ sample paths and vary the number of
discretization points from $N=2^3$ to $N=2^{11}$; the reported runtime only includes
the signature computation and averaging step, excluding the cost of
sampling the fractional Brownian paths. For the $\Psi$-coordinates we compute the Gaussian correlators explicitly and integrate them by
deterministic quadrature, using increasing quadrature orders. Errors are measured
as the maximal absolute coordinate error with respect to a high-accuracy
quadrature reference. The plot illustrates the numerical advantage of separating
expectation and integration: while the conventional Monte Carlo estimator is
limited by the slow weak discretization rate of the piecewise-linear
approximation, especially for small Hurst parameters, the change of coordinates method reaches substantially smaller errors at smaller 
comparable runtimes.

\begin{figure}[h]
    \centering
    \includegraphics[width=0.75\textwidth]{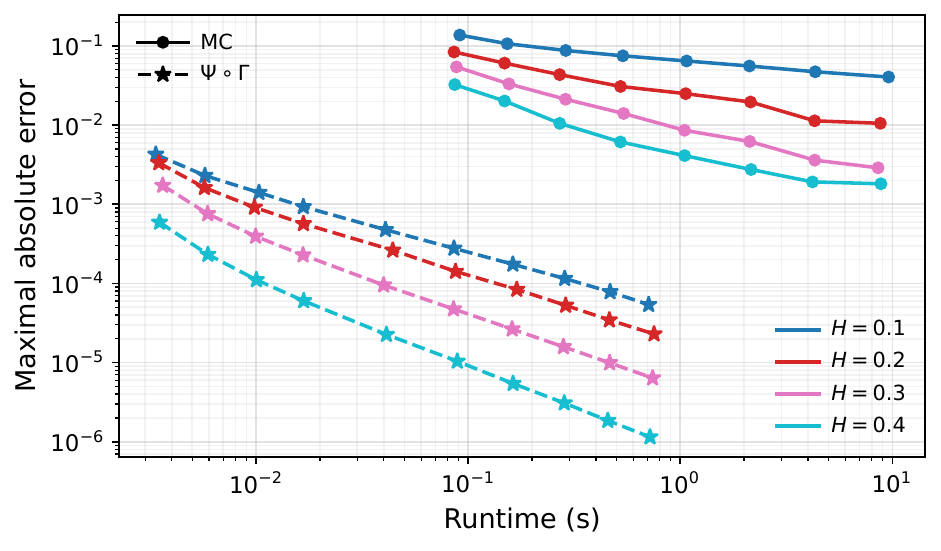}
    \caption{
    Runtime versus maximal absolute error for the computation of the truncated
    expected signature of the time-augmented fractional Brownian motion.
    Solid lines with circular markers correspond to the conventional Monte Carlo
    estimator based on piecewise-linear samples, while dashed lines with star
    markers correspond to the correlator-based $\Psi$ transform method using
    deterministic quadrature. Colors indicate the Hurst parameter $H$.
    }
    \label{fig:fbm_runtime_error}
\end{figure}

\subsection{Optimal tracking of fractional Brownian motion}\label{sec:control}
Let $(\Omega,\mathcal{F},\mathbb{F}=(\mathcal{F}_t)_{t\in[0,T]},\mathbb{P})$ be a complete filtered
probability space supporting an $\mathbb{F}$-adapted fractional Brownian
motion $X^H$. We study the stochastic control problem
\begin{equation}\label{eq:control_problem}
    V_0=\inf_{\alpha\in\mathcal{U}}
    \frac12
    \mathbb{E}\left[
        \int_0^T
        \left(
            (Y_t^\alpha)^2+\kappa\alpha_t^2
        \right)
        \dd t
    \right],
    \qquad
    Y_t^\alpha
    :=
    Y_0+\int_0^t\alpha_s\,\dd s-X_t^H,
\end{equation}
where $Y_0\in\mathbb{R}$, $\kappa>0$ is a penalization parameter, and
$\mathcal{U}$ denotes the set of real-valued, $\mathbb{F}$-progressively
measurable processes $\alpha=(\alpha_t)_{t\in[0,T]}$ satisfying $\mathbb{E}\left[
        \int_0^T|\alpha_t|^2\,\dd t
    \right]
    <\infty.$ The controlled process $Y^\alpha$ represents the \emph{tracking error} between the finite-variation control trajectory and the fractional Brownian motion $X^H$. Problems of this type have been studied in \cite{bank2017hedging} and, more recently, in the context of signature-based stochastic control in \cite{bank2024stochastic}. Loosely speaking, they describe the trade-off between closely tracking a frictionless target strategy (here $X^H$) and limiting trading costs caused by rapid adjustments of the actual position. In particular, the minimal tracking costs have an explicit solution, given in \cite[Theorem 5.2]{bank2024stochastic}, which we will use as a benchmark.

\begin{figure}[t]
    \centering
    \includegraphics[width=\textwidth]{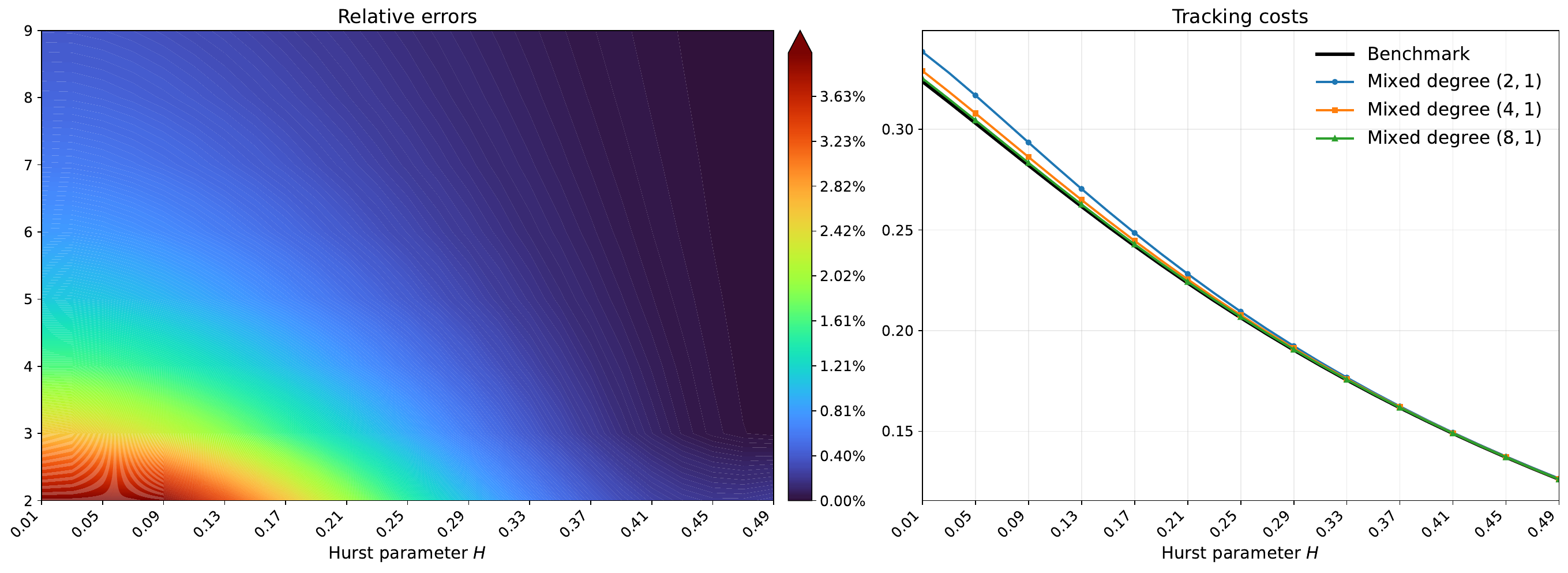}
    \caption{Left: Relative approximation errors of the mixed-degree controls $(N_A,1)$ with respect to the benchmark, across Hurst parameters $H$ and time degrees $N_A$. Right: Minimal tracking costs obtained from expected signatures in $\Psi$-coordinates for selected mixed degrees $(N_A,1)$, compared with the benchmark. Recall that the  corresponding mixed degree in the optimization problem is $(2N_A+3,2)$; see \eqref{eq:shear_tracking_problem}. Parameters: $Y_0=0, \kappa=0.1$ and $T=1$.}
    \label{fig:tracking_heatmap_costs}
\end{figure}Inspired by signature-methods in stochastic control initiated by \cite{kalsi2020optimal}, we consider here the class of mixed-degree signature functionals in $\Psi$-coordinates, that is  \[
\mathcal{U}_{N_A,N_X}^\Psi = \left \{\alpha_t = \langle \ell, \mathrm{Sig}(A,X)_{0,t}\rangle_\Psi: \ell \in  \mathbb{R}^{N_A,N_X}\langle \{\blue{\bullet},\orange{\bullet}\}\rangle \right  \} \subseteq \mathcal{U}, \qquad N_A,N_X \in \mathbb{N},
\] where $ \mathbb{R}^{N_A,N_X}\langle \{\blue{\bullet},\orange{\bullet}\} \rangle$  consists of linear combination of words $w \in \mathcal{W}_{\{\blue{\bullet},\orange{\bullet}\}}^{(n_A,n_X)}$ (see~\eqref{eq:mixed_degree_words}) with $n_A\leq N_A$ and $n_X \leq N_X$.

\begin{remark}
    By Proposition~\ref{prop:Psi},  together with the approximation result
    \cite[Theorem~4.7]{bank2024stochastic}, one may obtain convergence 
    \[
        V_0^{(N_A,N_X)}
        :=
        \inf_{\alpha\in\mathcal{U}_{(N_A,N_X)}^{\Psi}}
        \frac12
        \mathbb{E}\left[
            \int_0^T
            \left(
                (Y_t^\alpha)^2+\kappa\alpha_t^2
            \right)
            \dd t
        \right]
        \longrightarrow V_0,
        \qquad
        \min\{N_A,N_X\}\longrightarrow\infty.
    \]
    A similar conclusion can be obtained for multivariate processes $X$
    using partially symmetrized $\wh{\Psi}$-controls.
\end{remark}
Now consider a signature-control in $\Psi$-coordinates: $\alpha_t=\langle \ell, \mathrm{Sig}(A,X^H)_{0,t}\rangle_\Psi$,  by definition \begin{align*}
Y_t^\alpha & = \langle Y_0 \varnothing, \mathrm{Sig}(A,X)_{0,t}\rangle_\Psi+\int_0^t\langle \ell, \mathrm{Sig}(A,X^H)_{0,s}\rangle_\Psi \dd s - \langle \orange{\bullet},\mathrm{Sig}(A,X)_{0,t}\rangle_\Psi \\ & = \langle Y_0\varnothing+ \ell \blue{\bullet}-\orange{\bullet},\mathrm{Sig}(A,X)_{0,t}\rangle_\Psi.
\end{align*} Moreover, recalling the product $\bullet$ from \eqref{eq:gamma_dual_product_explicit}, we observe \[
(Y_t^\alpha)^2= \langle (Y_0\varnothing+ \ell \blue{\bullet}-\orange{\bullet})^{\bullet 2}, \mathrm{Sig}(A,X)_{0,t}\rangle_\Psi, \qquad (\alpha_t)^2=\langle \ell^{\bullet 2},\mathrm{Sig}(A,X)_{0,t}\rangle_\Psi, 
\] and thus, by Corollary~\ref{cor:rough_gamma} we can translate it into the deterministic optimization problem \begin{equation}\label{eq:shear_tracking_problem}
 V_0^{(N_A,N_X)} =\inf_{\ell \in T^{(N_A,N_X)}} \left \langle \frac12(Y_0\varnothing +\ell \blue{\bullet}-\orange{\bullet})^{\bullet 2}\blue{\bullet}+\frac12\kappa \ell^{\bullet 2}\blue{\bullet}, \bmu_{0,T}\right \rangle_\Psi.
\end{equation}
Note that, when the control is truncated at mixed degree $(N_A,N_X)$,
the words appearing in \eqref{eq:shear_tracking_problem} have mixed
degree at most $(2N_A+3,2N_X)$. When comparing the resulting
approximations with the exact benchmark value $V_0$, we observe that
already $N_X=1$ yields very accurate results, allowing for 
large values of $N_A$. This is illustrated for several Hurst parameters $H\in(0,1/2)$ in
Figure~\ref{fig:tracking_heatmap_costs}. The heat map on the left-hand
side shows the relative error
${(V_0^{(N_A,1)}-V_0)}/{V_0}$
where $N_X=1$ is fixed and $N_A\in\{1,\ldots,9\}$ is varied along the
$y$-axis. We emphasize that the largest truncation considered here,
namely $N_A=9$, requires $\Psi$-coordinates up to mixed
degree $(21,2)$. For comparison, the mixed truncation $(21,2)$ contains only $2299$
entries, whereas the full truncation $(21,21)$ would
contain more than $2.1\times 10^{12}$ entries, which would already be a challenging task in terms of memory. In Table~\ref{tab:tracking_costs} we summarize the minimal tracking costs $V_0$ and the signature approximations $V^{(N_A,1)}$ for the same mixed-degrees. The numerical result show fast convergence to the benchmark as the time
degree $N_A$ increases. The largest errors occur in the roughest regime,
that is, for $H$ close to zero, whereas the approximation becomes
increasingly accurate as $H$ approaches $1/2$. Interestingly, the rather low degree $N_X=2$ suffices for very accurate approximations.

\begin{table}[t]
\centering

\setlength{\tabcolsep}{3pt}
\renewcommand{\arraystretch}{1.12}

\resizebox{\textwidth}{!}{%
\begin{tabular}{lrrrrrrrrrrrrr}
\toprule

& \multicolumn{13}{c}{\textbf{Hurst parameter $H$}} \\
\cmidrule(lr){2-14}
& $0.01$ & $0.05$ & $0.09$ & $0.13$ & $0.17$
& $0.21$ & $0.25$ & $0.29$ & $0.33$ & $0.37$
& $0.41$ & $0.45$ & $0.49$ \\
\midrule
\textbf{Exact}
& 0.324 & 0.303 & 0.282 & 0.262 & 0.242
& 0.224 & 0.206 & 0.190 & 0.175 & 0.162
& 0.149 & 0.137 & 0.126 \\
\midrule
$(2,1)$
& 0.3385 & 0.3169 & 0.2935 & 0.2704 & 0.2486
& 0.2282 & 0.2095 & 0.1923 & 0.1767 & 0.1624
& 0.1493 & 0.1374 & 0.1264 \\
$(3,1)$
& 0.3322 & 0.3107 & 0.2885 & 0.2667 & 0.2459
& 0.2264 & 0.2083 & 0.1916 & 0.1762 & 0.1621
& 0.1491 & 0.1372 & 0.1262 \\
$(4,1)$
& 0.3291 & 0.3080 & 0.2863 & 0.2650 & 0.2446
& 0.2255 & 0.2077 & 0.1912 & 0.1760 & 0.1620
& 0.1490 & 0.1371 & 0.1262 \\
$(5,1)$
& 0.3274 & 0.3064 & 0.2849 & 0.2639 & 0.2438
& 0.2249 & 0.2073 & 0.1910 & 0.1758 & 0.1619
& 0.1490 & 0.1371 & 0.1262 \\
$(6,1)$
& 0.3264 & 0.3055 & 0.2842 & 0.2633 & 0.2434
& 0.2246 & 0.2071 & 0.1908 & 0.1757 & 0.1618
& 0.1490 & 0.1371 & 0.1262 \\
$(7,1)$
& 0.3258 & 0.3049 & 0.2837 & 0.2629 & 0.2431
& 0.2244 & 0.2069 & 0.1907 & 0.1757 & 0.1618
& 0.1490 & 0.1371 & 0.1262 \\
$(8,1)$
& 0.3254 & 0.3046 & 0.2834 & 0.2627 & 0.2429
& 0.2242 & 0.2068 & 0.1906 & 0.1756 & 0.1618
& 0.1490 & 0.1371 & 0.1262 \\
$(9,1)$
& 0.3251 & 0.3043 & 0.2832 & 0.2625 & 0.2428
& 0.2241 & 0.2068 & 0.1906 & 0.1756 & 0.1618
& 0.1490 & 0.1371 & 0.1262 \\
$(10,1)$
& 0.3249 & 0.3041 & 0.2830 & 0.2624 & 0.2427
& 0.2241 & 0.2068 & 0.1906 & 0.1756 & 0.1618
& 0.1490 & 0.1371 & 0.1262 \\
\bottomrule
\end{tabular}%
}\caption{Minimal tracking costs of fractional Brownian motion, benchmark and expected signature in $\Psi$-coordinates with different mixed-degrees. Parameters: $Y_0=0, \kappa=0.1$ and $T=1$.}
\label{tab:tracking_costs}
\end{table}

\bibliographystyle{plain}
\bibliography{biblio}

\appendix 

\section{Weak error of  naive expected signature approximation}\label{app:weak_error}
In this section we illustrate one of the main motivations for this work, namely the central issue of naive  numerical approximations of expected signatures for augmented stochastic processes $Y_t=(A_t,X_t)$. The typical procedure is to replace $Y$ with a piecewise linear approximation on some grid $\mathcal{P} = \{0=t_0<t_1\cdots <t_N=T \}$, that is $$Y^\mathcal{P}_t = Y_{t_k} (1-\theta_k(t))+Y_{t_{k+1}}\theta_k(t), \qquad t \in [t_k,t_{k+1}], \qquad \theta_k(t)= \frac{t-t_k}{t_{k+1}-t_k},$$ and then proceed with a Monte-Carlo approximation of $\bmu_t(\mathcal{P})= \mathbb{E}[\Sig{Y^\mathcal{P}}_{0,t}]$. The following lemma illustrates that this procedure is quite suboptimal in terms of convergence rates. More precisely, we consider the simple example $Y^H_t=(t,X^H_t)$, where $X^H$ is a one-dimensional fractional Brownian motion with Hurst parameter $H\in (0,1/2)$, and show in the following lemma that the weak error is $2H$. 

\begin{lemma}\label{lem:weak_rates}
    Suppose $\mathcal{P}$ is a uniform partition and set $\Delta t:=\frac{1}{N}$. We have \begin{equation}\label{eq:weak_rate_2H}
    \big \Vert 
        \pi_{\leq K}(\bmu_T^H)
        -
        \pi_{\leq K}(\bmu_T^H(\mathcal{P})) \big \Vert 
    \le C_{K,T,H} \Delta t^{2H}.\end{equation} Moreover, the rate is sharp, in the sense that we can find (infinitely many) words $w$ such that $$
        \langle w,\bmu_T\rangle
        -
        \langle w,\bmu_T^H(\mathcal{P})\rangle
    =
    o(\Delta t^{2H}).$$
\end{lemma}
\begin{proof}
    First note that for any $s \in [t_k,t_{k+1}]$ and $t \in [t_j,t_{j+1}]$, we have \begin{align}
        R^\mathcal{P}(s,t)=\mathbb{E}[X^{H,\mathcal{P}}_sX^{H,\mathcal{P}}_t] & =(1-\theta_k(s))(1-\theta_j(t)) R(t_k,t_j)+(1-\theta_k(s))\theta_j(t)R(t_k,t_{j+1})\\ & \qquad +(1-\theta_j(t))\theta_k(s)R(t_{k+1},t_{j})+\theta_k(t)\theta_j(s) R(t_{k+1},t_{j+1}),
    \end{align} which is simply the bilinear interpolation of $R(t,s)= \mathbb{E}[X_s^HX_t^H]$. Since $R(t,s) = \frac{1}{2}(t^{2H}+s^{2H}-|t-s|^{2H})$, it is thus not difficult to show that $\Vert R-R^\mathcal{P} \Vert_{\infty;[0,T]^2} \leq C_{H,T}(\Delta t)^{2H}$. Moreover, an application of Isserlis formula seen in Example~\ref{ex:Gaussian_isserlis} then also immediately shows $$\left | \mathbb{E}\left [\prod_{l=1}^nX^H_{r_l}\right ]-\mathbb{E}\left [\prod_{l=1}^nX^{H,\mathcal{P}}_{r_l}\right ] \right | \leq C_{K,H,T} (\Delta t)^{2H}, \qquad 1 \leq n \leq K.$$
    Thus, we can readily conclude \eqref{eq:weak_rate_2H} as a consequence of our expected signature transform Proposition~\ref{prop:exp_sig_prop}.

    For the second part, let us fix the word $w=221$, so that $\langle w,\bmu_T\rangle = \frac12\int_0^T\mathbb{E}[X_t^2]\dd t = \frac{T^{2H+1}}{2(2H+1)}$. On the other-hand, using the representation of $R^\mathcal{P}$ above for $s=t$, a direct computation shows \begin{align*}
    \langle w,\bmu_T(\mathcal{P})\rangle & = \sum_{[u,v]\in \mathcal{P}} R(u,u)\int_u^v (1-\theta(t))^2\dd t+ 2R(u,v)\int_u^v\theta(t)(1-\theta(t))\dd t+ R(v,v) \int_u^v\theta(t)^2\dd t  \\ & = \frac{\Delta t}{2} \sum_{[u,v] \in \mathcal{P}}(u^{2H}+v^{2H})- \frac{(\Delta t)^{2H+1}}{2}.
    \end{align*}
Thus we only have to show that the trapezoidal rule for $t\mapsto t^{2H}$ is of $o((\Delta t)^{2H})$, that is $$\frac{1}{(\Delta t)^{2H}}\left ( \int_0^Tt^{2H}\dd t- \frac{\Delta t}{2} \sum_{[u,v] \in \mathcal{P}}u^{2H}+v^{2H}\right )=:\varepsilon(\Delta t) \rightarrow 0, \qquad \Delta t \rightarrow 0.$$ Writing $\int_0^Tt^{2H}\dd t = \frac{(\Delta t)^{2H+1}}{2H+1}+ \sum_{[u,v] \in \mathcal{P},u \neq 0} \int_u^v t^{2H}\dd t$, we have $$\varepsilon(\Delta t) = (\Delta t)^{2H+1}\left (\frac{1}{2(2H+1)}-\frac12\right )+ \sum_{[u,v] \in \mathcal{P},u \neq 0} \left \{ \int_u^v t^{2H}\dd t-\frac{\Delta t}{2}(u^{2H}+v^{2H})\right \}.$$ Noting that $t \mapsto t^{2H}$ is $C^\infty$ on $[\Delta t,T]$, the claim follows from standard trapezoidal rules for smooth functions. The same conclusion holds for any word of the typ $w=221^{\otimes \ell}$ with $\ell \geq 1$, which concludes the proof.
\end{proof} 

\begin{figure}[t]
    \centering
    \includegraphics[width=0.75\textwidth]{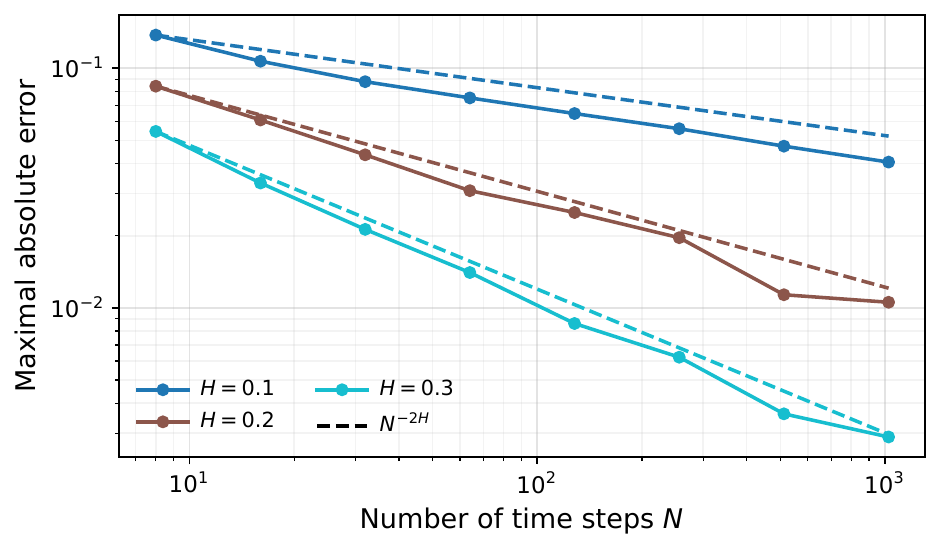}
    \caption{
    Weak error of the conventional Monte Carlo estimator for the
    expected signature of the time-augmented fractional Brownian motion. Solid
    lines correspond to the Monte Carlo errors,
    dashed lines show the reference rates $N^{-2H}$.
    }
    \label{fig:fbm_mc_timestep_error}
\end{figure}
We illustrate this numerically in Figure~\ref{fig:fbm_mc_timestep_error} for
$H\in\{0.1,0.2,0.3\}$. For increasing numbers of time steps $N$, corresponding
to the uniform mesh size $\Delta t=1/N$, we plot the left-hand side of
\eqref{eq:weak_rate_2H} with truncation level $K=4$. The norm is chosen as $\Vert (\mathbf a^{0},\dots,\mathbf a^{K})\Vert
    :=
    \max_{i=0,\dots,K}
    \Vert \mathbf a^{(i)}\Vert_{(\mathbb R^2)^{\otimes i}}$
The dashed reference lines show the rates $N^{-2H}$, normalized to the first
error value for each Hurst parameter. The observed errors are consistent with
the predicted weak rate $(\Delta t)^{2H}=N^{-2H}$.

\end{document}